\documentclass[11pt,letterpaper]{article}
\usepackage[T1]{fontenc}
\usepackage[utf8]{inputenc}
\usepackage[margin=1in]{geometry}
\usepackage{amsmath,amssymb,amsthm,mathtools}
\usepackage{newtxtext,newtxmath}
\usepackage{microtype,booktabs,tabularx,array,longtable}
\usepackage[numbers,sort&compress]{natbib}
\usepackage{enumitem,float,needspace}
\floatstyle{plain}
\newfloat{algorithm}{tbp}{loa}
\floatname{algorithm}{Algorithm}
\usepackage{xurl}
\usepackage[hidelinks]{hyperref}
\usepackage{fancyhdr}
\setlist{itemsep=2pt,topsep=4pt}
\allowdisplaybreaks[2]
\numberwithin{equation}{section}
\newtheorem{theorem}{Theorem}[section]
\newtheorem{lemma}[theorem]{Lemma}
\newtheorem{proposition}[theorem]{Proposition}
\newtheorem{corollary}[theorem]{Corollary}
\theoremstyle{definition}
\newtheorem{definition}[theorem]{Definition}
\newtheorem{example}[theorem]{Example}
\theoremstyle{remark}

\newcommand{\R}{\mathbb R}
\newcommand{\Q}{\mathbb Q}
\newcommand{\E}{\mathbb E}
\newcommand{\K}{\mathcal K}
\newcommand{\I}{\mathcal I}
\newcommand{\N}{\mathcal N}
\newcommand{\CC}{\mathfrak C}
\newcommand{\CQ}{\mathcal Q}
\newcommand{\CL}{\mathcal C}
\newcommand{\CR}{\mathcal R}
\newcommand{\CA}{\mathcal A}
\newcommand{\ind}{\mathbf 1}
\newcommand{\Prb}{\mathbb P}
\newcommand{\eps}{\varepsilon}
\DeclareMathOperator{\out}{out}
\DeclareMathOperator{\inw}{in}
\DeclareMathOperator{\Inf}{Inf}
\newcommand{\pos}[1]{[#1]_+}
\DeclareMathOperator{\conv}{conv}
\DeclareMathOperator{\cone}{cone}
\DeclareMathOperator{\supp}{supp}
\DeclareMathOperator{\argmax}{arg\,max}
\DeclareMathOperator{\OPT}{OPT}
\DeclareMathOperator{\DCG}{DCG}
\DeclareMathOperator{\CG}{CG}
\DeclareMathOperator{\cut}{cut}
\DeclareMathOperator{\cov}{cov}
\DeclareMathOperator{\Bin}{Bin}
\DeclareMathOperator{\vol}{vol}
\newcolumntype{Y}{>{\raggedright\arraybackslash}X}
\hypersetup{pdftitle={Concave Roofs and Downward Recovery for DR-Submodular Polynomial Maximization},pdfauthor={Donglei Du},pdfsubject={Optimization; submodularity; approximation algorithms; explicit-input hardness}}
\title{\bfseries Concave Roofs and Downward Recovery for DR-Submodular Polynomial Maximization}
\author{Donglei Du\\[3pt]\small Faculty of Management, University of New Brunswick\\\small Fredericton, NB E3B 5A3, Canada\\\small \href{mailto:ddu@unb.ca}{ddu@unb.ca}\quad ORCID: 0000-0003-0111-8572}
\date{}
\newcommand{\IndependentResultsNote}{%
\paragraph{Independently obtained results.}
The quadratic $1/2$ and cubic $8/17$ guarantees were obtained independently in this work and by \citet{Aggarwal2026}. Both approaches combine nonnegative decompositions, concave majorants, and common recovery maps for low-degree objectives. Aggarwal also develops geometry-dependent refinements and adversarial-online guarantees. Here we construct the orientation by linear programming and identify its optimal value through threshold sets in the dual. We retain repeated powers through an exact complement identity on the original cube and obtain the baseline cubic guarantee with two explicit maps.
}
\begin{document}
\maketitle
\begin{abstract}
We give approximation algorithms for $\max\{F(x):x\in\K\}$, where $\K$ is a nonempty compact down-closed convex subset of the unit cube and $F$ is a nonnegative diminishing-returns (DR) submodular polynomial: every entry of its Hessian is nonpositive on the cube.

Starting from quadratic or cubic coefficients, we decompose the objective exactly into nonnegative rooted atoms and affine terms. Concave upper bounds for these components give a tractable roof. We then decrease the coordinates of a roof optimizer using maps that work for all components at once, recovering a fraction of the roof value while preserving feasibility. For quadratics, we also retain the concave diagonal residual exactly.

Under the stated assumptions on optimization access, these constructions give deterministic polynomial-time factors of $1/2$ for quadratics and $8/17$ for cubics, up to any prescribed additive error. The cubic result includes positive cubic coefficients and repeated variables. The factor $8/17$ also holds for two larger classes: supplied compatible rooted certificates of literal degree at most twelve, and sums of nonnegative submodular local functions of arity at most four. At arbitrary finite literal degree, a common scalar map attains the optimal single-map factor $2\sqrt3-3$. Global monotonicity and concave-cardinality structure give stronger guarantees for multilinear local-function sums, with no recovery beyond the identity map. Standard matroid rounding transfers the multilinear guarantees without loss. We also prove online preservation for directed-hypergraph cuts and, separately, a Unique-Games inapproximability bound of approximately $0.82843$ for explicit quadratic input under a cardinality upper bound.
\end{abstract}


\tableofcontents
\clearpage
\section{Introduction}\label{sec:intro}

We consider
\begin{equation}\label{eq:continuous_problem}
 \OPT(F,\K)=\max\{F(x):x\in\K\},
\end{equation}
where $\K\subseteq[0,1]^n$ is nonempty, compact, convex, and \emph{down-closed}: $x\in\K$ and $0\le y\le x$ imply $y\in\K$. A polynomial is \emph{Boolean-nonnegative} when it is nonnegative at every Boolean vertex. It is \emph{diminishing-returns (DR) submodular} when all diagonal and mixed second partial derivatives are nonpositive on the cube. In particular, it is concave in each coordinate separately, so Boolean nonnegativity implies nonnegativity throughout the cube.

We maximize a tractable concave upper bound and recover a feasible point whose objective value is a fixed fraction of that bound. When components share variables, recovering them separately may require conflicting choices. We resolve this conflict with maps that decrease coordinates and work for all components at once. 

\subsection{Two coefficient-input algorithms}
For Boolean-nonnegative quadratic DR polynomials, we decompose the multilinear part into nonnegative components and keep the concave diagonal residual intact. Clipping then recovers half the roof value while preserving feasibility, giving a deterministic $1/2$ approximation.

Our second result is a deterministic $8/17$ approximation for every Boolean-nonnegative cubic DR polynomial. Two difficulties arise beyond the quadratic case. DR-submodularity permits positive cubic coefficients, and replacing every positive coordinate exponent by one (Boolean multilinearization) can leave a residual that is not concave. We address both with an exact complement identity and a linear program for the decomposition. Two common maps then balance the competing recovery requirements of the components. Both algorithms start from rational coefficients and incur only a prescribed additive error under the access assumptions in Section~\ref{sec:access}. To handle repetitions, Section~\ref{sec:cubic_maps} proves the scalar inequalities before identifying coordinates.

\begin{table}[htbp]
\centering\small
\caption{The two coefficient-input guarantees. Each algorithm constructs its representation and returns $y\in\K$ with $F(y)\ge\rho\OPT(F,\K)-\varepsilon$. Powers and repeated coordinates are retained.}\label{tab:main}
\begin{tabularx}{\textwidth}{@{}p{.20\textwidth}p{.11\textwidth}Y Y@{}}\toprule
Objective & Factor $\rho$ & Construction & Recovery \\\midrule
Quadratic DR polynomial & $1/2$ & Pair orientation and exact concave diagonal residual & One clipping map; Theorem~\ref{thm:quadratic}\\[4pt]
Cubic DR polynomial & $8/17$ & Complement identity and uniform-polarity orientation & Two scalar maps; Theorem~\ref{thm:cubic}\\\bottomrule
\end{tabularx}
\end{table}

Boolean values alone do not determine the continuous objective. For example, $t(1-t)$ vanishes at both Boolean points but has a positive interior optimum. A multilinear extension therefore cannot replace an arbitrary continuous polynomial. Section~\ref{sec:model_comparison} makes this distinction precise. Section~\ref{sec:worked_cubic} follows one constrained continuous cubic from its coefficients to a recovered point and compares the continuous optimum, the Boolean optimum, and the roof bound.

\subsection{The mechanism and its extensions}
A \emph{concave roof} is a concave majorant of the objective. A \emph{literal} is a coordinate or its complement. A \emph{rooted atom} subtracts a literal product from one distinguished factor, its root; the other occurrences are its heads. A \emph{rooted certificate} is an exact nonnegative sum of atoms, literals, and a constant. Orientation allocates each subtracted product's coefficient among its permitted root choices. Literal degree counts occurrences, including repetitions. The low-degree construction uses uniform-polarity products, consisting only of coordinates or only of complements. The more general compatibility rule requires a complemented root whenever the product contains a complemented literal.

A \emph{downward map} satisfies $0\le R(x)\le x$. A \emph{portfolio} is a finite collection of downward maps with nonnegative weights summing to one. A common-coordinatewise map uses the same scalar function on every coordinate. The recovery inequalities add because every component uses the same maps and weights, including components with overlapping supports. Termwise recovery requires a separate inequality for each generator. The algorithm evaluates all branches and keeps the one with the largest original objective value. Theorem~\ref{thm:common} gives the guarantee; Algorithm~\ref{alg:recovery} specifies how to allocate numerical error.

The extensions in Table~\ref{tab:extensions} go in two directions. Larger rooted certificates allow more literal occurrences, while richer local functions allow any nonnegative submodular function on a small support. A local function's \emph{arity} counts distinct coordinates. Its concave closure maximizes expected value over joint distributions with prescribed coordinate marginals. Recentering subtracts an affine function from each local component to make its remaining marginals nonnegative.

\begin{table}[htbp]
\centering\small
\caption{Extensions and structural refinements. Weights are nonnegative and representations are supplied unless a construction is stated. The local-function rows concern multilinear extensions.}\label{tab:extensions}
\begin{tabularx}{\textwidth}{@{}>{\hsize=1.25\hsize\linewidth=\hsize}Y >{\hsize=.7\hsize\linewidth=\hsize}Y >{\hsize=1.05\hsize\linewidth=\hsize}Y@{}}\toprule
Additional input structure & Factor & Additional argument \\\midrule
Compatible rooted certificate of literal degree at most twelve & $8/17$ & Higher-rank atom inequalities; Section~\ref{sec:larger_atoms}\\[3pt]
Compatible rooted certificate of arbitrary finite literal degree & $2\sqrt3-3$ & An optimal common scalar map; Theorem~\ref{thm:single}\\[3pt]
Nonnegative submodular local functions of arity at most four & $8/17$ & A different pair on the full local cone; Section~\ref{sec:local}\\[3pt]
Globally monotone sum of nonnegative submodular local functions of arity at most $d\ge1$ & $1-(1-1/d)^d$ & Recentering; local monotonicity is unnecessary\\[3pt]
Nonnegative concave-cardinality local functions of arities $d_b\ge2$ & $\min_b[1-(1-1/d_b)^{d_b-1}]$ & Exact cardinality closure and binomial comparison\\\bottomrule
\end{tabularx}
\end{table}

The degree-twelve atom inequalities are formalized in Lean. The arity-four inequalities are certified by exact arithmetic on a finite generating family. The exhibited map pairs differ, but the scalar map with factor $2\sqrt3-3$ works for both families and their nonnegative sums. We construct certificates for higher-degree negatively signed polynomials. The remaining higher-degree guarantees assume that a compatible certificate is supplied.

\paragraph{Boundaries and stronger roofs.}
The quadratic and cubic examples also reveal a limit of downward recovery: a roof optimizer may set to zero a coordinate needed by a better feasible solution. Section~\ref{sec:frontiers} proves sharpness for specified fixed roofs, representation families, and additive local roofs of arity at most four. For multilinear objectives whose representations may be optimized, fractional triple packing captures the remaining freedom in the coefficients; triangle-free pair capacities eliminate that freedom. At literal degree thirteen, a complemented atom conflicts with the two-point equality structure of the cubic witness. This yields a strict supremum gap for finite common-coordinatewise termwise portfolios. Table~\ref{tab:sharpness_scope} records the roof, benchmark, and recovery class to which each claim applies.

Grouping local functions forces their joint distributions to agree within each group and can tighten the roof. Section~\ref{sec:design} develops these formulations, finite games for choosing recovery maps, and recovery adapted to a supplied instance.

\paragraph{Discrete, online, and computational consequences.}
For a nonnegative submodular set function $f$ and a down-closed feasible family $\I$, the discrete problem is
\begin{equation}\label{eq:discrete_problem}
 \max\{f(S):S\in\I\}.
\end{equation}
Its multilinear extension belongs to the continuous model. Standard matroid rounding converts the applicable fractional guarantees without loss, although this conversion can fail for general continuous polynomials. The online theorem uses a more specific representation by directed-hypergraph cuts and positive association. To obtain a prophet guarantee, we also need the outcome-space benchmark in Appendix~\ref{app:outcomes}.

Sections~\ref{hard:sec:methodology}--\ref{hard:sec:instantiations} address computational hardness for explicit rational quadratic input. Padding transfers the established $12/13$ NP-hardness bound to cardinality density $1/2$. A heterogeneous directed test and an affine Gaussian certificate give a Unique-Games inapproximability bound of $0.828426578\ldots$ at asymptotic density $0.0933003645$. These bounds apply to deterministic polynomial-time algorithms and leave a gap above our quadratic factor $1/2$.

\subsection{Relation to prior work}
\IndependentResultsNote

Continuous DR-submodularity combines lattice submodularity with coordinatewise concavity \citep{Topkis1978,SomaKakimuraInabaKawarabayashi2014,BianMirzasoleimanBuhmannKrause2017,BianLevyKrauseBuhmann2017}. Multilinear relaxation, continuous greedy, and measured continuous greedy provide general algorithms for submodular maximization \citep{CalinescuChekuriPalVondrak2011,FeldmanNaorSchwartz2011,BuchbinderFeldman2019}. The $0.401$ guarantee of \citet{BuchbinderFeldman2024} applies to nonnegative smooth DR-submodular objectives under their assumptions on smoothness and optimization access. Our larger ratios rely on the additional structure of the objective classes considered here.

\paragraph{Oracle barriers and the role of structure.}
Oveis Gharan and Vondr\'ak proved value-oracle barriers of $0.478$ under a matroid independence constraint and $0.491$ under a cardinality constraint \citep{gharan2011submodular}. Qi strengthened the cardinality barrier to $0.478$ \citep[Theorem~4.1]{qi2024maximizing}. These are exponential value-query lower bounds for unrestricted nonnegative submodular functions. They do not directly impose the same bounds on explicitly represented fixed-degree polynomials or objectives supplied with the decompositions used here. Under a fixed-degree promise, a multilinear polynomial can be reconstructed with polynomially many value queries; the bit-length qualification and the limitation to multilinear objectives are stated in Section~\ref{sec:access}. In particular, the general $0.478$ barrier does not establish near-optimality of $8/17$ for the intrinsic cubic class.

Three kinds of limitation must be kept separate: the information available from a value oracle, the factor recoverable from a given roof under a given rule, and polynomial-time hardness for explicitly encoded instances. Oracle bounds motivate structural assumptions. Recovery bounds describe the limits of a method. Computational hardness constrains algorithms whether or not they use a roof. We also distinguish matroid-base constraints, which are not down-closed, from the matroid-independence constraints used here.

Negaforms, roof duality, and concave extensions have a long history in pseudo-Boolean optimization \citep{HammerRudeanu1968,HammerRudeanu1969,HammerHansenSimeone1984,Lu1987,BorosHammer1993,Crama1993,BorosHammer2002,CramaHammer2011}. The cubic unary-negaform representation, in which each nonlinear product uses only variables or only complements, is due to \citet{BillionnetMinoux1985}; see also \citet{BorosGruber2012}. \citet{XuLiberti2026} study relaxations based on signed certificates. Our construction retains the exact polynomial identity on the continuous cube. An orientation linear program produces nonnegative rooted coefficients, and common downward maps recover value from the resulting roof.

\paragraph{Quadratic optimization.}
The pairwise linear roof used in the quadratic construction belongs to the classical bilinear-relaxation framework. Section~\ref{sec:quadratic_relaxations} compares it with semidefinite relaxations and explains which additional feasibility and objective assumptions those methods use.

For box constraints, \citet{niazadeh2020optimal} obtain a randomized $1/2$ approximation, up to additive error, for general nonnegative Lipschitz continuous lattice-submodular objectives. Section~\ref{sec:quadratic_box} gives a direct quadratic specialization: replace positive squares by endpoint chords, apply the DR quadratic algorithm, and choose endpoints in those coordinates without further loss.

Concave closures and correlation gaps compare correlated and independent distributions \citep{AgrawalDingSaberiYe2012,Yan2011,ChekuriVondrakZenklusen2014,HusicKohLohoVegh2025}. For nonmonotone functions, the max-subset comparison of \citet{RubinsteinSingla2017} permits deletions, and the downward comparison of \citet{ChekuriLivanos2024} permits decreases of the marginals. Local closures introduce a further difficulty: optimal joint distributions for different components may disagree on shared coordinates. Grouping enforces some of the missing consistency; a common map ensures that the recovery inequalities can be added.

For monotone objectives, the constants $1-(1-1/d)^d$ follow from the finite-ground-set correlation inequality of \citet[Lemma~4.10]{ChekuriVondrakZenklusen2014}. Coverage relaxations and pipage rounding give the coverage specializations \citep{AgeevSviridenko2004}. Our recentering argument requires monotonicity only of the sum and preserves the local ground sets. Concave-coverage refinements and curvature bounds use other structural assumptions \citep{BarmanFawziFerme2021,SviridenkoVondrakWard2017}. Monotonic decompositions of submodular functions are studied by \citet{BercziEtAl2026}.

Offline contention resolution \citep{ChekuriVondrakZenklusen2014}, greedy online contention resolution \citep{FeldmanSvenssonZenklusen2021}, and submodular dominance \citep{QiuSingla2022} provide different ways to round fractional solutions. Our online preservation proof applies Harris's inequality to directed-hyperedge cuts \citep{Harris1960}. The offline hypergraph application uses the more general splitting functions of \citet{VeldtBensonKleinberg2022}; the online theorem requires the specific cut representation.

\paragraph{Explicit-input hardness.}
For unconstrained Max-DiCut, Trevisan, Sorkin, Sudan, and Williamson proved
NP-hardness within $12/13+\eps$ \citep{TrevisanSorkinSudanWilliamson2000}.
Brakensiek et al.\ give an unconditional $0.87448$-approximation and,
assuming Khot's Unique Games Conjecture (UGC) \citep{Khot2002},
$0.87461$-hardness for unconstrained Max-DiCut
\citep{BrakensiekHuangPotechinZwick2026}.  Padding transfers either hardness
bound to the at-most model at density $1/2$; these unconstrained bounds do not
determine the fixed-density approximability studied here.

Approximation frameworks for Boolean constraint-satisfaction problems
with exact global-cardinality constraints were developed through
semidefinite-programming hierarchies and reductions based on
coordinate-selection tests by Raghavendra and Tan
\citep{RaghavendraTan2012}. Their algorithmic methodology treats exact cardinality, whereas our recovery setting uses down-closed constraints. Austrin and Stankovi\'c showed that prescribed
cardinality can strengthen UGC-inapproximability, obtaining the approximately
$0.858$ boundary for cardinality-constrained Max-Cut
\citep{AustrinStankovic2019}.  For directed cuts with a prescribed part size,
Ageev, Hassin, and Sviridenko established the classical
$1/2$-approximation benchmark \citep{AgeevHassinSviridenko2001}; enumerating
all prescribed sizes $p\le k$ transfers that guarantee to an upper bound
$|S|\le k$.  Ghuge et al.\ also give a direct $1/2$-approximation for
weighted Max-DiCut with nonnegative arc weights under an at-most-cardinality constraint
\citep{GhugeKwonNagarajanSharma2022}.  The present hardness test directly
treats this at-most model and uses heterogeneous biases to preserve directed
asymmetry.

The analytic soundness step uses Gaussian noise-correlation bounds
\citep{Mossel2010}.  The finite certificate uses Plackett's reduction formula
for multivariate normal probabilities \citep{Plackett1954} and the
correctly rounded multiple-precision floating-point library MPFR \citep{FousseEtAl2007}.  We also evaluate the
generic UGC-to-NP ``losing-half'' transfer of Bhangale and Khot
\citep{BhangaleKhot2022}; for the certified test, the available soundness
certificate does not preserve a gap.

\subsection{Organization and proof routes}
Part~\ref{part:recovery} develops the quadratic and cubic algorithms. Section~\ref{sec:models} defines the continuous and multilinear models and states the common-recovery guarantee. Section~\ref{sec:quadratic_main} gives the quadratic construction, clipping proof, applications, and first obstruction. Section~\ref{sec:cubic_main} handles the additional cubic difficulties, proves the two-map guarantee, and works through a continuous example. Appendix~\ref{app:cubic} supplies the remaining scalar analysis.

Part~\ref{part:extensions} treats larger certificates, richer local functions, and stronger structural assumptions in Sections~\ref{sec:larger_atoms}--\ref{sec:structured}. It then establishes recovery limits, develops stronger roofs and recovery rules, and gives the discrete and correlation consequences. Part~\ref{part:online} proves online preservation; Appendix~\ref{app:outcomes} gives the outcome-space argument. Part~\ref{part:hardness} presents the separate computational-hardness results and the remaining questions. Appendices~\ref{app:twelve}, \ref{app:arity_four}, and~\ref{hard:app:certificate} document the three proof-software components.

\Needspace{14\baselineskip}
\part{Quadratic and cubic coefficient-input algorithms}\label{part:recovery}
\section{Continuous and multilinear models; common recovery}\label{sec:models}
We first distinguish a continuous polynomial from the multilinear interpolation of its Boolean values. A two-edge example then shows why overlapping components need a common recovery rule.

\subsection{Continuous and Boolean objects}
Let $N=[n]=\{1,\ldots,n\}$ be a finite ground set. Inequalities between vectors are coordinatewise. Write $[0,x]=\{y:0\le y\le x\}$ for the \emph{lower box}, $x_A=(x_i)_{i\in A}$, and $\ind_S$ for the incidence vector of $S\subseteq N$. The symbol $\ind\{E\}$ is one when statement $E$ holds and zero otherwise; $e_i$ is the $i$th coordinate unit vector. For a nonnegative integer multiindex $\beta$, set $|\beta|=\sum_i\beta_i$, $\beta!=\prod_i\beta_i!$, $\supp\beta=\{i:\beta_i>0\}$, and $x^\beta=\prod_i x_i^{\beta_i}$. For a set $T$, $x^T=\prod_{i\in T}x_i$ and $(1-x)^T=\prod_{i\in T}(1-x_i)$; for a multiindex these products include multiplicities. Also set $[a]_+=\max\{a,0\}$.

A set function $f:2^N\to\R$ is \emph{submodular} if
\[
 f(S)+f(T)\ge f(S\cup T)+f(S\cap T)\qquad(S,T\subseteq N).
\]
It is \emph{monotone} if $f(S)\le f(T)$ whenever $S\subseteq T$. Its marginal at $i\notin S$ is $f(S\cup\{i\})-f(S)$. A \emph{modular} function has the form $c+\sum_{i\in S}a_i$. A function on the Boolean cube is also called a \emph{pseudo-Boolean function}.

A continuous function on a box is \emph{lattice submodular} if
\[
 F(x)+F(y)\ge F(x\vee y)+F(x\wedge y),
\]
where $\vee$ and $\wedge$ take coordinatewise maxima and minima. For a twice continuously differentiable function defined on a neighborhood of a full-dimensional box, lattice submodularity is equivalent to $\partial_{ij}^2F(x)\le0$ for every $x$ in the box and $i\ne j$. The function is \emph{diminishing-returns (DR) submodular} on the cube when
\begin{equation}\label{eq:DR}
 \partial_{ij}^2F(x)\le0\qquad(i,j\in N,\ x\in[0,1]^N).
\end{equation}
The diagonal conditions are included. A polynomial is \emph{Boolean-nonnegative} if $F(v)\ge0$ at every $v\in\{0,1\}^N$.

\begin{lemma}[Separate concavity]\label{lem:separate}
If $F$ is concave in each coordinate separately, then
\[
 F(x)\ge\sum_{S\subseteq N}F(\ind_S)\prod_{i\in S}x_i\prod_{i\notin S}(1-x_i).
\]
Consequently, a Boolean-nonnegative DR-submodular polynomial is nonnegative throughout the cube.
\end{lemma}
\begin{proof}
For one coordinate, with the others fixed, concavity gives
$F(x)\ge x_iF(x_{-i},1)+(1-x_i)F(x_{-i},0)$. Apply this inequality successively to all coordinates.
\end{proof}

For a set function $f:2^N\to\R$, its \emph{multilinear extension} is
\begin{equation}\label{eq:multilinear}
 F(x)=\E[f(R_x)]=\sum_{S\subseteq N}f(S)\prod_{i\in S}x_i\prod_{i\notin S}(1-x_i),
\end{equation}
where $R_x$ includes each element $i$ independently with probability $x_i$. These selection probabilities are the \emph{marginals}; a joint distribution with specified marginals is a \emph{coupling}. The function $F$ is the unique multilinear polynomial agreeing with $f$ on the Boolean cube. If $f$ is submodular, its multilinear extension satisfies \eqref{eq:DR}, and its diagonal second derivatives vanish. When we discuss a general continuous polynomial, $F$ denotes that polynomial itself, rather than the extension of its Boolean restriction.

For $g:2^A\to\R$, write $G$ for its multilinear extension and define its \emph{concave closure} by
\begin{equation}\label{eq:closure}
 g^+(x)=\max\left\{\sum_{T\subseteq A}\lambda_Tg(T):
  \lambda\ge0,\ \sum_T\lambda_T=1,\ \sum_{T\ni i}\lambda_T=x_i\ (i\in A)\right\}.
\end{equation}
It is the smallest concave function majorizing the Boolean values. The product distribution is feasible, so $G\le g^+$. 
A \emph{concave roof} of a continuous function $F$ is a concave majorant $L$, meaning $L(x)\ge F(x)$ throughout the cube. It is \emph{Boolean-exact} when $L(\ind_S)=F(\ind_S)$ for every $S$. An \emph{affine majorant} is an affine function with the stated domination property. The \emph{hypograph} of $L$ is $\{(x,t):t\le L(x)\}$. A lifted formulation describes this set using additional variables.

For a multilinear polynomial $F(x)=\sum_{T\subseteq N}a_Tx^T$, the M\"obius coefficients are
\[
 a_T=\sum_{U\subseteq T}(-1)^{|T|-|U|}f(U).
\]
Set $\deg(f)=\max\{|T|:a_T\ne0\}$ for nonzero $f$, and $\deg(0)=-\infty$.

\subsection{Restriction, interpolation, and relaxation}\label{sec:model_comparison}
Let $F$ be a continuous polynomial, let $f(S)=F(\ind_S)$, and write $M$ for the multilinear extension of $f$. Passing from $F$ to $M$ is \emph{Boolean multilinearization}: replace every positive coordinate exponent by one and collect coefficients. Three operations must be distinguished: evaluating $F$ at Boolean points, interpolating those values, and relaxing an integrality constraint. An \emph{exact polynomial identity} holds throughout the cube, before imposing $x_i^2=x_i$. Exactness of a representation does not imply that its objective can be optimized exactly.

For a DR polynomial, Lemma~\ref{lem:separate} gives $M\le F$, with equality throughout the cube precisely when $F$ is multilinear. We also have $M\le f^+$, since independent sampling is an admissible coupling. The functions $F$ and $f^+$ have no general ordering: the next example has $F>f^+$ in the interior, whereas the directed edge below has $F<f^+$ at the midpoint.

\begin{example}[A continuous roof is not necessarily a Boolean concave closure]\label{ex:self_atom}
Let $F(t)=t(1-t)$. Its Boolean restriction and the multilinear extension of that restriction are identically zero. Its standard rooted roof is $L(t)=\min\{t,1-t\}$, while $F(1/2)=1/4$. The \emph{continuous concave envelope}, the smallest concave majorant on the interval, is $F$ itself because $F$ is concave. Thus the Boolean concave closure, the continuous concave envelope, and the standard roof are three different functions. In particular, a Boolean closure need not majorize a continuous objective.
\end{example}

For a set function with multilinear extension $M$,
\[
 \max_{S\in\I}f(S)\le
 \max_{x\in\conv\{\ind_S:S\in\I\}}M(x).
\]
The lossless rounding theorem for matroid independence makes this an equality under its stated hypotheses. For a swap direction, the relevant second derivative is
\[
 \frac{d^2}{dt^2}F\bigl(x+t(e_i-e_j)\bigr)
 =\partial_{ii}F+\partial_{jj}F-2\partial_{ij}F,
\]
with all derivatives evaluated at the displayed point. For a multilinear submodular objective, the diagonal terms vanish and the expression is nonnegative. For a general continuous DR polynomial, those terms can be negative. Section~\ref{sec:rounding} states the applicable rounding guarantee; Example~\ref{ex:self_atom} shows why lossless integral conversion cannot hold for arbitrary continuous polynomials.

\subsection{One edge and a shared-coordinate conflict}\label{sec:edge_motivation}
For a directed edge with distinct endpoints,
\[
 h(x_1,x_2)=x_1(1-x_2),\qquad
 \ell(x_1,x_2)=\min\{x_1,1-x_2\}.
\]
The product is bounded by each factor, so their minimum is a concave majorant. Here $\ell$ is also the concave closure of the Boolean edge function: a coupling can make the intersection of the two relevant events attain the smaller marginal. At $(1/2,1/2)$, independent sampling gives $1/4$, while the coupling assigning probability $1/2$ to each of $(1,0)$ and $(0,1)$ gives closure value $1/2$.

Now consider the edges $1\to2$ and $2\to3$. Maximizing their contributions separately below $(1/2,1/2,1/2)$ requires $y_2=0$ and $y_2=1/2$, respectively. No single point meets both requirements. A common recovery rule makes both component inequalities refer to the same vector, allowing us to add them.

\subsection{The upper-bound interface}\label{sec:common}
A map $R:[0,1]^N\to[0,1]^N$ is \emph{downward} if $0\le R(x)\le x$. A \emph{portfolio} is a finite family of downward maps with nonnegative weights summing to one; each map is a branch. A \emph{common-coordinatewise} map applies the same scalar function $\phi$ to every coordinate: $R(x)_i=\phi(x_i)$. A common map uses one global output in every component, including components that overlap. Downwardness preserves every down-closed set containing $x$. It is also necessary for this universal feasibility property, as the set $[0,x]$ shows. A map tailored to one feasible region may satisfy weaker conditions.

\begin{theorem}[Upper bound and recovery]\label{thm:common}
Let $\K\subseteq[0,1]^N$ be nonempty, compact, and down-closed, let $F$ be continuous, and let $H$ be real-valued and upper semicontinuous on $\K$. Suppose
\[
 \max_{\K}H\ge\max_{\K}F.
\]
Let $k\ge1$, let $R_1,\ldots,R_k$ be downward maps, and let $p_s\ge0$ sum to one, and let $\rho\ge0$. If
\begin{equation}\label{eq:common_inequality}
 \sum_s p_sF(R_s(x))\ge\rho H(x)\qquad(x\in\K),
\end{equation}
then, for every $\delta\ge0$, a feasible $\bar x$ satisfying $H(\bar x)\ge\max_\K H-\delta$ gives
\begin{equation}\label{eq:common_guarantee}
 \max_sF(R_s(\bar x))\ge\rho\max_{x\in\K}F(x)-\rho\delta.
\end{equation}
Every recovered point is feasible.
\end{theorem}
\begin{proof}
Down-closedness gives feasibility. The largest branch value dominates the weighted average, which is at least $\rho H(\bar x)\ge\rho(\max_\K H-\delta)\ge\rho(\max_\K F-\delta)$.
\end{proof}
A concave majorant provides the comparison of maxima required by Theorem~\ref{thm:common}. The recovery argument needs only that comparison, not concavity itself.

\begin{corollary}[Nonnegative aggregation]\label{cor:aggregation}
Suppose the sums $F=\sum_bw_bF_b$ and $L=\sum_bw_bL_b$ are finite, $w_b\ge0$, and $\rho\ge0$. An active component is one with $w_b>0$. If one fixed probability distribution over downward maps satisfies
\begin{equation}\label{eq:component_payment}
 \sum_sp_sF_b(R_s(x))\ge\rho L_b(x)\qquad\text{for every active component }b,
\end{equation}
then $\sum_sp_sF(R_s(x))\ge\rho L(x)$. If $L\ge F$ and the continuity, compactness, and down-closedness hypotheses of Theorem~\ref{thm:common} hold, that theorem gives its approximation guarantee to \eqref{eq:continuous_problem}.
\end{corollary}
\begin{proof}
Multiply the component inequalities by $w_b$ and sum. Nonnegative weights preserve their direction, and the same global map in each summand makes the left sides add to $F(R_s(x))$.
\end{proof}
For a fixed objective, componentwise recovery is sufficient but may be stronger than necessary, since weighted deficits can cancel.

\paragraph{Certificate terminology used in the low-degree theorems.}
A literal is $x_i$ or $1-x_i$; these are also called unary terms. A rooted atom is a root literal minus a product containing it, with repeated occurrences permitted; the other occurrences are its heads. It is nonnegative on the cube. Its standard atom roof is the minimum of the root literal and the sum of the complements of its head literals. An exact negaform writes the objective as an affine function minus a nonnegative sum of literal products. A nonnegative rooted certificate is a finite exact identity consisting of a nonnegative constant, nonnegative unary terms, and a nonnegative sum of rooted atoms. It is \emph{compatible} when every product containing a complemented literal has a complemented root, and every positive-only product has a positive root. The \emph{literal degree} counts occurrences in the product. Write $\CC_{\le d}$ for the cone of finite nonnegative combinations of $1$, literals, and compatible atoms of literal degree at most $d$. Here a cone generated by a family consists of its finite nonnegative linear combinations. Definition~\ref{def:certificate} gives the occurrence-indexed form used for larger certificates. The intrinsic quadratic and cubic constructions use only uniform-polarity products: all coordinates or all complements.

For a rooted certificate, positive unaries require
$\sum_sp_s[R_s(x)]_i\ge\rho x_i$. Complemented unaries lose no value because $1-[R_s(x)]_i\ge1-x_i$; a nonnegative constant also needs no additional condition when $\rho\le1$. We say that a map or portfolio \emph{pays} a component by factor $\rho$ when it satisfies \eqref{eq:component_payment} for that component. Thus atomwise payment and the positive-unary inequality suffice.

\subsection{Input and optimization access}\label{sec:access}
We supply continuous polynomials as sparse lists of rational coefficients and exponents, and measure complexity in the total binary encoding length. The low-degree algorithms construct their certificates from these coefficients. Later sections specify the encodings needed for supplied certificates, local tables, and cardinality sequences.

For fixed multilinear degree $d$, M\"obius inversion uses only values on sets of size at most $d$. We can therefore recover the polynomial with $n^{O(d)}$ exact value queries and arithmetic operations. A polynomial bound on bit complexity also requires polynomially bounded encoding lengths for the returned rational values. Boolean queries alone cannot recover a general continuous polynomial.

For $\delta\ge0$, a \emph{$\delta$-accurate maximizer} of a function $H$ over $\K$ is a feasible point $\bar x$ with $H(\bar x)\ge\max_\K H-\delta$. The analytical guarantees apply whenever we can obtain such a roof optimizer. For unconditional polynomial-time implementations, we use a rational polytope $\K$ with an explicit polynomial-size inequality system, or a polynomially bounded rational-polytope encoding together with a polynomial-time strong-separation oracle. On a queried rational point, this oracle either confirms membership or returns a rational hyperplane that strictly separates the point from the polytope \citep{Schrijver2003}. The coefficient-input algorithms have polynomial-size lifted roofs. After reversing the sign, the quadratic hybrid roof gives a convex quadratic program. For a general compact convex set, we assume a procedure that returns a \emph{feasible} roof optimizer with the prescribed additive accuracy.

A rooted atom roof uses one auxiliary variable constrained below its two affine bounds. The quadratic hybrid roof retains the concave diagonal residual; the cubic roof is piecewise affine. Section~\ref{sec:implementation_size} gives the formulation sizes.

We evaluate powers with binary exponents by repeated squaring and controlled rounding. Approximating fixed algebraic maps coordinatewise downward preserves feasibility. The following continuity bound controls the error in the objective value.

A sparse polynomial $F(x)=\sum_\beta c_\beta x^\beta$ satisfies
\[
 |F(x)-F(y)|\le M_F\|x-y\|_\infty,\qquad
 M_F=\sum_\beta |c_\beta|\,|\beta|\quad(x,y\in[0,1]^N).
\]
The bound follows by integrating the first-derivative bounds along the segment from $x$ to $y$. For a certificate, an analogous bound is the sum of each atom coefficient times one plus its literal degree, together with the unary contributions. The bound has polynomial logarithmic encoding length even when multiplicities are given in binary. We can therefore round an algebraic recovered point downward to a rational point with any prescribed additive loss. If each branch value is evaluated within $\eta$, choosing the largest estimate loses at most $2\eta$. Thus the implementation does not require exact comparisons of algebraic numbers.

\section{Quadratics: a complete roof-and-recovery algorithm}\label{sec:quadratic_main}
The quadratic algorithm has three steps: separate the diagonal terms, orient the pair coefficients into nonnegative atoms, and clip a roof optimizer to obtain a feasible point. The oriented coefficients also give a directed-cut representation. An odd cycle will show why downward recovery from this roof cannot guarantee more than one-half.

\subsection{The multilinear part and the exact diagonal residual}\label{sec:quadratic_representation}
Every quadratic DR polynomial has the form
\begin{equation}\label{eq:quadratic}
 F(x)=c+\sum_i\ell_ix_i-\sum_id_ix_i^2-\sum_{i<j}b_{ij}x_ix_j,
 \qquad d_i,b_{ij}\ge0.
\end{equation}
The multilinear polynomial agreeing with it at every Boolean point is
\[
 M(x)=c+\sum_i(\ell_i-d_i)x_i-\sum_{i<j}b_{ij}x_ix_j.
\]
Then
\begin{equation}\label{eq:diagonal_split}
 F(x)=M(x)+\sum_id_ix_i(1-x_i).
\end{equation}
If $F$ is Boolean-nonnegative, so is $M$. We construct a degree-two certificate of $M$ and retain the diagonal residual exactly. Its concavity lets us use a tighter bound than the piecewise-linear self-atom roof.

\subsection{Orienting the pair coefficients}\label{sec:quadratic_orientation}
Put $a_i=\ell_i-d_i$, so $M(x)=c+\sum_i a_ix_i-\sum_{i<j}b_{ij}x_ix_j$, and let $f$ be its Boolean restriction:
\begin{equation}\label{hard:eq:problem}
 f(S)=c+\sum_{i\in S}a_i-
 \sum_{\{i,j\}\subseteq S}b_{ij},
 \qquad b_{ij}\ge0.
\end{equation}

Split each $b_{ij}$ into nonnegative directed weights $\theta_{ij}$ and $\theta_{ji}$. If this leaves a negative linear coefficient, a complemented unary with coefficient $q_i$ absorbs it. We choose the weights by solving
\begin{equation}\label{eq:quadratic_orientation}
\begin{aligned}
 \min\quad&\sum_iq_i\\
 \text{s.t.}\quad&\sum_{j\ne i}\theta_{ij}-q_i\le a_i &&(i\in N),\\
 &\theta_{ij}+\theta_{ji}=b_{ij} &&(i<j),\\
 &\theta,q\ge0.
\end{aligned}
\end{equation}
Missing pairs have weight zero. Any split is feasible after increasing the $q_i$ sufficiently. Eliminating the dual variables for the equality constraints gives
\[
 \max_{0\le u\le1}
 \left\{\sum_{i<j}b_{ij}\min\{u_i,u_j\}-\sum_i a_iu_i\right\}.
\]
For $S_t=\{i:u_i\ge t\}$, the identities
\[
 u_i=\int_0^1\ind\{i\in S_t\}\,dt,
 \qquad
 \min\{u_i,u_j\}=\int_0^1\ind\{i,j\in S_t\}\,dt
\]
express the fractional dual value as an average of Boolean values. It is therefore at most
\[
 \max_{S\subseteq N}\left\{\sum_{\{i,j\}\subseteq S}b_{ij}-\sum_{i\in S}a_i\right\}
 =c-\min_{S\subseteq N}f(S),
\]
and a Boolean vector attains this bound. Strong linear-programming duality therefore gives the same primal optimum. For an optimal allocation, set
\[
 p_i=a_i-\sum_{j\ne i}\theta_{ij}+q_i\ge0,
 \qquad r=c-\sum_iq_i=\min_S f(S)\ge0.
\]
Expanding the terms gives the identity
\begin{equation}\label{hard:eq:cut-decomposition}
 M(x)=r+\sum_i p_ix_i+\sum_iq_i(1-x_i)
       +\sum_{i\ne j}\theta_{ij}x_i(1-x_j).
\end{equation}
This rational program has polynomial size and admits an optimal solution of polynomial encoding length. The threshold-set argument extends to both uniform polarities and higher powers in Section~\ref{sec:uniform_orientation}.

\subsection{The roof and clipping}
Replacing every directed edge in \eqref{hard:eq:cut-decomposition} by its roof gives
\[
 L_M(x)=r+\sum_i p_ix_i+\sum_iq_i(1-x_i)
       +\sum_{i\ne j}\theta_{ij}\min\{x_i,1-x_j\}.
\]
Add the exact diagonal residual from \eqref{eq:diagonal_split} to obtain the algorithm's concave upper bound. The edge terms admit a linear hypograph formulation; the diagonal terms are concave quadratics.

Define $\phi_2(t)=\min\{t,1/2\}$. For $a,u\in[0,1]$,
\begin{equation}\label{eq:clipping}
 \phi_2(a)\ge a/2,\qquad
 \phi_2(a)(1-\phi_2(u))\ge\tfrac12\min\{a,1-u\}.
\end{equation}
If $a\le1/2$, the product is at least $a/2$; otherwise it is at least $(1-u)/2$. Interchanging root and head gives the complemented one-head inequality, also when both occurrences refer to the same coordinate.

\begin{theorem}[Quadratic half approximation]\label{thm:quadratic}
Every Boolean-nonnegative quadratic DR polynomial \eqref{eq:quadratic} admits a deterministic $1/2$ approximation over $\K$, up to any prescribed additive error, under Section~\ref{sec:access}. The same factor applies to every supplied degree-two compatible rooted certificate with its standard roof.
\end{theorem}
\begin{proof}
Let $L_M$ be the standard roof of the certificate for $M$ in \eqref{eq:diagonal_split}. Optimize the hybrid concave roof
\[
 H_F(x)=L_M(x)+\sum_id_ix_i(1-x_i)
\]
to additive accuracy $\delta$, and put $y_i=\phi_2(\bar x_i)$. Equation~\eqref{eq:clipping} pays each rooted quadratic term and positive unary by $1/2$; constants and complemented unaries lose no value. Clipping also preserves the diagonal residual: if $\bar x_i\le1/2$, then $y_i=\bar x_i$; otherwise
\[
 y_i(1-y_i)-\bar x_i(1-\bar x_i)
 =\tfrac14-\bar x_i(1-\bar x_i)=(\bar x_i-1/2)^2\ge0.
\]
Consequently $F(y)\ge H_F(\bar x)/2\ge\OPT(F,\K)/2-\delta/2$. Choose tolerances with total loss at most $\varepsilon$. For a supplied degree-two compatible certificate, the positive and complemented uniform-polarity atoms satisfy the same clipping inequality. A mixed atom has the form $(1-a)(1-u)$, including identified coordinates, and roof $\min\{1-a,1-u\}$. Both recovered factors are at least $1/2$ and dominate their original values, so their product is at least half this roof. Constants and complemented unaries lose no value, so aggregation gives the same factor for its standard roof.
\end{proof}
The diagonal-free sharpness examples in Section~\ref{sec:frontiers} apply both to the hybrid roof and to the standard atom roof.

\paragraph{The implemented quadratic algorithm.}
Construct \eqref{hard:eq:cut-decomposition}, maximize the hybrid roof to the accuracy specified in Section~\ref{sec:access}, and return $y_i=\min\{\bar x_i,1/2\}$. Negating the roof gives a convex quadratic program; when the diagonal residual vanishes, the formulation is linear. Algorithm~\ref{alg:recovery} specifies the numerical tolerances.

\subsection{Lattice-submodular quadratics on a box}\label{sec:quadratic_box}
On a box, positive diagonal curvature can be removed before applying Theorem~\ref{thm:quadratic} and restored by endpoint choices afterward. After substituting fixed coordinates and normalizing the remaining intervals to the unit cube, a lattice-submodular quadratic has the form
\begin{equation}\label{eq:lattice_quadratic}
 F(x)=c+\sum_i\ell_ix_i+\sum_i\kappa_ix_i^2
       -\sum_{i<j}b_{ij}x_ix_j,
 \qquad b_{ij}\ge0,\quad\kappa_i\in\R.
\end{equation}
Lattice submodularity constrains the mixed coefficients but allows either sign for $\kappa_i$. Put $I_+=\{i:\kappa_i>0\}$ and replace each positive square by its endpoint chord:
\begin{equation}\label{eq:quadratic_chord}
 G(x)=F(x)+\sum_{i\in I_+}\kappa_ix_i(1-x_i).
\end{equation}
Then $G\ge F$, its mixed derivatives agree with those of $F$, and $\partial_{ii}^2G=2\min\{\kappa_i,0\}\le0$. Thus $G$ is DR-submodular and is affine in every coordinate in $I_+$.

\begin{corollary}[Box-constrained submodular quadratic maximization]\label{cor:quadratic_box}
Let $B=\prod_i[\alpha_i,\beta_i]$ be a nonempty finite box with rational bounds, and let $F$ be a lattice-submodular quadratic with rational coefficients, nonnegative on $B$. For every $\varepsilon>0$, a deterministic polynomial-time algorithm returns $z\in B$ with
\[
 F(z)\ge\tfrac12\max_{x\in B}F(x)-\varepsilon.
\]
\end{corollary}
\begin{proof}
Substitute fixed coordinates and map each remaining interval to $[0,1]$ by $x_i=\alpha_i+(\beta_i-\alpha_i)t_i$. If none remain, return the unique point. This positive affine change preserves degree, mixed derivative signs, and nonnegativity. Work with the normalized polynomial and its surrogate \eqref{eq:quadratic_chord}.

Starting from any $y$ in the normalized cube, visit the coordinates in $I_+$ once. Replace the current coordinate by whichever endpoint, zero or one, gives the larger value of $G$, choosing zero in a tie. Affinity in that coordinate means the replacement cannot decrease $G$. At the final point $z$, every correction in \eqref{eq:quadratic_chord} vanishes, so $F(z)=G(z)\ge G(y)$. Applying this procedure to a maximizer of $G$, and using $G\ge F$, proves
\[
 \max_{[0,1]^n}G=\max_{[0,1]^n}F.
\]
Theorem~\ref{thm:quadratic} gives a rational point $y$ with $G(y)\ge\tfrac12\max G-\varepsilon$. The endpoint procedure transfers this guarantee to $F$ without further loss. Map its output back to $B$. Normalization and its inverse have polynomial encoding length; the added step uses at most $2|I_+|$ exact rational value evaluations.
\end{proof}

The algorithm therefore consists of the existing orientation LP, concave quadratic roof optimization, clipping, and one pass of endpoint choices. Form $Z$ by independently replacing each $y_i$, $i\in I_+$, with a Bernoulli variable of mean $y_i$, leaving the other coordinates fixed. Then $\E F(Z)=G(y)$, and the endpoint procedure makes this rounding deterministic. Endpoint choices can increase coordinates; their feasibility uses the product structure of the box.

\begin{example}[A convex coordinate and a fractional optimum]\label{ex:quadratic_box}
On $[0,1]^2$, take
\[
 F(x_1,x_2)=x_1^2+(1-x_1)x_2+x_2(1-x_2).
\]
This function is nonnegative and lattice submodular, while $\partial_{11}^2F=2>0$. Replacing $x_1^2$ by $x_1$ gives $G=x_1(1-x_2)+x_2+x_2(1-x_2)$, with hybrid roof
\[
 H_G(x)=\min\{x_1,1-x_2\}+x_2+x_2(1-x_2)
       \le1+x_2-x_2^2\le5/4.
\]
The point $\bar x=(1/2,1/2)$ attains this roof bound. Clipping leaves it fixed, with $G(\bar x)=1$ and $F(\bar x)=3/4$. Comparing $G(0,1/2)=3/4$ and $G(1,1/2)=5/4$ chooses $z=(1,1/2)$. Thus $F(z)=5/4$, attaining the certified optimum while retaining a fractional coordinate with negative curvature.
\end{example}

\citet[Theorem~13]{niazadeh2020optimal} obtain a randomized $1/2$ approximation, up to additive error, for general nonnegative Lipschitz continuous lattice-submodular functions on boxes, using coordinatewise double greedy and zero-sum games. Their class already includes \eqref{eq:lattice_quadratic}. For explicit quadratics, the chord reduction and Theorem~\ref{thm:quadratic} give a shorter construction and proof of the same factor, with an explicit roof certificate. The same reduction can also precede their deterministic DR algorithm \citep[Theorem~14]{niazadeh2020optimal}; the simplification here comes from the quadratic representation and the lossless endpoint step.

\subsection{A fixed-terminal cut representation}

For a loopless directed graph $G=(V,A)$ with nonnegative arc weights $w$, the directed
cut leaving $T\subseteq V$ has value
\[
 \delta^+(T):=\sum_{(u,v)\in A}w_{uv}\ind\{u\in T,v\notin T\}.
\]
A \emph{fixed-terminal directed-cut representation} has two terminals
$s,t$, requires $s\in T$ and $t\notin T$, and counts only nonterminal
vertices toward the cardinality budget.

\begin{proposition}[Fixed-terminal equivalence]\label{hard:prop:cut-equivalence}
A quadratic submodular pseudo-Boolean function is globally nonnegative if
and only if it is the cut-value function of a directed graph with
nonnegative arc weights and two fixed terminals.  A representation can be
constructed in polynomial time, and the cardinality of $S$ equals the number
of selected nonterminal vertices.
\end{proposition}

\begin{proof}
The orientation construction in Section~\ref{sec:quadratic_orientation} gives the nonnegative coefficients in \eqref{hard:eq:cut-decomposition}, with $r=\min_S f(S)$. Introduce arcs $s\to t$, $i\to t$, $s\to i$, and $i\to j$ with weights $r,p_i,q_i$, and $\theta_{ij}$, respectively. The cut leaving $\{s\}\cup S$ has value $M(\ind_S)=f(S)$. Its number of selected nonterminal vertices is exactly $|S|$, so the original cardinality budget is preserved. The orientation program and the graph construction have polynomial encoding length.

Conversely, a fixed-terminal directed cut with nonnegative arc weights is a nonnegative sum of $1$, $x_i$, $1-x_i$, and $x_i(1-x_j)$ after zero-contribution arcs are discarded. Its value function is globally nonnegative, quadratic, and submodular.
\end{proof}

Because the arc weight $r=\min_S f(S)$ must be nonnegative, the construction requires nonnegativity on the whole Boolean cube, including infeasible sets.

\paragraph{Discrete applications.}
A loopless directed graph with nonnegative arc weights has a multilinear quadratic cut objective of this form. Both an upper cardinality bound and a matroid-independence polytope are down-closed, and Section~\ref{sec:rounding} converts the recovered point to a feasible set without further loss. The terminals in Proposition~\ref{hard:prop:cut-equivalence} enter as constants, so recovery acts only on the nonterminal coordinates.

\subsection{Relation to linear and semidefinite relaxations}\label{sec:quadratic_relaxations}
McCormick's bilinear inequalities provide a basic linearization of quadratic optimization \citep{mccormick1976computability}. Goemans and Williamson established semidefinite relaxation and randomized rounding for cut problems \citep{goemans1995improved}; Nesterov studied semidefinite approximation for quadratic models with diagonal quadratic constraints \citep{nesterov1998semidefinite}; and Helmberg, Rendl, and Weismantel developed semidefinite bounds and cutting-plane strengthenings for quadratic knapsack \citep{helmberg2000semidefinite}. Semidefinite rounding gives stronger guarantees in particular cut models \citep{BrakensiekHuangPotechinZwick2026,goemans1995improved}. Our theorem applies to the stated down-closed regions while retaining the continuous diagonal residual. To compare the bounds precisely, we fix the objective and moment constraints of the relaxation below.

For the multilinear part, the standard roof has the pairwise form
\[
 L_M(x)=c+\sum_i a_ix_i-\sum_{i<j}b_{ij}\max\{0,x_i+x_j-1\}.
\]
Here $X_{ij}$ is a Boolean pair-moment variable, intended to represent $\Pr(i,j\in S)$ for a random set $S$ with marginals $x$. At fixed $x$, each pair moment takes its smallest value allowed by the McCormick inequalities. The semidefinite programming (SDP) bound $L_{\rm SDP}$, obtained by adding a positive-semidefinite constraint on the first and second moments while retaining those inequalities, is concave and satisfies $M\le f^+\le L_{\rm SDP}\le L_M$. Appendix~\ref{app:quadratic_comparison} gives the formulation and proof. The identity $X_{ii}=x_i$ concerns Boolean moments; the original continuous objective still contains $x_i^2$, with its diagonal residual kept separately. Tightening the upper bound alone does not establish a stronger recovery factor.

\subsection{A first boundary: the odd directed cycle}\label{sec:quadratic_boundary}
For an odd cycle on $n\ge3$ vertices, let
\[
 M_n(x)=\sum_{i=1}^n x_i(1-x_{i+1}),\qquad
 L_n(x)=\sum_{i=1}^n\min\{x_i,1-x_{i+1}\},\qquad x_{n+1}=x_1.
\]
The roof has the unique cube maximizer $\bar x=\ind/2$, with value $n/2$. Every partial derivative of $M_n$ is nonnegative on $[0,\bar x]$, so the best downward value is $n/4$. Even coordinated recovery is therefore limited to one-half of this roof value. The full Boolean optimum is $(n-1)/2$, and the ratio $n/[2(n-1)]$ tends to one-half. Section~\ref{sec:frontiers} proves these claims and gives an obstruction that also survives representation optimization. 

\section{Cubics: new algebra and two-map recovery}\label{sec:cubic_main}\label{sec:representation}\label{sec:kernels}
Quadratics have nonpositive pair coefficients and a separable, concave residual after Boolean multilinearization. Cubics can nevertheless have positive cubic coefficients, and their residual after Boolean multilinearization can fail to be concave. An exact complement identity gives a certificate and concave roof without relying on a concave residual. Two common maps then provide the recovery guarantee.

\subsection{What changes beyond the quadratic decomposition}
Consider the first term of the worked example below:
\[
 F(x,y)=(1-y)(2x-x^2),\qquad M(x,y)=x(1-y).
\]
This continuous polynomial is nonnegative and DR-submodular, and $M$ extends its Boolean restriction. The residual $F-M=x(1-x)(1-y)$ is nonnegative and vanishes at Boolean points, but
\[
 \nabla^2(F-M)=\begin{pmatrix}-2(1-y)&2x-1\\2x-1&0\end{pmatrix},
 \qquad \det\nabla^2(F-M)=-(2x-1)^2.
\]
The residual is therefore not generally concave. Moreover, the positive coefficient of $x^2y$ is compatible with all the DR inequalities of $F$. A cubic algorithm must address both features.

\subsection{The DR test and the complement identity}
Every second derivative of a cubic is affine, so maximizing it over the cube amounts to setting the coordinates with positive coefficients to one. This gives an explicit DR test, including repeated-variable and diagonal terms.

Write a cubic polynomial as
\begin{equation}\label{eq:cubic_input}
 F(x)=c+\sum_i a_ix_i+\sum_{|\alpha|=2}q_\alpha x^\alpha
                   +\sum_{|\beta|=3}t_\beta x^\beta.
\end{equation}
For multiindices, $\binom\beta\alpha=\prod_i\binom{\beta_i}{\alpha_i}$ and $\partial^\alpha=\prod_i\partial_i^{\alpha_i}$.
\begin{lemma}[Exact cubic DR condition]\label{lem:cubic_DR}
The polynomial \eqref{eq:cubic_input} is DR-submodular on the cube if and only if
\begin{equation}\label{eq:cubic_DR}
 q_\alpha+\sum_{\substack{|\beta|=3\\\beta\ge\alpha}}
 \binom\beta\alpha\pos{t_\beta}\le0\qquad(|\alpha|=2).
\end{equation}
This includes $\alpha=2e_i$, hence the diagonal conditions.
\end{lemma}
\begin{proof}
For $|\alpha|=2$, the normalized second derivative is the affine function
\[
 \frac1{\alpha!}\partial^\alpha F(x)
 =q_\alpha+\sum_{\substack{|\beta|=3\\\beta\ge\alpha}}
   \binom\beta\alpha t_\beta x^{\beta-\alpha}.
\]
Its cube maximum sets to one exactly those coordinates with positive coefficients. The resulting maximum is the left side of \eqref{eq:cubic_DR}.
\end{proof}

For example, the repeated-variable identity
\[
 x^2y=1-2x-y+x^2+2xy-(1-x)^2(1-y)
\]
replaces a positive cubic product by a negative complemented product and lower-degree terms. The multiplicities on $x^2$ and $xy$ match those in the normalized second derivatives. In the multilinear case, this is the classical unary-negaform transformation \citep{BillionnetMinoux1985}. Retaining the multiplicities makes the identity valid on the continuous cube. The quadratic allocation argument can then be extended to the two polarities.

\subsection{Uniform-polarity orientation}\label{sec:uniform_orientation}
Consider an exact uniform-polarity negaform
\begin{equation}\label{eq:uniform_negaform}
 F(x)=c+\sum_i a_ix_i-\sum_{\eta\in H^+}b_\eta x^\eta
                  -\sum_{\beta\in H^-}d_\beta(1-x)^\beta,
 \qquad b_\eta,d_\beta\ge0,
\end{equation}
where the multiindices are nonnegative integers of total degree at least two. Split each nonlinear coefficient among its distinct coordinate roots by solving
\begin{equation}\label{eq:orientation}
\begin{aligned}
 \min\quad&\sum_iq_i\\
 \text{s.t.}\quad&\sum_{\eta:\,i\in\supp\eta}\theta^+_{\eta,i}
       -\sum_{\beta:\,i\in\supp\beta}\theta^-_{\beta,i}-q_i\le a_i &&(i\in N),\\
 &\sum_{i\in\supp\eta}\theta^+_{\eta,i}=b_\eta &&(\eta\in H^+),\\
 &\sum_{i\in\supp\beta}\theta^-_{\beta,i}=d_\beta &&(\beta\in H^-),\\
 &\theta^+,\theta^-,q\ge0.
\end{aligned}
\end{equation}
Occurrences of the same coordinate and polarity share one root allocation, so the exponents need not be expanded. The program minimizes the complemented-unary mass needed to make the remaining linear coefficients nonnegative. A threshold-set argument for its dual determines the residual constant.

\begin{theorem}[Orientation theorem]\label{thm:orientation}
The value of \eqref{eq:orientation} is
\begin{equation}\label{eq:orientation_value}
 c-\sum_{\beta\in H^-}d_\beta-\min_{S\subseteq N}F(\ind_S).
\end{equation}
If $F$ is Boolean-nonnegative, an optimal solution yields a nonnegative uniform-polarity rooted certificate of the same maximum literal degree, with residual constant exactly $\min_SF(\ind_S)$. The certificate is computable by a rational linear program of size polynomial in the sparse support-incidence input. Conversely, a nonnegative certificate implies nonnegativity on the cube.
\end{theorem}
\begin{proof}
An arbitrary distribution of each coefficient becomes feasible after increasing $q_i$ sufficiently. Eliminating the dual variables for the equality constraints gives
\begin{equation}\label{eq:orientation_dual}
 \max_{0\le u\le1}\left\{
 \sum_\eta b_\eta\min_{i\in\supp\eta}u_i
 -\sum_\beta d_\beta\max_{i\in\supp\beta}u_i-\sum_i a_iu_i\right\}.
\end{equation}
For $S_t=\{i:u_i\ge t\}$, integrate the identities
\[
 u_i=\int_0^1\ind\{i\in S_t\}\,dt,\quad
 \min_{i\in T}u_i=\int_0^1\ind\{T\subseteq S_t\}\,dt,\quad
 \max_{i\in T}u_i=\int_0^1\ind\{T\cap S_t\ne\varnothing\}\,dt.
\]
The dual objective becomes
\[
 \int_0^1\left[
 \sum_{\eta:\supp\eta\subseteq S_t}b_\eta
 -\sum_{\beta:\supp\beta\cap S_t\ne\varnothing}d_\beta
 -\sum_{i\in S_t}a_i\right]dt.
\]
This average cannot exceed the largest bracketed set value, and the corresponding Boolean vector attains that value. At a set $S$, the bracket equals $c-\sum_\beta d_\beta-F(\ind_S)$. Strong linear-programming duality gives \eqref{eq:orientation_value}.

Set
\[
 p_i=a_i-\sum_\eta\theta^+_{\eta,i}+\sum_\beta\theta^-_{\beta,i}+q_i,
 \qquad c_0=c-\sum_\beta d_\beta-\sum_iq_i.
\]
Missing root-index terms are zero. The constraints give $p_i\ge0$, and the optimum value gives $c_0=\min_SF(\ind_S)\ge0$. Expanding
\begin{align*}
 F(x)={}&c_0+\sum_ip_ix_i+\sum_iq_i(1-x_i)\\
 &+\sum_{\eta,i}\theta^+_{\eta,i}(x_i-x^\eta)
 +\sum_{\beta,i}\theta^-_{\beta,i}\bigl((1-x_i)-(1-x)^\beta\bigr)
\end{align*}
recovers every coefficient of \eqref{eq:uniform_negaform}. Each atom is nonnegative because its root is one of the factors in the product. All operations have polynomial encoding length.
\end{proof}
The Boolean dual depends only on supports, whereas the polynomial identity retains multiplicities. The proof therefore covers both distinct-variable multilinear negaforms and continuous powers. Orientation supplies a nonnegative certificate; Proposition~\ref{prop:root_independence} shows that the roof of a fixed negaform is independent of this choice.

The complement identity produces exactly the quadratic coefficients controlled by \eqref{eq:cubic_DR}.

\begin{proposition}[Cubic certificate construction]\label{prop:cubic_certificate}
Every Boolean-nonnegative cubic DR polynomial has a polynomially constructible nonnegative uniform-polarity rooted certificate of literal degree at most three. Its optimally oriented residual constant is its minimum Boolean value. Conversely, every nonnegative sum of constants, literals, and uniform-polarity rooted atoms of literal degree at most three is Boolean-nonnegative and DR-submodular.
\end{proposition}
\begin{proof}
For each positive cubic coefficient, use the exact identity
\begin{equation}\label{eq:cubic_complement}
 x^\beta=1-\sum_i\beta_ix_i+
 \sum_{\substack{|\alpha|=2\\\alpha\le\beta}}\binom\beta\alpha x^\alpha
 -(1-x)^\beta\qquad(|\beta|=3).
\end{equation}
The remaining quadratic coefficients are nonpositive by Lemma~\ref{lem:cubic_DR}. The result is a uniform-polarity negaform of literal degree at most three. Apply Theorem~\ref{thm:orientation}. Conversely, each uniform-polarity atom has nonpositive diagonal and mixed second partials: its nonlinear part is the negative of a product of literals of one polarity. Nonnegative sums preserve this property and nonnegativity.
\end{proof}

\subsection{Occurrence slots and the failure of clipping}
Treat the root and head values as independent real variables, called \emph{occurrence slots}. For an integer $m\ge1$ and $a,u_1,\ldots,u_m\in[0,1]$, define
\begin{equation}\label{eq:slot_atoms}
\begin{aligned}
 h_m^+(a,u)&=a\left(1-\prod_{j=1}^m u_j\right),&
 \ell_m^+(a,u)&=\min\left\{a,\sum_j(1-u_j)\right\},\\
 h_m^-(a,u)&=(1-a)\left(1-\prod_{j=1}^m(1-u_j)\right),&
 \ell_m^-(a,u)&=\min\left\{1-a,\sum_j u_j\right\}.
\end{aligned}
\end{equation}
The \emph{rank} is the number $m$ of heads, so the literal degree is $m+1$. A scalar map is applied to a vector coordinatewise. Repeated variables are allowed by restricting independent occurrence slots to the relevant diagonal subspaces.

For the positive cubic atom $h_2^+(a,u,v)=a(1-uv)$ at $(a,u,v)=(1,1/2,1/2)$, the roof equals one, but quadratic clipping gives only $\tfrac12(1-\tfrac14)=3/8$. This falls short of $8/17$ and raises a broader question: how much can any one common scalar map guarantee?

\subsection{A single-map obstruction already in cubic degree}\label{sec:single_upper}
Suppose a downward map $\phi$ pays every positive-root atom through rank two by factor $\rho\ge0$, and put $r=\phi(1/2)\in[0,1/2]$. A positive rank-one midpoint atom gives
\[
 \rho\le2r(1-r).
\]
For any $0<\eta<1$, let $q_\eta=\phi(1-\eta)$. A positive rank-one atom with root $\eta$ and head $1-\eta$ gives
\[
 \rho\eta\le\phi(\eta)(1-q_\eta)\le\eta(1-q_\eta),
 \qquad\text{hence}\qquad q_\eta\le1-\rho.
\]
A positive rank-two atom with root $1-\eta$ and two midpoint heads gives
\[
 \rho(1-\eta)\le q_\eta(1-r^2)\le(1-\rho)(1-r^2).
\]
Letting the scalar parameter $\eta$ tend to zero yields
\[
 \rho\le\frac{1-r^2}{2-r^2}.
\]
This argument takes no limit of $\phi$. On $[0,1/2]$, $2r(1-r)$ increases and $(1-r^2)/(2-r^2)$ decreases. They intersect at $r=(\sqrt3-1)/2$, since
\[
 2r(1-r)(2-r^2)-(1-r^2)=(r-1)^2(2r^2+2r-1).
\]
The common value is $2\sqrt3-3$, proving the upper bound.

The bound uses only positive-root atoms through degree three and assumes no regularity of the map. Section~\ref{sec:single} constructs a continuous scalar map attaining it at every finite rank. Since $2\sqrt3-3<8/17$, a universal componentwise $8/17$ guarantee requires more than one scalar map.

\subsection{An elementary cubic portfolio}\label{sec:cubic_maps}
The following two branches satisfy one weighted inequality for both atom polarities and every positive unary:
\begin{equation}\label{eq:elementary_maps}
\begin{aligned}
 \phi_0(t)&=\begin{cases}
 t,&0\le t\le1/2,\\ t/4+3/8,&1/2\le t\le3/4,\\9(1-t)/4,&3/4\le t\le1,
 \end{cases}\\
 \phi_1(t)&=\begin{cases}
 t,&0\le t\le1/4,\\t/3+1/6,&1/4\le t\le3/4,\\(7t-4)/3,&3/4\le t\le1.
 \end{cases}
\end{aligned}
\end{equation}
Both maps are continuous and downward. The first fixes $1/2$ and maps $1$ to zero; the second maps $1/2$ to $1/3$ and fixes $1$. Downwardness does not require either map to be increasing.

\paragraph{The two contact configurations.}
At the midpoint edge, the two branches attain fractions $1/2$ and $4/9$ of the roof value. At the positive cubic atom with root one and two midpoint heads, they attain zero and $8/9$. Weighting the first branch by $p$ gives
\[
 \frac p2+\frac{4(1-p)}9,
 \qquad \frac{8(1-p)}9.
\]
Balancing these values gives $p=8/17$ and common value $8/17$. The contact configurations thus determine the weights. To prove the guarantee, however, we must verify the inequality for every occurrence-slot configuration.

\paragraph{From contact points to the whole domain.}
For $m=1,2$ and $\sigma\in\{+,-\}$, define the \emph{recovery deficit} by
\[
 D_m^\sigma(a,u)=8h_m^\sigma(\phi_0(a),\phi_0(u))
 +9h_m^\sigma(\phi_1(a),\phi_1(u))-8\ell_m^\sigma(a,u).
\]
The required scalar inequality is $D_m^\sigma(a,u)\ge0$.
 A \emph{knot} is a breakpoint between affine pieces; a \emph{knot box} is a Cartesian product of intervals between consecutive knots. A \emph{switching face} of an atom roof is the hyperplane where the two affine arguments of its minimum agree. An off-switch region lies on one side of that hyperplane.
Within a knot box and a fixed affine branch of the roof, the recovery deficit is affine in each occurrence slot. Moving a minimizing slot to a knot or to the switching face reduces the continuum to knot configurations and switching triangles. Appendix~\ref{app:cubic} checks them all. Its triangle certificates use the Bernstein basis, whose nonnegative functions sum to one; nonnegative coefficients therefore certify the inequality. Two triangles need separate arguments because some coefficients are negative.

\begin{lemma}[Cubic scalar recovery]\label{lem:cubic_maps}
For $m=1,2$, $\sigma\in\{+,-\}$, and all occurrence-slot values,
\begin{equation}\label{eq:cubic_slot_inequalities}
 8h_m^\sigma(\phi_0(a),\phi_0(u))+
 9h_m^\sigma(\phi_1(a),\phi_1(u))\ge8\ell_m^\sigma(a,u).
\end{equation}
Also $8\phi_0(t)+9\phi_1(t)\ge8t$ for $t\in[0,1]$.
\end{lemma}
\begin{proof}
Appendix~\ref{app:cubic} proves the inequalities by checking the breakpoints and certifying the switching faces. The slots range independently over real values, so the same inequalities hold when repeated occurrences are identified.
\end{proof}

\begin{theorem}[Cubic approximation]\label{thm:cubic}
Under Section~\ref{sec:access}, every explicitly represented Boolean-nonnegative cubic DR polynomial admits a deterministic $8/17$ approximation over $\K$, up to any prescribed additive error. More generally, the same elementary portfolio applies to any supplied certificate in $\CC_{\le3}$.
\end{theorem}
\begin{proof}
Proposition~\ref{prop:cubic_certificate} constructs a uniform-polarity certificate for the intrinsic polynomial class. Lemma~\ref{lem:cubic_maps} pays its atoms and positive unaries, while constants and complemented unaries lose no value. Add these common inequalities using Corollary~\ref{cor:aggregation}. For supplied compatible certificates, Lemma~\ref{lem:mixed} covers the mixed-polarity atoms. If $\bar x$ is a $\delta$-accurate roof optimizer, Theorem~\ref{thm:common} gives value at least $8(\OPT(F,\K)-\delta)/17$ for the better of $\phi_0(\bar x)$ and $\phi_1(\bar x)$. Allocate the numerical errors as in Section~\ref{sec:access}.
\end{proof}

\paragraph{A certificate on a switching triangle.}
On the complemented switching face $a+u+v=1$, the exceptional cell with $u,v\ge0$ and $u+v\le1/4$ has the exact deficit
\[
 3(u-v)^2+uv\bigl[4-3(u+v)\bigr]\ge0.
\]
The identity certifies the entire triangle. Appendix~\ref{app:cubic} treats the remaining cells, including the positive-polarity exception.

\Needspace{26\baselineskip}
\subsection{Implementing the algorithm}\label{sec:algorithm}
\begin{algorithm}[H]
\caption{Certificate, roof, and common downward recovery}\label{alg:recovery}
\centering
\fbox{\begin{minipage}{0.95\textwidth}
\small

\textbf{Input.} A rational quadratic/cubic coefficient list satisfying the stated promises, or a representation for one of the supplied classes in Table~\ref{tab:extensions}; the feasible-region access of Section~\ref{sec:access}; and a rational tolerance $\varepsilon>0$.
\begin{enumerate}[leftmargin=*,itemsep=3pt,topsep=4pt]
\item Construct a nonnegative certificate by Section~\ref{sec:quadratic_orientation} or Proposition~\ref{prop:cubic_certificate}, or read the supplied certificate/decomposition. For the quadratic hybrid implementation, separate the exact diagonal residual as in \eqref{eq:diagonal_split}.
\item Form the corresponding concave roof $H$. Use clipping with $\rho=1/2$ for quadratics, the elementary pair with $\rho=8/17$ for cubics, or the maps certified for the supplied class.
\item Obtain a feasible $\delta$-accurate maximizer $\bar x$ of $H$, with $\delta=\varepsilon/3$.
\item Evaluate every prescribed common branch at $\bar x$. Where necessary, approximate it coordinatewise downward by a rational point $y^s$, losing at most $\varepsilon/3$ in objective value per branch.
\item Estimate every $F(y^s)$ with absolute error at most $\varepsilon/6$ and return a branch with largest estimate.
\end{enumerate}
\textbf{Output.} A point $y\in\K$ with $F(y)\ge\rho\OPT(F,\K)-\varepsilon$. A discrete solution additionally requires the rounding hypothesis in Section~\ref{sec:rounding}.
\end{minipage}}
\end{algorithm}

For roof-optimization error $\delta$, map-approximation loss at most $\varepsilon_{\rm map}$ per branch, and value-estimation error at most $\eta$, Theorem~\ref{thm:common} gives
\begin{equation}\label{eq:implementation_error}
 F(y)\ge\rho\OPT(F,\K)-\rho\delta-\varepsilon_{\rm map}-2\eta.
\end{equation}
Because $0\le\rho\le1$, the tolerances in Algorithm~\ref{alg:recovery} give total loss at most $\varepsilon$. When $M_F>0$, the bound in Section~\ref{sec:access} makes coordinate accuracy $\varepsilon/(3M_F)$ sufficient; for a constant objective, any feasible point suffices. Rational piecewise-affine branches can be evaluated exactly at rational $\bar x$.
\subsection{A constrained continuous cubic example}\label{sec:worked_cubic}
Consider the rational coefficient input
\[
 F(x_1,x_2)=2x_1-x_1^2-2x_1x_2+x_1^2x_2+2x_2-2x_2^2.
\]
We maximize it over the unit simplex. Factoring the polynomial gives
\begin{equation}\label{eq:worked_cubic}
 F(x_1,x_2)=(1-x_2)\bigl[1-(1-x_1)^2\bigr]+2x_2(1-x_2),
 \qquad \K=\{x\ge0:x_1+x_2\le1\}.
\end{equation}
This proves nonnegativity on the cube. Its Hessian is
\[
 \nabla^2F=\begin{pmatrix}
 -2(1-x_2)&-2(1-x_1)\\
 -2(1-x_1)&-4
 \end{pmatrix},
\]
so both the diagonal and mixed DR conditions hold. The Boolean restriction extends to $M(x)=x_1(1-x_2)$, but the coefficient-input algorithm evaluates recovered points using the original continuous objective $F$.

\paragraph{From coefficients to a negaform.}
Apply \eqref{eq:cubic_complement} to the sole positive cubic term, $x_1^2x_2$, and collect coefficients:
\[
 F(x)=1+x_2-2x_2^2-(1-x_1)^2(1-x_2).
\]
In \eqref{eq:uniform_negaform}, this gives $(c,a_1,a_2)=(1,0,1)$ and two subtracted products: the positive-only product $x_2^2$ with weight $2$, and the complemented product $(1-x_1)^2(1-x_2)$ with weight $1$.

\paragraph{Constructing the orientation.}
Let $\theta_2^+$ be the weight assigned to the root $x_2$ of the positive product, and let $\theta_1^-,\theta_2^-$ be the weights assigned to the roots $1-x_1,1-x_2$ of the complemented product. As in \eqref{eq:orientation}, $q_1,q_2$ are the coefficients of the complemented-unary corrections. The two occurrences of $1-x_1$ define the same rooted atom, so $\theta_1^-$ records their total allocation; the products retain their powers. Program~\eqref{eq:orientation} becomes
\[
\begin{aligned}
 \min\quad &q_1+q_2\\
 \text{s.t.}\quad
 &-\theta_1^- -q_1\le0,\\
 &\theta_2^+-\theta_2^- -q_2\le1,\\
 &\theta_2^+=2,\qquad \theta_1^-+\theta_2^-=1,\\
 &\theta_2^+,\theta_1^-,\theta_2^-,q_1,q_2\ge0.
\end{aligned}
\]
Eliminating the equalities gives $0\le\theta_1^-\le1$ and $q_2\ge\theta_1^-$. The objective is nonnegative and attains zero at the unique aggregated optimum
\[
 (\theta_2^+,\theta_1^-,\theta_2^-,q_1,q_2)=(2,0,1,0,0).
\]
Indeed, objective value zero forces $q_1=q_2=0$, and hence $\theta_1^-=0$.
The dual in \eqref{eq:orientation_dual} specializes to
\[
 \max_{0\le u_1,u_2\le1}\{u_2-\max(u_1,u_2)\}=0,
\]
with, for example, $u_1=u_2=0$ attaining the bound. This matching dual certificate verifies the allocation's optimality. The residual coefficients are
\[
 p_1=\theta_1^-+q_1=0,\qquad
 p_2=1-\theta_2^++\theta_2^-+q_2=0,\qquad
 c_0=1-1-q_1-q_2=0.
\]
Substituting this optimal allocation constructs the certificate
\[
 F(x)=\bigl[(1-x_2)-(1-x_1)^2(1-x_2)\bigr]
       +2\bigl[x_2-x_2^2\bigr].
\]

\paragraph{Roof optimization.}
The standard roof is
\begin{equation}\label{eq:worked_roof}
 L(x)=\min\{1-x_2,2x_1\}+2\min\{x_2,1-x_2\}.
\end{equation}
The roof program has the explicit lift
\[
\begin{aligned}
 \max\quad& z_1+2z_2\\
 \text{s.t.}\quad&x\in\K,\quad
 z_1\le1-x_2,\quad z_1\le2x_1,\\
 &z_2\le x_2,\quad z_2\le1-x_2.
\end{aligned}
\]
For $x_2\le1/2$, $L\le1+x_2\le3/2$; for $x_2\ge1/2$, $L\le3(1-x_2)\le3/2$. The roof maximizers are exactly $x_2=1/2$, $1/4\le x_1\le1/2$. Choose $\bar x=(1/2,1/2)$, whose roof value is $3/2$.

\paragraph{Recovery and benchmarks.}
The elementary pair gives $\phi_0(\bar x)=(1/2,1/2)$ and $\phi_1(\bar x)=(1/3,1/3)$, with original values $7/8$ and $22/27$, respectively. We therefore return $(1/2,1/2)$ with value $7/8$. Comparing this feasible value with the roof optimum certifies an instance-specific ratio of $(7/8)/(3/2)=7/12$.

Since $\partial_1F=2(1-x_2)(1-x_1)\ge0$, a continuous maximum can be sought on $x_2=1-x_1$, where $F=2x_1-x_1^3$. Hence
\[
 x^*=\bigl(\sqrt{2/3},1-\sqrt{2/3}\bigr),\qquad
 \OPT(F,\K)=\frac{4\sqrt6}{9}>1.
\]
The actual approximation ratio is $21\sqrt6/64$; the computable roof certificate gives the lower bound $7/12$. The best feasible Boolean value is $1$, which exceeds the recovered value $7/8$.

\subsection{The same contacts in the sharpness geometry}\label{sec:cubic_boundary}
The two maps' midpoint and endpoint values also arise in the cubic sharpness construction. For $0\le q\le1$ and $0\le t\le1/2$, put
\[
 H(q,t)=\frac5{12}+\frac q{12}+t-\left(1+\frac q2\right)t^2.
\]
Then
\[
 \frac23-H(q,t)
 =(1-q)(t-1/2)^2+\frac32q(t-1/3)^2.
\]
Equality holds only at $(q,t)=(0,1/2)$ and $(1,1/3)$. Section~\ref{sec:frontiers} embeds this expression in a down-closed instance with roof value approaching $17/12$ and lower-box maximum $2/3$. A better feasible point lies outside that box. These equality configurations also constrain higher-degree portfolios in Section~\ref{sec:degree13}.

\paragraph{A multilinear application.}
With three distinct coordinates, $x_i(1-x_jx_k)$ is the probability of selecting $i$ while not selecting both $j$ and $k$. Nonnegative sums of these directed-hyperedge indicators, their complemented counterparts, and unaries form submodular cubic set functions. Apply the cubic recovery theorem to their multilinear extensions and then the rounding guarantee in Section~\ref{sec:discrete}. Identifying coordinates produces continuous powers; the scalar proof still applies, but the discrete interpretation changes.

\subsection{Formulation size and bit complexity}\label{sec:implementation_size}
Let $s$ be the number of nonzero monomials in a rational polynomial of degree at most three. Each positive cubic monomial in \eqref{eq:cubic_complement} generates at most three linear terms and three quadratic terms, in addition to one complemented product and a constant. After collecting coefficients, the negaform therefore has $O(s+n)$ encoded terms. If $r$ is its number of nonlinear products and $J$ the sum of their distinct support sizes, then $J\le3r$. Program~\eqref{eq:orientation} has $J+n$ variables, $r$ coefficient equalities, and $n$ residual inequalities, in addition to nonnegativity. A certificate has at most $J$ rooted atoms and $2n$ unaries. Its standard roof uses at most $J$ auxiliary variables and $2J$ affine inequalities; the root-independent form \eqref{eq:standard_roof} can instead use $r$ auxiliaries and $2r$ inequalities.

At fixed degree, expanding and collecting coefficients preserves polynomial binary encoding length. The determinant bounds for rational linear programs also give an optimal basic solution of the orientation LP with polynomial encoding length \citep{Schrijver2003}. Thus cubic roof optimization is a polynomial-size rational LP over an explicitly encoded polytope. The quadratic hybrid roof adds at most $n$ concave diagonal terms; negating it gives a convex quadratic program. Both conclusions also hold under the bounded rational-polytope and strong-separation assumptions in Section~\ref{sec:access}. Feasible optimization to the prescribed additive tolerance, followed by the fixed finite recovery maps, has polynomial bit complexity in the input length and $\log(1+1/\varepsilon)$. For a general compact convex region, a roof-optimization procedure must be supplied separately. Equation~\eqref{eq:implementation_error} bounds the numerical losses.

\Needspace{14\baselineskip}
\part{Extensions, recovery boundaries, and discrete consequences}\label{part:extensions}
\section{Beyond cubics I: larger rooted certificates}\label{sec:larger_atoms}
Allowing more literal occurrences in a rooted certificate requires recovery inequalities for products with more heads. We assume that the certificate is supplied, except in the negatively signed subclass, where orientation constructs one. A degree bound alone does not guarantee a certificate.

\subsection{Occurrence-indexed certificates}\label{sec:literals}
To account for repeated variables, we index literal occurrences rather than coordinates. A product has uniform polarity when all its literals are positive or all are complemented; otherwise it has mixed polarity. A literal monomial is a finite product indexed by occurrences
\[
 m_\sigma(x)=\prod_{q=1}^{d_\sigma}\lambda_{\sigma q}(x),
 \qquad \lambda_{\sigma q}\in\{x_i,1-x_i:i\in N\}.
\]
We take $d_\sigma\ge2$, absorbing products of degree zero or one into the affine part. Occurrences may repeat, including both polarities of one coordinate. The \emph{literal degree} $d_\sigma$ counts occurrences with multiplicity, not distinct coordinates. Multiplicities may be stored in binary: $m_\sigma=x^{\mu^+_\sigma}(1-x)^{\mu^-_\sigma}$.

\begin{definition}[Exact negaform and compatible rooted certificate]\label{def:certificate}
An exact negaform is a polynomial identity
\begin{equation}\label{eq:negaform}
 F(x)=A(x)-\sum_{\sigma\in\Sigma}b_\sigma m_\sigma(x),\qquad b_\sigma\ge0,
\end{equation}
where $A$ is affine. For any occurrence index $r\in\{1,\ldots,d_\sigma\}$ in $m_\sigma$, rooting at $r$ defines two functions:
\begin{equation}\label{eq:atom_roof}
 h_{\sigma,r}=\lambda_{\sigma r}-m_\sigma,
 \qquad
 \ell_{\sigma,r}=\min\left\{\lambda_{\sigma r},\sum_{q\ne r}(1-\lambda_{\sigma q})\right\}.
\end{equation}
The distinguished occurrence is the \emph{root}; the remaining occurrences are its \emph{heads}. An expression $h_{\sigma,r}$ is a \emph{rooted atom}. A root is \emph{polarity-compatible} if a positive-only monomial uses a positive root and every monomial containing a complemented occurrence uses a complemented root. A nonnegative compatible rooted certificate subordinate to \eqref{eq:negaform} is an identity
\begin{equation}\label{eq:root_certificate}
 F(x)=c_0+\sum_i p_ix_i+\sum_i q_i(1-x_i)
       +\sum_{\sigma,r}\theta_{\sigma,r}h_{\sigma,r}(x),
\end{equation}
with all coefficients nonnegative, all selected roots compatible, and
$\sum_r\theta_{\sigma,r}=b_\sigma$. Let $\CC_{\le d}$ denote the cone generated by the constant function one, literals, and compatible rooted atoms of literal degree at most $d$. Here a \emph{conic combination} is a finite sum with nonnegative coefficients, and the cone generated by a family consists of all its conic combinations.
\end{definition}
Denote the exact negaform~\eqref{eq:negaform} by $\N$.
Membership in $\CC_{\le d}$ guarantees a finite certificate, but the algorithms require a supplied or constructible encoding of polynomial size. It does not imply DR-submodularity. For example, rooting the mixed product $x(1-x)$ at $1-x$ gives the rooted atom $h=(1-x)-x(1-x)=(1-x)^2$, whose second derivative is positive.

\begin{proposition}[Root independence for a fixed negaform]\label{prop:root_independence}
For every root occurrence,
\[
 \ell_{\sigma,r}-\lambda_{\sigma r}
 =\min\left\{0,\sum_{q=1}^{d_\sigma}(1-\lambda_{\sigma q})-1\right\}
 =:\psi_\sigma.
\]
Thus a fixed negaform induces the root-independent \emph{standard concave roof}
\begin{equation}\label{eq:standard_roof}
 L_{\N}(x)=A(x)+\sum_\sigma b_\sigma\psi_\sigma(x).
\end{equation}
It majorizes $F$ on the cube and agrees with $F$ at Boolean points. If a subordinate certificate exists, this is also the sum of its atom roofs and affine terms.
\end{proposition}
\begin{proof}
Subtract the root literal from the minimum in \eqref{eq:atom_roof}. The union bound gives
$\lambda_r(1-\prod_{q\ne r}\lambda_q)\le\min\{\lambda_r,\sum_{q\ne r}(1-\lambda_q)\}$.
Each correction $\psi_\sigma$ is concave, since it is the minimum of two affine functions, and the bound is exact for Boolean literal values. Summation and coefficient matching prove the result.
\end{proof}

With distinct underlying coordinates, the standard atom roof is the continuous concave envelope; Section~\ref{sec:local} gives the coupling argument for uniform polarities. Repeated occurrences can make it nonminimal, as Example~\ref{ex:self_atom} shows.

\paragraph{Encoding.}
A supplied certificate records rational weights, coordinate indices, polarities, and binary multiplicities. An atom roof adds one auxiliary variable and two affine inequalities; multiplicities appear in their coefficients and need not be expanded. Unless a construction is available, the certificate forms part of the input. Objective and map evaluation follow the accuracy assumptions of Section~\ref{sec:access}.

\subsection{A constructive signed subclass}
An NNS (nonlinearly negatively signed) polynomial is supplied in the form
\begin{equation}\label{eq:NNS}
 F(x)=c+\sum_i a_ix_i-\sum_{\beta\in B}b_\beta x^\beta,
 \qquad b_\beta\ge0,\quad |\beta|\ge2.
\end{equation}
The sign condition applies to the continuous monomial expansion, rather than to a multilinearization of the Boolean restriction. Every such polynomial is DR-submodular. With $H^-=\varnothing$, the orientation value becomes
\[
 \max_{S\subseteq N}\left\{\sum_{\beta:\supp\beta\subseteq S}b_\beta-\sum_{i\in S}a_i\right\}
 =c-\min_SF(\ind_S).
\]
Thus Boolean-nonnegative NNS inputs have polynomially constructible positive-only rooted certificates, whose literal degree equals their polynomial degree.

\subsection{Mixed heads from uniform-polarity inequalities}
The compatibility rule also permits heads of mixed polarity. Downwardness bounds their recovered product by a complemented-only product, so the complemented slot inequalities suffice.

\begin{lemma}[Mixed-polarity domination]\label{lem:mixed}
Suppose a downward scalar portfolio pays positive unaries and both slot-atom types through rank $d-1$ by factor $\rho\in[0,1]$. It then pays every compatible rooted certificate in $\CC_{\le d}$ by the same factor.
\end{lemma}
\begin{proof}
It remains to treat mixed-polarity atoms. Compatibility forces their root to be $1-a$. Write each head literal as $1-v_j$. A complemented head recovers as $1-\phi_s(v_j)$. For a positive head, downwardness gives
\[
 \phi_s(1-v_j)\le1-v_j\le1-\phi_s(v_j).
\]
The recovered head product is therefore no larger than the complemented-only product, and the recovered atom is at least $h_m^-(\phi_s(a),\phi_s(v))$. The roof equals $\ell_m^-(a,v)$. Sum the atom and unary inequalities using Corollary~\ref{cor:aggregation}.
\end{proof}

\subsection{One common map at arbitrary literal degree}\label{sec:single}
Put
\[
 \gamma_*=2\sqrt3-3,\qquad q_*=1-\gamma_*,\qquad
 \sigma_*=15-\frac{17\sqrt3}{2},
\]
and define
\begin{equation}\label{eq:single_map}
 \phi_*(t)=\sigma_*t+(\gamma_*-\sigma_*)\min\{t,2/3\}
                  +(1-\gamma_*)\min\{t,1/4\}.
\end{equation}
Its slopes are $1,\gamma_*,\sigma_*$, and $\phi_*(1)=q_*$.
\begin{theorem}[Optimal common-coordinatewise single-map factor]\label{thm:single}
The continuous map \eqref{eq:single_map} satisfies $\gamma_*t\le\phi_*(t)\le t$ and, for every finite $m\ge1$,
\[
 h_m^\sigma(\phi_*(a),\phi_*(u))\ge\gamma_*\ell_m^\sigma(a,u)
 \qquad(\sigma\in\{+,-\}).
\]
No downward scalar map gives a larger common factor for these atoms and positive unaries. No continuity, monotonicity, or measurability is required of a competing map. The upper bound already holds for positive-root atoms of literal degree at most three.
\end{theorem}
\begin{proof}
Appendix~\ref{app:single} proves the inequalities for every finite rank. The matching upper bound in Section~\ref{sec:single_upper} uses only positive-root atoms through rank two and applies to every downward scalar map.
\end{proof}
For literal degree two, clipping and the midpoint cut give the exact single-map factor $1/2$. For every degree cap at least three, Theorem~\ref{thm:single} gives the exact value $\gamma_*$. By Lemma~\ref{lem:mixed}, every polynomial supplied with a compatible rooted certificate of arbitrary finite literal degree admits the deterministic factor $\gamma_*$, up to the additive error specified in Section~\ref{sec:access}.

\subsection{The degree-twelve portfolio}\label{sec:twelve}
Let $\alpha_{12},\beta_{12}$ be the rational piecewise-affine maps in Table~\ref{tab:knots}. They use the same 23 knots as the arity-four pair introduced later, with different ordinates. They satisfy
\[
 0\le\alpha_{12}(t),\beta_{12}(t)\le t,\quad
 \alpha_{12}(1/2)=1/2,\ \beta_{12}(1/2)=1/3,\quad
 \alpha_{12}(1)=0,\ \beta_{12}(1)=1.
\]
\begin{theorem}[Degree-twelve atomwise recovery]\label{thm:twelve_atoms}
For $m=1,\ldots,11$, $\sigma\in\{+,-\}$, and $a,u_1,\ldots,u_m\in~[0,1]$,
\begin{equation}\label{eq:twelve}
 8h_m^\sigma(\alpha_{12}(a),\alpha_{12}(u))+
 9h_m^\sigma(\beta_{12}(a),\beta_{12}(u))\ge8\ell_m^\sigma(a,u).
\end{equation}
Moreover $8\alpha_{12}(t)+9\beta_{12}(t)\ge8t$. Repeated variables and literal occurrences are permitted.
\end{theorem}
\begin{proof}
Appendix~\ref{app:twelve} identifies the formal statement and source version proving nonnegativity of both slotwise deficits and the unary deficit. The knot-agreement results match the formal maps to Table~\ref{tab:knots}, and the atom definitions give exactly \eqref{eq:twelve}. Since the inequality holds for all independent slots, identifying slots covers repeated occurrences.
\end{proof}
\begin{corollary}[Degree-twelve certificate recovery]\label{cor:twelve}
A polynomial supplied with a polynomial-size compatible certificate in $\CC_{\le12}$ admits a deterministic $8/17$ approximation over $\K$, up to any prescribed additive error. Optimize its standard roof, apply $\alpha_{12}$ and $\beta_{12}$, and keep the better branch.
\end{corollary}
\begin{proof}
Apply Theorem~\ref{thm:twelve_atoms}, Lemma~\ref{lem:mixed}, and Theorem~\ref{thm:common} with probabilities $8/17,9/17$.
\end{proof}
In particular, Boolean-nonnegative NNS polynomials of degree at most twelve have factor $8/17$ by Theorem~\ref{thm:orientation}; arbitrary finite-degree NNS polynomials have factor $\gamma_*$. Their degree-two subclass has the stronger $1/2$ factor.

Section~\ref{sec:degree13} proves a strict gap below $8/17$ at degree thirteen for termwise guarantees using finite common-coordinatewise portfolios and both polarities.

\section{Beyond cubics II: richer local set functions}\label{sec:local}
Another extension replaces rooted generators by arbitrary nonnegative submodular functions on small supports. Explicit value tables make their local concave closures tractable, although these functions need not admit the preceding rooted representations. To apply common recovery, we must establish an inequality for this larger local cone. Throughout this section the continuous objective is a multilinear extension.

\subsection{The multilinear local-function model}\label{sec:local_model}
A \emph{local conic combination} consists of finite $B$, local ground sets $N_b\subseteq N$, weights $w_b\ge0$, and nonnegative submodular functions $g_b:2^{N_b}\to\R_{\ge0}$ such that
\begin{equation}\label{eq:local_decomposition}
 f(S)=\sum_{b\in B}w_bg_b(S\cap N_b).
\end{equation}
The arity of $g_b$ is $|N_b|$. Overlap is unrestricted. If $G_b$ is the multilinear extension of $g_b$, then $F(x)=\sum_bw_bG_b(x_{N_b})$.

Projecting a global distribution onto the local ground sets gives
\begin{equation}\label{eq:local_roof}
 F(x)\le f^+(x)\le L(x):=\sum_bw_bg_b^+(x_{N_b}).
\end{equation}
The second inequality can be strict when the distributions attaining the local closures cannot be coupled into a single global distribution.

Unless a shorter encoding is specified, each local function is supplied by its full rational value table. The lifted program jointly optimizes their closures:
\begin{equation}\label{eq:local_LP}
\begin{aligned}
 \max\quad&\sum_bw_b\sum_{T\subseteq N_b}\lambda_{b,T}g_b(T)\\
 \text{s.t.}\quad&x\in\K,\qquad \lambda_{b,T}\ge0,\\
 &\sum_{T\subseteq N_b}\lambda_{b,T}=1 &&(b\in B),\\
 &\sum_{T\subseteq N_b:\,i\in T}\lambda_{b,T}=x_i &&(b\in B,\ i\in N_b).
\end{aligned}
\end{equation}
It has $n+\sum_b2^{|N_b|}$ variables and is polynomial in the full-table input length even without bounded arity.

\paragraph{Local-function classes.}
For $d\ge2$, let $\CQ_d$ be the cone of nonnegative submodular set functions of multilinear degree at most $d$, and let $\CL_d$ be those admitting \eqref{eq:local_decomposition} with arity at most $d$. For $i\in T$, $|T|\ge2$, write
\begin{equation}\label{eq:UV}
 U_{T,i}(x)=x_i-x^T,\qquad V_{T,i}(x)=(1-x_i)-(1-x)^T.
\end{equation}
Let $\CR_d$ consist of the set functions whose multilinear extensions are nonnegative sums of $1,x_i,1-x_i,U_{T,i},V_{T,i}$ with $|T|\le d$. A particular coefficient list is denoted by $\mathcal R$, and its additive atom-closure roof by $L_{\mathcal R}$. Set $\CR_\infty=\bigcup_{d\ge2}\CR_d$. These are distinct-variable, uniform-polarity classes, unlike the larger continuous certificate cones $\CC_{\le d}$.

A local function is \emph{concave-cardinality} if $g(S)=h(|S|)$ for a nonnegative sequence with nonincreasing consecutive differences. Let $\CA_d$ be the corresponding local-conic-combination class with arity at most $d$, and put $\CA_\infty=\bigcup_{d\ge1}\CA_d$. 
\subsection{Rooted multilinear functions and their exact closures}
For distinct variables indexed by $T$ and $i\in T$, the Boolean restrictions of $U_{T,i},V_{T,i}$ in \eqref{eq:UV} are
\[
 u_{T,i}(S)=\ind\{i\in S,\ T\nsubseteq S\},\qquad
 v_{T,i}(S)=\ind\{i\notin S,\ (T\setminus\{i\})\cap S\ne\varnothing\}.
\]
Their exact concave closures are
\begin{equation}\label{eq:rooted_closures}
 u_{T,i}^+(x_T)=\min\left\{x_i,\sum_{j\in T\setminus\{i\}}(1-x_j)\right\},\qquad
 v_{T,i}^+(x_T)=\min\left\{1-x_i,\sum_{j\in T\setminus\{i\}}x_j\right\}.
\end{equation}
The root probability and the sum of the head-event probabilities both bound each event probability. To attain the bounds, arrange the head events on a nonatomic probability space so that their union has measure equal to the smaller of one and the sum of their measures. Then arrange the root event to maximize its intersection with this union. Complementing the head events gives the first formula. The standard atom roof therefore equals the exact local closure for distinct variables; Example~\ref{ex:self_atom} explains why repeated variables require care.

\begin{proposition}[Class relations]\label{prop:classes}
For $d\ge2$,
\[
 \CR_d\subseteq\CL_d\subseteq\CQ_d,\qquad \CA_d\subseteq\CL_d,
 \qquad \CR_2=\CL_2=\CQ_2,\quad \CR_3=\CL_3=\CQ_3.
\]
On a four-element ground set, $\CR_\infty=\CR_4\subsetneq\CL_4=\CQ_4$. On ground sets of size at least five, $\CL_4$ and $\CR_\infty$ are incomparable. For $d<n$, $\CA_\infty\nsubseteq\CL_d$; for $n\ge4$, $\CA_\infty$ and $\CR_\infty$ are incomparable.
\end{proposition}
\begin{proof}
Every rooted function is nonnegative and submodular, and its degree is at most its arity. Nonnegative sums preserve both properties. The decreasing consecutive differences make concave-cardinality functions submodular. The reverse low-degree inclusions follow from Theorem~\ref{thm:orientation} and Proposition~\ref{prop:cubic_certificate}, restricted to multilinear inputs.

On four elements, consider the cardinality sequence $(0,2,3,3,3)$, with polynomial
\[
 F(x)=2\sum_i x_i-\sum_{i<j}x_ix_j+x_1x_2x_3x_4.
\]
Concavity places this function in $\CA_4\subseteq\CL_4$. In a nonnegative uniform-polarity rooted representation, every degree-four monomial has a nonpositive coefficient: a positive rooted product contributes a negative coefficient on its support, and a complemented rooted product contributes a nonpositive coefficient on every four-subset of its support. The positive quartic coefficient therefore excludes $\CR_\infty$, even after embedding the function into a larger ground set. Conversely, $u_{T,i}$ with $|T|=5$ has degree five and cannot belong to $\CL_4$.

The function $S\mapsto\min\{1,|S\cap A|\}$ with $|A|=d+1$ belongs to $\CA_{d+1}$ and has degree $d+1$, excluding $\CL_d$. Finally, $\ind\{1\in S,2\notin S\}$ cannot be a nonnegative sum of concave-cardinality functions. Restrict any proposed sum to $\{1,2\}$. Because the function vanishes at $\varnothing,\{2\},\{1,2\}$, every nonnegative summand must vanish at these sets. A two-coordinate summand has equal values on the two singletons, and a one-coordinate summand takes its selected value also on $\{1,2\}$. Every summand must consequently vanish at $\{1\}$, a contradiction. The four-element example proves the reverse noninclusion.
\end{proof}

An \emph{extreme ray} is a one-dimensional face of a cone. A rational ray has a unique \emph{primitive integer generator}: the nonzero integer vector on that ray whose coordinates have greatest common divisor one. Relabeling coordinates acts on these generators; a \emph{permutation orbit} is the set of generators obtained from one another by such relabelings.

\begin{proposition}[Recognition at fixed local arity]\label{prop:membership}
For each fixed $d$, membership of a rational degree-at-most-$d$ multilinear polynomial in $\CL_d$ is decidable by a rational linear program of size $n^{O(d)}$. A feasible solution supplies a local conic combination.
\end{proposition}
\begin{proof}
For each $A\subseteq N$ with $|A|\le d$, enumerate the primitive generators of the cone of nonnegative submodular value tables on $A$. Embed their multilinear coefficient vectors into $N$ and match the input coefficients using nonnegative multipliers. For fixed $d$, the number of generators per support is a constant, and there are $O(n^d)$ supports. The system is feasible precisely when the required decomposition exists.
\end{proof}
For arbitrary encoded arity, a rooted representation can instead be supplied succinctly. Corollary~\ref{cor:twelve} gives $8/17$ for a supplied representation in $\CR_{12}$, while Theorem~\ref{thm:single} gives $\gamma_*$ for $\CR_\infty$ without a degree bound.

\subsection{Arbitrary nonnegative submodular local functions on four elements}
Let $\psi_0,\psi_1$ be the arity-four maps specified separately from $\alpha_{12},\beta_{12}$ in Table~\ref{tab:knots}. Their probabilities are again $8/17$ and $9/17$.

The nonnegative submodular cone on four coordinates has finitely many extreme rays. Appendix~\ref{app:arity_four} verifies the common-map inequalities on these rays over the full marginal domain. Nonnegative aggregation then extends them to local functions with overlapping supports.

\begin{theorem}[Arity-four recovery]\label{thm:arity_four}
Let $g:2^A\to\R_{\ge0}$ be submodular with $|A|\le4$, and let $G$ be its multilinear extension. For every $x\in[0,1]^A$,
\begin{align}
 G(\phi_*(x))&\ge\gamma_*g^+(x),\label{eq:arity_single}\\
 8G(\psi_0(x))+9G(\psi_1(x))&\ge8g^+(x).\label{eq:arity_pair}
\end{align}
The factor in \eqref{eq:arity_single} is optimal among all downward scalar maps applied commonly to the coordinates, without regularity assumptions. The pair in \eqref{eq:arity_pair} is compatible with arbitrary overlap among local ground sets.
\end{theorem}
\begin{proof}
Appendix~\ref{app:arity_four} describes the finite reduction and the scope of the exact-arithmetic certificates for both inequalities. On four elements the cone has 222 primitive integer rays in 21 permutation orbits. The ray inequalities extend to every $g=\sum_r\lambda_r r$, because
\[
 G=\sum_r\lambda_rG_r,\qquad g^+\le\sum_r\lambda_r r^+,
 \qquad\lambda_r\ge0.
\]
Every ray and local function uses the same maps and weights. Corollary~\ref{cor:aggregation} therefore covers overlapping supports. For fixed-map sharpness in \eqref{eq:arity_single}, use a directed cut on two elements at $x=(1/2,1/2)$. With $z=\phi_*(1/2)=(1+\gamma_*)/4$, its ratio to the closure is $2z(1-z)=\gamma_*$. For optimality among all common downward scalar maps, apply Theorem~\ref{thm:single}: its positive-root witnesses have distinct coordinates and arity at most three, so their roofs are exact closures by \eqref{eq:rooted_closures}.
\end{proof}

For $|A|<4$, extend $g$ to a four-element set by ignoring the added coordinates and set their marginals to zero. Nonnegativity and submodularity are preserved. The multilinear value and concave closure are those on $A$, and every downward map fixes zero. Thus the four-coordinate verification includes all smaller arities.

\begin{corollary}[Arity-four optimization]\label{cor:arity_four}
Given a nonnegative local conic combination of nonnegative submodular functions of arity at most four, optimize \eqref{eq:local_LP}, apply $\psi_0,\psi_1$, and keep the better point. The result is an $8/17$ approximation to the multilinear maximum over $\K$, up to any prescribed additive error under Section~\ref{sec:access}.
\end{corollary}
\begin{proof}
Sum \eqref{eq:arity_pair} and apply Theorem~\ref{thm:common}.
\end{proof}

\begin{corollary}[Mixed component families]\label{cor:mixed_families}
Suppose a continuous objective is a nonnegative sum of components, each supplied either with a compatible rooted certificate of arbitrary finite literal degree or as the multilinear extension of a nonnegative submodular local function of arity at most four. The sum of the corresponding roofs admits the common recovery factor $\gamma_*$ under $\phi_*$.
\end{corollary}
\begin{proof}
Theorems~\ref{thm:single} and~\ref{thm:arity_four} use exactly the same scalar map. Apply Corollary~\ref{cor:aggregation}.
\end{proof}
Proposition~\ref{prop:classes} exhibits arity-four functions outside the distinct-variable uniform-polarity rooted cone. Appendix~\ref{app:maps} shows that neither $8/17$ pair satisfies the other family’s fixed-weight inequality. The mixed-family corollary therefore uses their shared single map; an $8/17$ guarantee for the union would require a common portfolio.

\section{Stronger structure that simplifies recovery}\label{sec:structured}
For globally monotone objectives or concave-cardinality components, the identity map suffices. In the first case, subtracting the modular terms given by each element's gain when added last (its local terminal marginal) produces monotone factors. In the second, one-dimensional closures and binomial expectations determine the recovery factor. The unary inequality also allows separable concave additions.

\subsection{Global monotonicity without local monotonicity}\label{sec:monotone}
Set $\gamma_0=\gamma_1=1$ and, for $d\ge2$, let
\[
 \gamma_d=1-(1-1/d)^d.
\]
\begin{theorem}[Monotone bounded-arity recovery]\label{thm:monotone}
Let $d\ge1$. Suppose $f=\sum_bw_bg_b$ is a nonnegative local conic combination of nonnegative submodular functions of arity at most $d$, and the full function $f$ is monotone. The individual $g_b$ need not be monotone. Then
\begin{equation}\label{eq:monotone_recovery}
 F(x)\ge\gamma_d L(x)\ge\gamma_d f^+(x).
\end{equation}
The constant is sharp. A $\delta$-accurate maximizer of $L$ over any compact feasible region gives value at least $\gamma_d(\max F-\delta)$ without coordinate decrease; down-closedness is unnecessary for this fractional conclusion.
\end{theorem}
\begin{proof}
Absorb empty-support terms into a constant. For $i\in N_b$, put
\[
 m_{b,i}=g_b(N_b)-g_b(N_b\setminus\{i\}),\qquad
 h_b(S)=g_b(S)-g_b(\varnothing)-\sum_{i\in S}m_{b,i}.
\]
Diminishing returns implies that every marginal of $h_b$ is nonnegative. Hence $h_b$ is monotone and submodular; since $h_b(\varnothing)=0$, it is also nonnegative. Let
\[
 c=f(\varnothing),\qquad
 r_i=\sum_{b:i\in N_b}w_bm_{b,i}=f(N)-f(N\setminus\{i\})\ge0.
\]
Let $H_b$ be the multilinear extension of $h_b$. Subtracting an affine function commutes with both extension operations, giving
\[
 F(x)=c+\sum_ir_ix_i+\sum_bw_bH_b(x_{N_b}),\quad
 L(x)=c+\sum_ir_ix_i+\sum_bw_bh_b^+(x_{N_b}).
\]
For a nonnegative monotone submodular function on $q$ elements, the finite-ground-set correlation inequality gives $H_b\ge\gamma_q h_b^+$ \citep[Lemma~4.10]{ChekuriVondrakZenklusen2014}. Since $\gamma_q$ is nonincreasing and $q\le d$, summing with $c,r_i,w_b\ge0$ proves \eqref{eq:monotone_recovery}. Sharpness follows from $g(S)=\ind\{S\ne\varnothing\}$ on $d$ elements at $x_i=1/d$: the closure is one and the multilinear value is $\gamma_d$.
\end{proof}
In particular, $\gamma_2=3/4$, $\gamma_3=19/27$, and $\gamma_4=175/256$. The first two apply to every monotone nonnegative quadratic or cubic set function, by Proposition~\ref{prop:classes}. These are sharp correlation and roof-recovery constants.

\subsection{Concave-cardinality components}\label{sec:cardinality}
Let $|A|=d\ge2$ and $g(S)=h(|S|)$, where $h(0),\ldots,h(d)$ is a nonnegative sequence with nonincreasing consecutive differences. Write $\bar h$ for its piecewise-affine interpolation.
\begin{lemma}[Exact cardinality closure]\label{lem:cardinality_closure}
The local concave closure is
\[
 g^+(x)=\bar h\left(\sum_{i\in A}x_i\right).
\]
\end{lemma}
\begin{proof}
Jensen's inequality gives the upper bound. If $s=\sum_i x_i<d$ and $k=\lfloor s\rfloor$, the slab $\{z\in[0,1]^A:k\le\sum_i z_i\le k+1\}$ is the convex hull of incidence vectors of sets of sizes $k$ and $k+1$. Indeed, a vertex cannot have two fractional coordinates, since they could be perturbed in opposite directions. Nor can it have exactly one: at an integer sum bound that coordinate must be integral, and away from both sum bounds it can be perturbed alone. A convex combination of these vertices representing $x$ gives a distribution attaining $\bar h(s)$. The case $s=d$ is immediate.
\end{proof}
\begin{theorem}[Sharp cardinality recovery]\label{thm:cardinality}
Set $\rho_0=\rho_1=1$ and $\rho_d=1-(1-1/d)^{d-1}$ for $d\ge2$. Every nonnegative concave-cardinality function on $d$ elements satisfies
\[
 G(x)\ge\rho_d g^+(x).
\]
Moreover, the infimum of $\max_{0\le y\le x}G(y)/g^+(x)$ over this class and positive denominators equals $\rho_d$. For any nonnegative sum of such local functions, the identity map gives factor $\min_b\rho_{|N_b|}$, regardless of overlap.
\end{theorem}
\begin{proof}
For $d=0,1$, the extension and closure are the same affine function; a positive constant witnesses sharpness of factor one. For $d\ge2$, Appendix~\ref{app:cardinality} proves the scalar inequality and its sharpness under arbitrary downward recovery. Every component uses the identity map, so summation and Theorem~\ref{thm:common} give the optimization guarantee.
\end{proof}
A sequence of length $d+1$ supplies a roof with $d$ affine hypograph constraints. The multilinear expectation is evaluated from the distribution of a sum of independent Bernoulli variables with possibly unequal success probabilities, called the \emph{Poisson--binomial distribution}. Convolving the Bernoulli coefficients requires $O(d^2)$ arithmetic operations and has polynomial bit complexity for rational inputs. No fixed arity is needed when the sequence is the input representation.

For symmetric zero-endpoint sequences, $h(k)=h(d-k)$ and $h(0)=h(d)=0$, the sharp identity-map constants are
\begin{equation}\label{eq:symmetric_cardinality}
 \rho_d^{\rm sym}=\begin{cases}
 1/2,&d=2,\\2/3,&d=3,\\5/8,&d=4,\\
 1-(1-1/d)^d-d^{-d},&d\ge5.
 \end{cases}
\end{equation}
Appendix~\ref{app:cardinality} proves the finite characterization and these values. Their infimum for $d\ge3$ is $5/8$, and they converge to $1-1/e$.

\subsection{Separable concave additions}\label{sec:addons}
For efficiently represented separable concave terms, subtract each endpoint chord and absorb the chords into the base polynomial. The remaining residual $r$ is concave, nonnegative, and satisfies $r(0)=r(1)=0$. If a downward portfolio satisfies $\sum_sp_s\phi_s(t)\ge\rho t$, then, for $t>0$,
\[
 \sum_sp_sr(\phi_s(t))\ge\frac{r(t)}t\sum_sp_s\phi_s(t)\ge\rho r(t).
\]
Every map vanishes at zero. These concave additions therefore preserve the portfolio guarantees whenever the base polynomial after chord absorption has the required nonnegative certificate and the resulting roof can be optimized efficiently.

\section{Sharpness of the specified recovery models}\label{sec:frontiers}
The sharpness results depend on the roof, benchmark, and permitted recovery rule, as recorded in Table~\ref{tab:sharpness_scope}. The examples distinguish losses from a fixed roof, representation constraints, incompatible local couplings, and the requirement to use common scalar maps. For termwise portfolios using both polarities, the gap becomes strict at degree thirteen.

\begin{table}[htbp]
\centering\small
\setlength{\tabcolsep}{4pt}
\caption{The comparison and recovery class in each sharpness theorem. Positive denominators are understood. Each guarantee at a roof optimizer is required at every such optimizer.}\label{tab:sharpness_scope}
\begin{tabularx}{\textwidth}{@{}p{.14\textwidth}YYY@{}}\toprule
Result & Roof / representation & Permitted recovery & Comparison and boundary\\\midrule
Theorem~\ref{thm:joint_frontiers} & Standard roof of a fixed negaform & Arbitrary joint, instance-dependent points below a roof maximizer & Roof value: $1/2$ at literal degree two; $8/17$ at three through twelve\\[4pt]
Theorem~\ref{thm:best_representation} & Infimum over uniform-polarity rooted representations of a multilinear objective & Arbitrary joint points below a maximizer of the optimized roof & True optimum: $1/2$ for quadratic and $8/17$ for cubic input\\[4pt]
Theorem~\ref{thm:aggregate_frontier} & Sum of the specified local concave closures; arity at most four & Arbitrary joint points in the lower box, with unrestricted overlap & Additive local roof: exactly $8/17$\\[4pt]
Theorem~\ref{thm:thirteen} & Termwise atom roofs through literal degree thirteen; both polarities and positive unaries & Finite common-coordinatewise portfolios with fixed weights & Supremum of universal termwise factors: strictly below $8/17$\\\bottomrule
\end{tabularx}
\end{table}

The first three rows permit arbitrary joint optimization within the lower box. The last requires common scalar maps satisfying a separate inequality for every generator. Part~\ref{part:hardness} treats polynomial-time hardness independently of these recovery models.

\subsection{Fixed-negaform joint downward recovery}
For a literal-degree cap $d$, define
\begin{equation}\label{eq:joint_frontier}
 \rho_d^{\rm joint}=
 \inf_{(F,\N,\K,\bar x)}
 \frac{\max\{F(y):0\le y\le\bar x\}}{L_\N(\bar x)}.
\end{equation}
The infimum ranges over all finite dimensions, supplied certificates in $\CC_{\le d}$ subordinate to an exact negaform $\N$, compact convex down-closed sets $\K$, and roof maximizers $\bar x$ with $L_\N(\bar x)>0$. Recovery may use the full instance and unlimited computation, but its output must lie below $\bar x$. Randomization cannot improve on the lower-box maximum. The infimum over $\bar x$ requires the guarantee to hold at every roof maximizer.

\begin{theorem}[Fixed-roof joint frontiers]\label{thm:joint_frontiers}
The values in \eqref{eq:joint_frontier} satisfy
\[
 \rho_2^{\rm joint}=\tfrac12,\qquad
 \rho_d^{\rm joint}=\tfrac8{17}\quad(3\le d\le12),\qquad
 \gamma_*\le\rho_d^{\rm joint}\le\tfrac8{17}\quad(d\ge13).
\]
The cubic upper bound holds already for positive-only NNS certificates. The quadratic and cubic bounds can also be realized asymptotically relative to the original objective optimum, not only to the roof value.
\end{theorem}
\begin{proof}
For degree two, the rank-one clipping inequalities in the proof of Theorem~\ref{thm:quadratic}, followed by Lemma~\ref{lem:mixed} and nonnegative aggregation, cover the entire certificate cone, including mixed-polarity atoms. Corollary~\ref{cor:twelve} and Theorem~\ref{thm:single} give the remaining lower bounds. The following constructions prove the quadratic and cubic upper bounds.

For odd $n\ge3$, use the directed cycle
\[
 F_n(x)=\sum_{i=1}^n x_i(1-x_{i+1}),\qquad
 L_n(x)=\sum_{i=1}^n\min\{x_i,1-x_{i+1}\},\quad x_{n+1}=x_1.
\]
On the cube, $L_n\le n/2$, with equality only at $\bar x=\ind/2$: equality in every minimum requires $x_i=1-x_{i+1}$, and oddness forces all coordinates to be $1/2$. On $[0,\bar x]$, each derivative $1-x_{i-1}-x_{i+1}$ is nonnegative, so the largest value is $n/4$. The roof ratio is exactly $1/2$. Multiaffinity gives the full-cube maximum $(n-1)/2$, so the ratio to that optimum is $n/[2(n-1)]\to1/2$. There is no diagonal residual, so this example also applies to the hybrid quadratic roof.

For the cubic bound, fix $0<\eta<1$ and put
\begin{equation}\label{eq:cubic_obstruction}
 \begin{split}
 G_\eta(x_0,x_1,x_2)&=\frac5{12\eta}x_0(1-x_1)
       +\frac12x_1(1-x_2^2)+x_2(1-x_2),\\
 U_\eta&=\frac{17}{12}-\frac\eta2,\qquad
 F_{\eta,\delta}(x,z)=G_\eta(x)+(U_\eta-\delta)z,
 \end{split}
\end{equation}
where $0<\delta<U_\eta$. Optimize over
\[
 \K_\eta=\left\{(x,z)\in[0,1]^4:
 x_0\le\eta(1-z),\quad x_1\le(1-\eta)(1-z),\quad
 x_2\le\tfrac12(1-z)\right\}.
\]
Every displayed term is nonnegative and has a positive-only rooted certificate. On this set the three nonlinear roofs equal $x_0,x_1,x_2$, respectively. Because their coefficients are positive, a roof maximizer saturates the coordinate bounds. The roof is at most
$U_\eta(1-z)+(U_\eta-\delta)z$.
Its unique maximizer is $(\eta,1-\eta,1/2,0)$ with value $U_\eta$.

Below this maximizer, $z=0$; monotonicity in the first coordinate permits $x_0=\eta$. Writing $s=x_1$ and $t=x_2$, the objective becomes
\[
 \frac5{12}(1-s)+\frac12s(1-t^2)+t(1-t).
\]
For fixed $t$, it is affine in $s\in[0,1-\eta]$. At $s=0$ its maximum is $2/3$ at $t=1/2$. At $s=1-\eta$, its maximum over $0\le t\le1/2$ is
\[
 \frac12-\frac\eta{12}+\frac1{2(3-\eta)}\le\frac23.
\]
The lower-box maximum is therefore $2/3$, and its ratio to $U_\eta$ tends to $8/17$. The incomparable feasible point $(0,0,0,1)$ has value $U_\eta-\delta$, so the ratio to the true optimum has the same upper limit as $\eta,\delta\downarrow0$. This cubic example belongs to every larger nested literal-degree class.
\end{proof}

\subsection{Representation-optimized multilinear roofs}
Let $F$ be the multilinear extension of a nonnegative submodular set function of degree at most three. Among all finite nonnegative uniform-polarity rooted representations $\mathcal R$ of $F$, initially allowing arbitrary supports, define
\begin{equation}\label{eq:best_representation_roof}
 L_F^*(x)=\inf_{\mathcal R}L_{\mathcal R}(x).
\end{equation}
The representation family is nonempty by Proposition~\ref{prop:cubic_certificate}. The bounds $F\le L_F^*\le L_{\mathcal R_0}$ for any fixed feasible representation make the roof finite-valued. An infimum of concave functions is concave, so it is a concave majorant. As an infimum of continuous roofs it is upper semicontinuous and thus attains its maximum on every nonempty compact feasible set. The infimum is over the stated distinct-variable uniform-polarity representations.

The freedom to change this roof can be expressed as a packing problem. After matching the cubic coefficients, one may add equal positive- and complemented-product mass to a triple only by consuming residual capacity on each of its three pairs. Appendix~\ref{app:representation_frontiers} characterizes all such changes by a fractional triple-packing system. If the graph of positive residual pair capacities is triangle-free, no triple can receive this extra mass, and every permitted representation induces the same roof. The resulting uniqueness of the roof prevents representation optimization from removing the following obstructions.

\begin{theorem}[Representation-optimized boundaries]\label{thm:best_representation}
The method first maximizes \eqref{eq:best_representation_roof} over a nonempty compact down-closed set and then may return the best point in $[0,\bar x]$, where $\bar x$ is a roof maximizer. Its largest universal ratio, required to hold for every roof maximizer and measured relative to the true multilinear optimum, is $1/2$ for nonnegative submodular polynomials of degree at most two and $8/17$ for those of degree at most three.
\end{theorem}
\begin{proof}
For a quadratic $F$, choose the degree-two representation constructed by Theorem~\ref{thm:orientation}; clipping recovers at least $L_{\mathcal R}(x)/2\ge L_F^*(x)/2$. For a cubic $F$, choose the degree-three representation constructed by Proposition~\ref{prop:cubic_certificate}; Lemma~\ref{lem:cubic_maps} gives weighted recovered value at least $8L_{\mathcal R}(x)/17\ge8L_F^*(x)/17$. Theorem~\ref{thm:common} proves the lower bounds at a maximizer of $L_F^*$. Appendix~\ref{app:representation_frontiers} expresses the remaining coefficient freedom as a linear program and constructs instances for which it cannot change the roof. Augmented odd cycles prove the quadratic upper bound; the cubic graph construction proves the cubic upper bound, with all roof maximizers agreeing on the coordinates relevant to the objective.
\end{proof}
Theorem~\ref{thm:best_representation} treats multilinear objectives and optimizes over the representation family in \eqref{eq:best_representation_roof}; the fixed-negaform construction \eqref{eq:cubic_obstruction} uses repeated variables.

\subsection{Additive local roofs and incompatibility under overlap}
For all local conic combinations of arity at most four, let
\[
 \rho_4^{\rm agg}=
 \inf_{(f,\{w_b,g_b,N_b\}),\,x:L(x)>0}
 \frac{\max_{0\le y\le x}F(y)}{L(x)}.
\]
The infimum ranges over all finite ground sets and finite decompositions, with $x\in[0,1]^N$. Here $L$ is the sum of the closures of the \emph{specified} local functions.
\begin{theorem}[Exact aggregate local-roof factor]\label{thm:aggregate_frontier}
The aggregate value is $\rho_4^{\rm agg}=8/17$. The upper bound already uses local functions of arity three and permits arbitrary joint downward recovery.
\end{theorem}
\begin{proof}
The lower bound is Theorem~\ref{thm:arity_four}. Appendix~\ref{app:overlap} constructs, for every integer $q\ge2$, an arity-three local decomposition and a point $\bar x_q$ at which
\[
 L(\bar x_q)=\frac{17}{12}-\frac1{2q},\qquad
 \max_{0\le y\le\bar x_q}F(y)=\frac23.
\]
The resulting ratio $8q/(17q-6)$ decreases to $8/17$. The point $\bar x_q$ maximizes the additive roof over its lower box.
\end{proof}
An explicit two-function example in the same appendix has ratio $277/562<1/2$, even though every element occurs in at most two local supports. The two local closures require incompatible joint laws on their intersection.

These constants are sharp for additive local roofs. The examples do not establish the same constants relative to the global concave closure $f^+$ or restrict recovery methods based on other roofs.

\subsection{Finite common-coordinatewise portfolios at degree thirteen}\label{sec:degree13}
A \emph{termwise} guarantee holds separately for every generator. For a normalized finite downward scalar portfolio $\Pi=\{(p_s,\varphi_s)\}$, define its degree-$d$ modulus to be the infimum of its weighted recovery-to-roof ratios over positive unaries and both slot-atom types of ranks $1,\ldots,d-1$, omitting zero denominators. Let $\rho_d^{\rm cw}$ be the supremum of this modulus over all such finite portfolios. No continuity is imposed here. Universal slotwise payment and Lemma~\ref{lem:mixed} identify this modulus with termwise certification on $\CC_{\le d}$.

\paragraph{Extremal mechanism.}
Only two branch values are needed: a value $q\in[0,1]$ near the endpoint and a midpoint value $t\in[0,1/2]$. The cubic witness is controlled by
\[
 H(q,t)=\frac5{12}(1-q)+\frac12q(1-t^2)+t(1-t).
\]
Its exact squared-deficit identity is
\begin{equation}\label{eq:thirteen_deficit_identity}
 H(q,t)=\frac23-(1-q)\left(t-\frac12\right)^2
                  -\frac32q\left(t-\frac13\right)^2.
\end{equation}
Thus equality requires $(q,t)=(0,1/2)$ or $(1,1/3)$. For portfolios approaching the upper bound, any distribution representing the limiting moments must be supported on these two points. A midpoint atom bounds the mass at the second point by $9/17$, while a complemented degree-thirteen atom requires more recovery than this mixture supplies. Compactness of a three-dimensional moment set justifies taking the limit.

\begin{theorem}[Strict degree-thirteen portfolio gap]\label{thm:thirteen}
The supremum of finite common-coordinatewise downward portfolio factors satisfies
\[
 \rho_{13}^{\rm cw}<\frac8{17}.
\]
Thus $8/17$ cannot even be approached by such portfolios through literal degree thirteen, irrespective of their number of branches or the regularity of their maps.
\end{theorem}
\begin{proof}
The cubic witness in \eqref{eq:cubic_obstruction} shows that every such factor is at most $8/17$: its roof tends to $17/12$ while every downward value is at most $2/3$. Suppose that portfolios certify factors $\rho_j\to8/17$, and choose $\eta_j\downarrow0$ with $0<\eta_j<1$. For a random branch of the $j$th portfolio, write
\[
 q=\phi(1-\eta_j),\qquad t=\phi(1/2),\qquad
 (q,t)\in[0,1]\times[0,1/2],
\]
and let $\E_j$ denote expectation over its finite branch distribution. For the function $H$ defined above, replacing $\phi(\eta_j)/\eta_j$ by its upper bound one only increases the recovered value of $G_{\eta_j}$ at $(\eta_j,1-\eta_j,1/2)$. The termwise guarantee then implies
\begin{equation}\label{eq:thirteen_cubic_moment}
 \E_j H(q,t)\ge\rho_j\left(\frac{17}{12}-\frac{\eta_j}{2}\right).
\end{equation}
The positive rank-one midpoint atom gives the second necessary inequality
\begin{equation}\label{eq:thirteen_pair_moment}
 \E_j[t(1-t)]\ge\frac{\rho_j}{2}.
\end{equation}
Finally, put
\[
 b=1-\left(\frac{23}{24}\right)^{12}<\frac25,
 \qquad \frac25-b=\frac{12079851931877}{182601737180282880}>0.
\]
For a complemented atom with root $a=1/2$ and twelve heads $u_i=1/24$, the roof is $1/2$. Downwardness bounds every recovered head coordinate by $1/24$, so its branch value is at most $b(1-t)$. Consequently
\begin{equation}\label{eq:thirteen_head_moment}
 b\E_j(1-t)\ge\frac{\rho_j}{2}.
\end{equation}

The three moments in \eqref{eq:thirteen_cubic_moment}--\eqref{eq:thirteen_head_moment} belong to the convex hull of the continuous image
\[
 (q,t)\longmapsto\bigl(H(q,t),\ t(1-t),\ 1-t\bigr)
\]
of the compact rectangle $[0,1]\times[0,1/2]$. This convex hull is compact in $\R^3$: by Carath\'eodory's theorem, each of its points is a convex combination of at most four image points. A convergent subsequence therefore yields a limiting moment vector represented by a finite probability distribution on the rectangle.

The identity \eqref{eq:thirteen_deficit_identity} bounds $H$ by $2/3$ and identifies its two equality points. Equation~\eqref{eq:thirteen_cubic_moment} forces the limiting expectation to equal $2/3$, so its representing distribution is supported on these two points. Let $w$ be its mass on $(1,1/3)$. The limit of \eqref{eq:thirteen_pair_moment} gives
\[
 \frac{1-w}{4}+\frac{2w}{9}\ge\frac4{17},
 \qquad\text{hence}\qquad w\le\frac9{17}.
\]
The limiting left side of \eqref{eq:thirteen_head_moment} is then at most
\[
 b\left(\frac{1-w}{2}+\frac{2w}{3}\right)
 \le\frac{10b}{17}<\frac4{17},
\]
a contradiction.
\end{proof}
Combining the lower portfolios, upper examples, and nesting gives
\[
 \rho_2^{\rm cw}=\tfrac12,\qquad
 \rho_d^{\rm cw}=\tfrac8{17}\quad(3\le d\le12),
\]
\[
 \gamma_*\le\rho_d^{\rm cw}\le\rho_{13}^{\rm cw}<\tfrac8{17}
 \quad(d\ge13).
\]
The sequence is nonincreasing. The strict gap applies to fixed-weight common-coordinatewise portfolios paying each atom separately in a family with both polarities. It does not establish a strict upper bound for arbitrary instance-dependent joint recovery, positive-only certificates, or recovery relative to the global concave closure.

\section{Stronger roofs and instance-specific recovery}\label{sec:design}
We consider two ways to tighten a roof: change the representation, or require local distributions to agree. We develop grouped formulations on supplied tree decompositions, then use finite games and cone duality to choose recovery rules. For a multilinear objective, recovery for a fixed instance reduces to choosing which coordinates to retain. Keeping feasible solutions and valid upper bounds lets us track the strongest guarantee obtained so far.

\subsection{Orientation is not representation choice}
By Proposition~\ref{prop:root_independence}, changing root weights within one exact negaform leaves its standard roof unchanged. Changing the negaform itself can change the roof, even when both representations are exact on the continuous cube. For example,
\[
 F(x)=2-x_1x_2-x_1x_3-x_2x_3+x_1x_2x_3
\]
has the exact identities
\begin{align*}
 F(x)&=3-x_1-x_2-x_3-(1-x_1)(1-x_2)(1-x_3),\\
 F(x)&=2-x_1x_3-x_2x_3-x_1x_2(1-x_3).
\end{align*}
Compatible nonnegative certificates give the roofs
\begin{align*}
 L_A(x)&=(1-x_2)+(1-x_3)+\min\{1-x_1,x_2+x_3\},\\
 L_B(x)&=(1-x_3)+\min\{x_3,1-x_1\}+\min\{x_3,1-x_2\}\\
 &\hspace{20mm}+\min\{1-x_3,2-x_1-x_2\}.
\end{align*}
At $x=(4/5,4/5,1/5)$, the two roofs equal $6/5$ and $8/5$, respectively. Given a polynomial-size menu of certificates, we can optimize each roof, apply its recovery rule, and keep the best original objective value. The common approximation factor is preserved.

\subsection{Grouping local functions before taking closure}
Let $\Pi$ partition the local-function index set $B$. For each block $C\in\Pi$, put
\[
 N_C=\bigcup_{b\in C}N_b,\qquad
 h_C(S)=\sum_{b\in C}w_bg_b(S\cap N_b),\qquad
 L_\Pi(x)=\sum_{C\in\Pi}h_C^+(x_{N_C}).
\]
We partition the functions, so the coordinate sets $N_C$ may still overlap. Write $\Pi\preceq\Pi'$ when $\Pi'$ is obtained by merging blocks.
\begin{theorem}[Nested grouped roofs]\label{thm:grouping}
Every $L_\Pi$ is concave and Boolean-exact, and
\[
 F\le f^+\le L_{\Pi'}\le L_\Pi\le L
 \qquad\text{whenever }\Pi\preceq\Pi'.
\]
The singleton partition gives $L$, and the one-block partition gives $f^+$. Any common recovery inequality paying $\rho L$ also pays $\rho L_\Pi$ and retains its guarantee after optimizing $L_\Pi$.
\end{theorem}
\begin{proof}
Mixing feasible fixed-marginal distributions proves concavity, and Boolean marginals force a point mass. Projecting a global distribution onto the block ground sets gives $f^+\le L_\Pi$. When two blocks merge, each expected contribution under a joint distribution is bounded by its separate closure. Maximizing the joint expectation therefore shows that merging cannot increase the roof. The endpoint identities follow from the definitions. Finally, $L_\Pi\le L$ and Theorem~\ref{thm:common} give the recovery guarantee.
\end{proof}
At fixed marginals, merging preserves the roof value precisely when one joint distribution attains both separate closures. Optimal distributions on disjoint supports can always be coupled. On overlapping supports, they may be incompatible.

A direct block formulation uses $\sum_C2^{|N_C|}$ distribution variables. An alternative exploits the interaction graph, which joins two elements whenever they occur in a common local function. A \emph{tree decomposition} is a tree whose nodes are subsets of graph vertices, called bags. Every vertex belongs to a bag, every edge has both endpoints in one bag, and the bags containing any fixed vertex form a connected subtree. The last condition is the \emph{running-intersection property}. The width is the maximum bag size minus one. The intersection of adjacent bags is their \emph{separator}. Separator consistency equates the full joint probability tables on adjacent-bag intersections.
\begin{proposition}[Junction-tree formulation]\label{prop:junction}
Given a width-$\tau_C$ tree decomposition $T_C$ of the interaction graph of block $C$, its exact closure has a linear formulation with $O(|T_C|2^{\tau_C+1})$ bag-probability variables and separator-consistency constraints. Thus fixed supplied widths yield polynomial-size grouped roof formulations.
\end{proposition}
\begin{proof}
Each local ground set is a clique, so it lies in some bag. Assign its objective to one such bag. For each bag, introduce a joint probability table, normalize it, prescribe its singleton marginals, and require adjacent bags to agree on their full intersection tables. The running-intersection property lets us glue these distributions successively into a global law. Conditional laws on zero-probability separator states can be chosen arbitrarily, since they carry no mass. Conversely, every global law induces consistent bag tables. The formulation is therefore exact, and counting table entries gives the size bound \citep{WainwrightJordan2008}.
\end{proof}

\subsection{Finite recovery games and universal certification}\label{sec:game}
Write $\Delta(J)=\{w\in\R_{\ge0}^J:\sum_{r\in J}w_r=1\}$ for the probability simplex on a finite nonempty set $J$. A \emph{realizable test row} consists of a nonnegative component $g$, marginals $x$, a positive roof value $u$, and data proving that value. For a local closure these can be an optimal coupling and a tight affine majorant. Write $G_r$ for row $r$'s continuous component. Given finitely many candidate maps, set $A_{rs}=G_r(R_s(x^r))/u_r$. The candidate probabilities solve a finite zero-sum game:
\begin{proposition}[Finite map--adversary duality]\label{prop:finite_game}
For a finite nonempty row set $J$ and $k\ge1$ candidate maps,
\[
 \max_{p\in\Delta([k])}\min_{r\in J}(Ap)_r
 =\min_{w\in\Delta(J)}\max_{s\in[k]}(A^\top w)_s.
\]
For common-coordinatewise maps, the dual mixture is realizable on disjoint copies of the rows: weighting row $r$ by $w_r/u_r$ gives total additive roof one and branch values $(A^\top w)_s$.
\end{proposition}
\begin{proof}
The equality is finite linear-programming duality. Independent copies have additive roofs and objectives, and the prescribed weights normalize their roof contributions.
\end{proof}
A finite game certifies only the rows it contains. To obtain a universal guarantee, we need an analytical argument covering the whole domain or exact separation over the generators and marginals. Reweighting fixed maps is a linear program, but choosing their knot values is generally nonconvex. Appendix~\ref{app:arity_four} describes the certificates used for arity four. Section~\ref{sec:retention} gives a minimax formulation for recovery at fixed marginals.

\subsection{A finite-dimensional dual-cone characterization}\label{sec:dual_cone}
Fix a local ground set $A$. Regard a set function as its vector in $\R^{2^A}$. Write $e_S$ for a coordinate unit vector and, for $i<j$, $T\subseteq A\setminus\{i,j\}$, define
\[
 s_{ij,T}=e_{T\cup\{i\}}+e_{T\cup\{j\}}-e_T-e_{T\cup\{i,j\}}.
\]
The cone $\mathcal S_A$ of nonnegative submodular functions is defined by $e_S^\top g\ge0$ and $s_{ij,T}^\top g\ge0$. The latter elementary inequalities are equivalent to submodularity: they compare one-element marginals at adjacent sets, and iteration gives diminishing returns at every pair of nested sets. This is a \emph{polyhedral cone}, a cone defined by finitely many linear inequalities. Its \emph{dual cone} is $\mathcal S_A^*=\{v:v^\top g\ge0\text{ for all }g\in\mathcal S_A\}$. Let $\nu_y$ be the probability vector of independent sampling with marginals $y$. Extreme rays generate any pointed polyhedral cone, where pointed means that it contains no line.

\begin{proposition}[Common-map dual certificate]\label{prop:dual_cone}
Fix $\rho\ge0$ and a finite portfolio $\{(p_s,R_s)\}$, where $R_s$ maps $[0,1]^A$ into itself. The portfolio satisfies
\[
 \sum_sp_sG_g(R_s(x))\ge\rho g^+(x)
 \quad(g\in\mathcal S_A,\ x\in[0,1]^A)
\]
if and only if, for every distribution $\mu$ on $2^A$, with marginals $x(\mu)$,
\begin{equation}\label{eq:dual_cone}
 \sum_sp_s\nu_{R_s(x(\mu))}-\rho\mu
 \in\cone\{e_S,s_{ij,T}:S,i,j,T\}.
\end{equation}
\end{proposition}
\begin{proof}
The finitely generated cone on the right is $\mathcal S_A^*$ by polyhedral duality. Pairing \eqref{eq:dual_cone} with any $g\in\mathcal S_A$ and maximizing over couplings with fixed marginals proves sufficiency. Conversely, the stated inequality dominates $\rho\mu^\top g$ for every such $g$, so the left vector belongs to the dual cone.
\end{proof}
The dual-cone characterization avoids both enumerating extreme rays and choosing roof facets.

\subsection{Common retained-coordinate sets}\label{sec:retention}
Here $\odot$ denotes coordinatewise multiplication. For a fixed $x$, let $x\odot\ind_M$ retain coordinates in $M\subseteq N$ and set the others to zero.
\begin{lemma}[Multilinear lower-box vertices]\label{lem:vertex}
For every multilinear $F$,
\[
 \max_{0\le y\le x}F(y)=\max_{M\subseteq N}F(x\odot\ind_M).
\]
\end{lemma}
\begin{proof}
With the other coordinates fixed, the objective is affine in each coordinate. Move that coordinate to a maximizing endpoint and repeat for all coordinates.
\end{proof}
We use multilinearity throughout the rest of this subsection. The continuous polynomial in Example~\ref{ex:self_atom} shows why the endpoint argument can fail without it.

Fix permitted supports $E$ and pointed rational polyhedral cones of nonnegative submodular functions on each support. Let $V_E$ contain their primitive ray generators. At fixed $x$, retain rows $j=(E,r)$ with $r^+(x_E)>0$, and define
\[
 J_x=\{(E,r):r\in V_E,\ r^+(x_E)>0\},\qquad
 A_{jM}=\frac{G_r(x_E\odot\ind_{M\cap E})}{r^+(x_E)}.
\]
\begin{theorem}[Optimal common retention game]\label{thm:retention_game}
If $J_x\ne\varnothing$, the largest guarantee supplied by one random retained-coordinate set simultaneously for all allowed local conic combinations equals
\begin{equation}\label{eq:retention_game}
 \rho^*(x)=\max_{\lambda\in\Delta(2^N)}\min_{j\in J_x}\sum_M\lambda_M A_{jM}
 =\min_{\theta\in\Delta(J_x)}\max_{M\subseteq N}\sum_j\theta_j A_{jM}.
\end{equation}
An optimal $\lambda$ satisfies $\E_M F(x\odot\ind_M)\ge\rho^*(x)L(x)$ for every allowed instance. A dual optimum is realized by assigning coefficient $\theta_j/r^+(x_E)$ to the corresponding ray on support $E$.
\end{theorem}
\begin{proof}
If one distribution pays every ray, it pays every allowed conic combination: the multilinear extension is linear, while concave closure is subadditive under nonnegative addition. Rows with zero roof value need only a nonnegative contribution, so they can be omitted. Conversely, each positive-roof ray is an allowed local function, making the minimum over rows necessary. Finite linear-programming duality gives the second expression. The stated coefficients make the dual instance's additive roof equal to one.
\end{proof}
The game may have exponentially many columns. Nevertheless, with rational $x$, fixed arity, explicit local rays, and a supplied bounded-width interaction-tree decomposition, dual separation reduces to maximizing a sum of binary local functions. Tree dynamic programming solves this problem and, under these assumptions, gives polynomial-time optimization of \eqref{eq:retention_game}.

For one supplied instance, the optimal downward point is found by tabulating
$\vartheta_b(z_{N_b})=w_bG_b(x_{N_b}\odot z_{N_b})$ for Boolean $z$, and maximizing their sum on the interaction tree. With fixed arity and a supplied width-$\tau$ decomposition $T$, this takes
\[
 O\left(\sum_b|N_b|2^{|N_b|}+(|B|+|T|)(\tau+1)2^{\tau+1}\right)
\]
arithmetic operations. The second term accounts for evaluating each local table on its assigned bag and combining the resulting tables before tree dynamic programming. By Lemma~\ref{lem:vertex}, this binary optimization solves the lower-box problem exactly.

\subsection{Nested instance-specific certificates}
For a sequence $\Pi_0\preceq\cdots\preceq\Pi_m$, set
\[
 U_k=\max_{x\in\K}L_{\Pi_k}(x),\qquad
 x^k\in\argmax_{x\in\K}L_{\Pi_k}(x),\qquad
 V_k=\max_{j\le k}F(y^j),\quad 0\le y^j\le x^j.
\]
\begin{proposition}[Best-so-far guarantees]\label{prop:anytime}
The sequence $U_k$ is nonincreasing and $V_k$ is nondecreasing. Whenever $U_k>0$, the value $V_k/U_k$ certifies the retained solution's approximation ratio relative to $\max_\K F$, and these certificates are nondecreasing. If the final partition has one block, $U_m=\max_\K f^+$, not necessarily $\max_\K F$.
\end{proposition}
\begin{proof}
Theorem~\ref{thm:grouping} makes the upper bounds decrease, while retaining earlier recovered points makes the feasible values increase. Since $\max_\K F\le U_k$, division gives the certificate.
\end{proof}
With inexact optimization, the denominator must remain a certified upper bound. A feasible dual value, or a computed roof value plus a valid optimization-error bound, suffices. A primal roof value alone does not bound the optimum from above. Keeping the smallest certified upper bound and the largest feasible objective value preserves monotonicity of the certificate. This retention matters because a tighter roof can have a different optimizer and need not produce a better recovered solution.

\section{Discrete optimization and correlation consequences}\label{sec:discrete}
For multilinear objectives, rounding converts a recovered point into a feasible set. The same recovery inequalities also bound the ordinary, downward, and max-subset correlation ratios.

\subsection{The rounding interface}\label{sec:rounding}
This section assumes that $F$ is the multilinear extension of a nonnegative submodular set function $f$, and $\K$ contains every feasible incidence vector for a down-closed family $\I$.
\begin{proposition}[Transfer by rounding]\label{prop:rounding}
Let $\rho,\eta,\delta\ge0$. Suppose recovery gives $y\in\K$ with $F(y)\ge\rho(\max_\K F-\delta)$, and a rounding procedure returns $S\in\I$ with $\E f(S)\ge\eta F(y)$. Then
\[
 \E f(S)\ge\eta\rho\left(\max_{T\in\I}f(T)-\delta\right).
\]
\end{proposition}
\begin{proof}
Every $\ind_T$, $T\in\I$, belongs to $\K$, so $\max_\K F\ge\max_{T\in\I}f(T)$.
\end{proof}
A \emph{matroid} is a nonempty down-closed family satisfying the exchange axiom: if $I,J$ belong to the family and $|I|<|J|$, there is an $e\in J\setminus I$ for which $I\cup\{e\}$ also belongs to the family. Its independence polytope is the convex hull of its incidence vectors. Its maximal independent sets are called bases; they all have the same size, called the rank. Their convex hull is its base polytope. The uniform matroid of rank $k$ consists of the sets of size at most $k$.

For the cube, successively choosing a maximizing endpoint of each affine coordinate section gives $\eta=1$. For a uniform-matroid independence polytope, or more generally a matroid independence polytope, deterministic pipage rounding also gives $\eta=1$ \citep{AgeevSviridenko2004,CalinescuChekuriPalVondrak2011}. A zero-value dummy is an added element whose inclusion leaves $f$ unchanged. Adjoining enough such elements embeds the fractional point into an enlarged matroid base polytope, by declaring a set feasible when its original elements are independent and its total cardinality, including dummies, is at most the original rank. Along a swap direction $e_i-e_j$, the multilinear extension is convex because its second derivative is $-2\partial_{ij}F\ge0$. One endpoint therefore does not decrease its value; repeated exchanges and deletion of the dummies yield an independent set. Explicit polynomial or local-function representations allow deterministic value comparisons. Swap rounding gives an alternative with no loss in expectation \citep{ChekuriVondrakZenklusen2010}.

Independence and at-most-cardinality constraints are down-closed, as recovery requires. Exact-cardinality and base constraints need a separate recovery argument. Rounding also requires additional structure; a separation oracle for an arbitrary down-closed convex set is insufficient.

For another independence system, put $P_{\I}=\conv\{\ind_S:S\in\I\}$. A \emph{monotone $(b,c)$-balanced contention-resolution scheme}, with $0<b,c\le1$, takes $x\in bP_{\I}$ and an independent sample $R_x$, and returns a feasible subset. Each sampled element is retained with conditional probability at least $c$, and its retention probability cannot increase as the input sample grows. Given $y\in P_{\I}$, use $x=by$. For a nonmonotone objective, prune the output in a fixed order, retaining an element only if its marginal to the retained set is positive. This gives $\E f(S)\ge cF(by)\ge bcF(y)$ \citep{ChekuriVondrakZenklusen2014}. The second inequality uses concavity of $t\mapsto F(ty)$ and $F(0)\ge0$, so Proposition~\ref{prop:rounding} applies with $\eta=bc$.

\begin{corollary}[Lossless discrete specializations]\label{cor:discrete}
Under no side constraint, an at-most-$k$ constraint, or a matroid-independence constraint, the following fractional factors transfer without further loss, subject to the representation and access assumptions already stated: $1/2$ for $\CQ_2$; $8/17$ for $\CQ_3$, supplied $\CR_{12}$ representations, and supplied $\CL_4$ decompositions; $\gamma_*$ for supplied arbitrary-finite-arity $\CR_\infty$ representations; and $\min_b\rho_{|N_b|}$ for supplied concave-cardinality decompositions. Globally monotone members of $\CL_d$ have factor $\gamma_d$.
\end{corollary}
\begin{proof}
Apply Proposition~\ref{prop:rounding} with $\eta=1$ to the corresponding roof-recovery results.
\end{proof}
These guarantees allow an arbitrarily small additive loss. Generic unconstrained algorithms attain $1/2$ \citep{BuchbinderFeldman2018}; Corollary~\ref{cor:discrete} gives the stated factors under matroid-independence constraints as well.

\subsection{Directed cuts and submodular hypergraph splitting}
For the loopless nonnegative weighted directed-cut model of Section~\ref{sec:quadratic_main}, Corollary~\ref{cor:discrete} gives factor $1/2$ under an at-most-$k$ or matroid-independence constraint on the source side.

A generalized submodular hypergraph splitting objective has the form
\[
 f(S)=\sum_{e\in\mathcal E}w_e\delta_e(S\cap e),\quad
 w_e\ge0,\quad\delta_e:2^e\to\R_{\ge0}\text{ submodular},\quad
 \delta_e(\varnothing)=\delta_e(e)=0.
\]
Undirected splitting also requires $\delta_e(T)=\delta_e(e\setminus T)$. We maximize the objective in this standard splitting-function model \citep{VeldtBensonKleinberg2022}. With supplied local tables and edges of size at most four, Corollary~\ref{cor:arity_four} gives $8/17$ under arbitrary overlap, followed by lossless matroid rounding. The model includes local functions beyond rooted directed-hyperedge indicators.

For concave-cardinality splitting functions, the factor is $\min_e\rho_{|e|}$ at arbitrary encoded edge sizes. For symmetric concave-cardinality splitting with all edges of size at least three, \eqref{eq:symmetric_cardinality} gives at least $5/8$.

\subsection{Three different correlation ratios}\label{sec:correlation}
For a nonzero nonnegative $f$, compare three benchmarks at points where $f^+(x)>0$: independent evaluation, the best independent evaluation below the marginals, and the best subset of an independent sample. Define
\[
 \CG(f)=\inf_x\frac{F(x)}{f^+(x)},\qquad
 \DCG(f)=\inf_x\frac{\max_{0\le y\le x}F(y)}{f^+(x)},
\]
\[
 f_{\max}(S)=\max_{T\subseteq S}f(T),\qquad
 \CG_{\max}(f)=\inf_x\frac{\E f_{\max}(R_x)}{f^+(x)}.
\]
The infima exclude zero denominators. The standard and max-subset ratios serve different purposes in stochastic optimization and prophet inequalities \citep{AgrawalDingSaberiYe2012,Yan2011,RubinsteinSingla2017,ChekuriLivanos2024}.
\begin{proposition}[Correlation consequence]\label{prop:correlation}
For every $x$,
\[
 F(x)\le\max_{0\le y\le x}F(y)\le\E f_{\max}(R_x).
\]
Consequently $\CG(f)\le\DCG(f)\le\CG_{\max}(f)$. A downward recovery inequality paying $\rho L(x)$ with $L\ge f^+$ implies $\DCG(f),\CG_{\max}(f)\ge\rho$.
\end{proposition}
\begin{proof}
For $y\le x$, sample $R_x$ and retain element $i$ independently with probability $y_i/x_i$, using zero when $x_i=0$. The retained set has law $R_y$ and is contained in $R_x$, so $F(y)\le\E f_{\max}(R_x)$. Maximizing and then taking infima proves the claims.
\end{proof}
For monotone $f$, all three ratios coincide. Theorem~\ref{thm:monotone} gives the exact class-wide value $\gamma_d$ for monotone $\CL_d$, $\CR_d$, and $\CA_d$: the sharp coverage witness belongs to all three classes.

For nonzero nonnegative quadratic submodular functions, the class-wide downward ratio is exactly $1/2$. The lower bound follows from the quadratic roof theorem and $L\ge f^+$. For sharpness, take $f(S)=\ind\{|S|=1\}$ on two elements, with extension $F(y)=y_1+y_2-2y_1y_2$. At $x=(1/2,1/2)$, the equal mixture of the two singletons gives $f^+(x)=1$. Both partial derivatives of $F$ are nonnegative on $[0,x]$, so its maximum there is $F(x)=1/2$.

The preceding arguments give a downward-correlation lower bound of $8/17$ for cubic functions, supplied rooted representations through degree twelve, and arbitrary overlapping arity-four local functions. For arbitrary-finite-degree rooted representations, the bound is $\gamma_*$. The examples using additive roofs do not prove exactness when the denominator is $f^+$. Theorem~\ref{thm:cardinality}, however, gives the exact downward ratio $\rho_d$ for nonnegative concave-cardinality functions on a fixed $d$-element ground set. Sums of these local functions inherit the minimum component factor. Section~\ref{sec:online} establishes online guarantees for globally monotone local sums and for nonmonotone objectives admitting the cut, unary, and coverage representation \eqref{eq:online_class}.

\Needspace{14\baselineskip}
\part{Online preservation}\label{part:online}
\section{Online preservation and a fractional prophet interface}\label{sec:online}
A selectable greedy online contention-resolution scheme can preserve the value of nonnegative sums of directed-hypergraph cuts, unaries, and coverage terms. The proof uses positive association of decreasing events. Combining this preservation result with the benchmark on the full outcome universe in Appendix~\ref{app:outcomes} gives a prophet guarantee.

\subsection{Greedy online contention resolution}
Let $P_{\I}=\conv\{\ind_S:S\in\I\}$ for a nonempty down-closed family $\I$, and fix $b,c\in[0,1]$. An \emph{online contention-resolution scheme (OCRS)} must accept or reject an active element when it arrives, without changing earlier decisions. For independent activation probabilities $z$, a randomized greedy OCRS first samples a nonempty down-closed family $\mathcal F_z\subseteq\I$, independently of the active set $R_z$. Thus $\varnothing\in\mathcal F_z$ and the accepted set starts at $\varnothing$. It accepts an arriving active element $e$ exactly when $J\cup\{e\}\in\mathcal F_z$, where $J$ is the set accepted so far. Define
\[
 Q_e(R_z,\mathcal F_z)=
 \{J\cup\{e\}\in\mathcal F_z\text{ for every }J\subseteq R_z
                                      \text{ with }J\in\mathcal F_z\}.
\]
Adding other active elements can only destroy $Q_e$; toggling the activation of $e$ leaves it unchanged. If $e$ is active and $Q_e$ holds, every arrival order leads to its acceptance. The scheme is $c$-selectable if $\Pr(Q_e)\ge c$ for every $e$, and $(b,c)$-selectable if this holds for every $z\in bP_{\I}$, following \citet{FeldmanSvenssonZenklusen2021}. All results below allow an \emph{almighty adversary}, which observes the active set and the scheme's random choices before choosing the arrival order.

\begin{corollary}[Monotone local objectives]\label{cor:online_monotone}
If $f\in\CL_d$ is globally monotone, $x$ is a $\delta$-accurate maximizer of its local roof over $P_{\I}$, and a $(b,c)$-selectable greedy OCRS is available, then its output on $R_{bx}$ satisfies
\[
 \E f(S_{\rm out})\ge bc\gamma_d\left(\max_{S\in\I}f(S)-\delta\right).
\]
\end{corollary}
\begin{proof}
The monotone OCRS theorem gives $\E f(S_{\rm out})\ge cF(bx)$ \citep{FeldmanSvenssonZenklusen2021}. Apply Theorem~\ref{thm:monotone}, followed by concavity and nonnegativity at zero of $L$, to obtain $F(bx)\ge\gamma_dL(bx)\ge b\gamma_dL(x)$.
\end{proof}
Monotonicity is required only of the sum. For a matroid, discard loops---elements belonging to no feasible independent set---and use $c=1-b$. Taking $b=1/2$ gives $\gamma_d/4$.

\subsection{Directed-hypergraph cut objectives}
For subsets $E,H\subseteq N$, write
\[
 \cov_E(S)=\ind\{S\cap E\ne\varnothing\},\qquad
 \cut_{E,H}(S)=\ind\{S\cap E\ne\varnothing,\ H\nsubseteq S\}.
\]
The rooted functions are $u_{T,i}=\cut_{\{i\},T}$ and $v_{T,i}=\cut_{T,\{i\}}$. Directed-hypergraph cut functions allow arbitrary $E,H$ and nonnegative weights \citep{Schrijver2003}. Each cut is nonnegative and submodular: for disjoint $E,H$, its multilinear extension is
\[
 \left(1-\prod_{i\in E}(1-x_i)\right)
 \left(1-\prod_{j\in H}x_j\right);
\]
if $E\cap H\ne\varnothing$, it is
\[
 1-\prod_{i\in E}(1-x_i)-\prod_{j\in H}x_j.
\]
Every mixed second derivative is nonpositive. Empty sides give the zero cut. 
\begin{theorem}[Online preservation]\label{thm:online_preservation}
Suppose $f$ has the representation
\begin{equation}\label{eq:online_class}
 f(S)=a+\sum_ip_i\ind\{i\in S\}+\sum_iq_i\ind\{i\notin S\}
       +\sum_r\mu_r\cov_{E_r}(S)+\sum_r\nu_r\cut_{E_r,H_r}(S),
\end{equation}
with all coefficients nonnegative. A $c$-selectable randomized greedy OCRS applied to $R_z$ satisfies
\[
 \E f(S_{\rm out})\ge cF(z)
\]
for every arrival order, including one selected by an almighty adversary. In particular, the statement holds for every function in $\CR_\infty$.
\end{theorem}
\begin{proof}
Consider one cut function and order $E=(e_1,\ldots,e_m)$. The disjoint events
\[
 C_k=\{e_k\in R_z,\ e_1,\ldots,e_{k-1}\notin R_z,\ H\nsubseteq R_z\}
\]
partition the event that the cut is one. Terms with $z_{e_k}=0$ have zero probability and may be omitted. Conditional on $e_k$ being active, its remaining condition is
$C_k'=\{e_1,\ldots,e_{k-1}\notin R_z,\ H\setminus\{e_k\}\nsubseteq R_z\}$.
Conditional on $\mathcal F_z$, both $C_k'$ and $Q_{e_k}$ decrease as the remaining independent activations grow, so Harris's inequality gives positive association \citep{Harris1960}. The event $C_k'$ is independent of $\mathcal F_z$, and $Q_{e_k}$ is independent of the activation of $e_k$. Averaging therefore gives
\[
 \Pr(C_k\cap Q_{e_k})\ge\Pr(C_k)\Pr(Q_{e_k})\ge c\Pr(C_k).
\]
On $C_k\cap Q_{e_k}$, the algorithm accepts an element of $E$, while an inactive element of $H$ remains absent. The output cut is therefore one. Summing proves preservation of its expectation. Omitting the missing-$H$ condition gives the coverage result. Positive literals are preserved with probability at least $cz_i$; complemented literals satisfy $\Pr(i\notin S_{\rm out})\ge1-z_i\ge c(1-z_i)$, and constants lose no value. Sum these bounds with the nonnegative coefficients.
\end{proof}
For a general nonnegative nonmonotone submodular function, the standard theorem applies independent probability-$1/2$ subsampling to the greedy output and guarantees $cF/4$ under this adversary \citep[Theorem~1.10]{FeldmanSvenssonZenklusen2021}. The representation in \eqref{eq:online_class} gives $cF$ for the unmodified greedy output.

\subsection{The additional prophet benchmark}
The preservation theorem compares the output with $F(z)$ at the activation marginals $z$. To obtain a prophet guarantee, we also need a fractional benchmark on the full outcome universe. Let independent items have disjoint finite outcome sets $U_i$, probabilities $p_e$, and possible null outcomes; let $\I$ constrain the selected item indices. Assume the valuation on all of $U=\bigsqcup_iU_i$ has representation \eqref{eq:online_class}. Define
\[
 P^{\rm out}=\left\{x\ge0:x_e\le p_e,\ \left(\sum_{e\in U_i}x_e\right)_i\in P_{\I}\right\}.
\]
If a $(b,c)$-selectable greedy item-space OCRS is efficient and the full outcome valuation has an efficiently optimizable roof $L\ge f^+$ with a common downward recovery factor $\rho$, then a feasible $\delta$-accurate roof optimizer gives
\[
 \E f(S_{\rm alg})\ge bc\rho(\OPT_{\rm prop}-\delta).
\]
Here $\OPT_{\rm prop}$ is the expected value of a feasible selector that sees all outcomes. Proposition~\ref{prop:outcome_transfer} and Corollary~\ref{cor:prophet_interface} in Appendix~\ref{app:outcomes} prove the transfer, including null outcomes and the comparison of one-outcome-per-item sampling with independent outcome sampling. The representation must hold on all of $U$. Arbitrary arity-four local functions need this additional hypothesis for the online transfer.

\Needspace{14\baselineskip}
\part{Explicit-input computational hardness}\label{part:hardness}
\section{Hardness methodology}\label{hard:sec:methodology}
Section~\ref{sec:frontiers} bounds what specified roofs and recovery rules can achieve. Here we consider computational hardness for explicit quadratic instances. The reductions use terminal-free directed cuts with nonnegative arc weights; Section~\ref{sec:quadratic_main} gives the more general fixed-terminal representation.

\subsection{The explicit cardinality-constrained model}\label{hard:sec:representation}
Let $f$ be a globally nonnegative quadratic submodular function as in \eqref{hard:eq:problem}, and let $F$ denote its multilinear extension.
We consider the promised explicit-input problem
\begin{equation}\label{hard:eq:qsub-cardinality}
 \max\{f(S):S\subseteq[n],\ |S|\le k\},
 \qquad f(S)\ge0\quad\text{for every }S\subseteq[n].
\end{equation}
The rational coefficients $(c,a,b)$ and the integer $k$ are encoded in binary.
The nonnegativity promise applies to the entire Boolean cube, including
sets that exceed the budget.

For $0<\alpha\le1$, an $\alpha$-approximation returns a feasible $S$ with
$f(S)\ge\alpha\OPT$ on every promised instance, where $\OPT$ denotes the
maximum in \eqref{hard:eq:qsub-cardinality}.  Thus a hardness factor $r$ is an
upper bound on the attainable approximation ratio: ``hard within
$r+\eps$'' means that no polynomial-time $(r+\eps)$-approximation exists
under the stated complexity assumption. All computational-hardness statements here concern deterministic polynomial-time algorithms.

The corresponding continuous feasible region is
\[
 P_k=\left\{x\in[0,1]^n:\sum_i x_i\le k\right\},
 \qquad k\in\{0,\ldots,n\}.
\]
The lossless rounding argument in Section~\ref{sec:discrete} gives
\[
 \max_{x\in P_k}F(x)=\max_{|S|\le k}f(S).
\]
Incidence vectors give one inequality; rounding any fractional point without
decreasing its value gives the other. A continuous approximation on this
explicit multilinear subclass therefore yields the same discrete factor.
In the normalized reductions below, the additive error can be chosen smaller
than the fixed gap between the yes-case and no-case optimum bounds.

\subsection{Test functions and proof mechanism}
A cardinality bound constrains the total selected vertex mass, but different parts of the reduction may have different selection probabilities. A global affine majorant will bound the cut value while allowing this variation.

A \emph{long-code cloud} is a copy of $\{0,1\}^R$, with one coordinate for
each label in $[R]=\{1,\ldots,R\}$. Its type specifies the product measure,
and a Boolean indicator $g$ describes the selected part of the cloud.
A \emph{dictator} selects a coordinate: $g(x)=x_i$. A dictatorship test
distinguishes these intended indicators from the cloud functions arising
in the reduction's no case. The directed test accepts when its sampled tail
is selected and its sampled head is not. Its \emph{completeness} is the acceptance
probability guaranteed by the intended dictators; its \emph{soundness}
bounds acceptance in the no case. After replication, we also use fractional
cloud indicators in $[0,1]$.

\subsection{Retaining directed asymmetry}

If both endpoints have the same distribution and use the same Boolean cloud function, forward and reverse crossings have equal probability.

\begin{proposition}[Equal-marginal collapse]\label{hard:prop:equal-marginal}
If random variables $X,Y$ have identical distributions and $g$ is Boolean,
then
\[
 \Prb[g(X)=1,g(Y)=0]
 =\Prb[g(X)=0,g(Y)=1]
 =\tfrac12\Prb[g(X)\ne g(Y)].
\]
For a one-type test using the same cloud function and bias at both ends,
symmetrization multiplies both completeness and soundness by two. The
soundness-to-completeness ratio is therefore unchanged whenever completeness is positive.
\end{proposition}

\begin{proof}
Since $\E[g(X)]=\E[g(Y)]$, subtracting the two crossing probabilities gives zero.
Their sum is the disagreement probability.
\end{proof}

With several types, the imbalance between forward and reverse acceptance depends on the cloud means and net flow through each type.

\begin{proposition}[Type-flow identity]\label{hard:prop:flow}
Let $W_{st}\ge0$ be directed type-pair weights, let $g_s$ be a Boolean cloud
function of type $s$, and set $m_s=\E[g_s]$.  For every used pair $(s,t)$,
assume that $X$ has the prescribed type-$s$ marginal and $Y$ has the
prescribed type-$t$ marginal, so
$\E[g_s(X)]=m_s$ and $\E[g_t(Y)]=m_t$.  If
\begin{align*}
 A&=\sum_{s,t}W_{st}\E[g_s(X)(1-g_t(Y))],\\
 B&=\sum_{s,t}W_{st}\E[(1-g_s(X))g_t(Y)],
\end{align*}
then
\begin{equation}\label{hard:eq:flow}
 A-B=\sum_s(\out_s-\inw_s)m_s,
 \quad
 \out_s:=\sum_tW_{st},\quad \inw_s:=\sum_tW_{ts}.
\end{equation}
In particular, $A=B$ for every family of cloud functions when $W$ is
Eulerian, meaning $\out_s=\inw_s$ for every type $s$.
\end{proposition}

\begin{proof}
For each ordered pair, the difference between forward and reverse acceptance
is $m_s-m_t$. Multiplying by the pair weights and collecting the coefficient
of each mean gives \eqref{hard:eq:flow}.
\end{proof}

On dictators, the cloud means equal the type biases. Unequal biases and a
non-Eulerian type matrix can therefore produce directed imbalance.

\subsection{Directed Gaussian profiles}

The soundness argument uses Gaussian crossing probabilities to bound the cut
test.

Let $\Phi$ be the standard normal distribution function.  For
$\rho\in(-1,1)$ and $a,b\in[0,1]$, define
\begin{align}
 \Gamma_\rho(a,b)
 &:=\Prb[Z\le\Phi^{-1}(a),\ Z'\le\Phi^{-1}(b)],\label{hard:eq:gamma}\\
 D_\rho(a,b)
 &:=a-\Gamma_\rho(a,b)
 =\Gamma_{-\rho}(a,1-b),\label{hard:eq:D}
\end{align}
where $(Z,Z')$ is a standard bivariate normal vector with correlation
$\rho$.  Endpoint values are defined by continuity.  The quantity
$D_\rho(a,b)$ is the Gaussian analogue of the directed event
``first output $1$, second output~$0$.''

A \emph{finite-type directed test} consists of:
\begin{itemize}[leftmargin=2em]
 \item a finite type set $\mathcal T$ and masses $\pi_s>0$ with
 $\sum_s\pi_s=1$;
 \item dictator biases $q_s\in(0,1)$;
 \item normalized directed weights $W_{st}\ge0$ with
 $\sum_{s,t}W_{st}=1$;
 \item for each used pair $(s,t)$, a connected probability distribution
 $\nu_{st}$ on $\{0,1\}^2$ with Bernoulli marginals $q_s,q_t$ and negative
 Pearson correlation $\rho_{st}$.
\end{itemize}
The bipartite support graph of $\nu_{st}$ has two copies of $\{0,1\}$,
with an edge from $x$ in the first copy to $y$ in the second whenever
$\nu_{st}(x,y)>0$. Connectedness of this graph is the irreducibility
condition used in correlated-space invariance principles. Define the dictator density, completeness, and
Gaussian soundness profile by
\begin{equation}\label{hard:eq:test-quantities}
 \mu:=\sum_s\pi_sq_s,
 \qquad
 C:=\sum_{s,t}W_{st}\nu_{st}(1,0),
 \qquad
 \Psi(z):=\sum_{s,t}W_{st}D_{\rho_{st}}(z_s,z_t).
\end{equation}

\begin{lemma}[Binary maximal correlation]\label{hard:lem:binary-correlation}
Let $(X,Y)$ have a connected distribution on $\{0,1\}^2$, nondegenerate
Bernoulli marginals, and Pearson correlation $\rho$.  Its maximal
correlation
\[
 \sup\{\E[u(X)v(Y)]:\E u=\E v=0,\ \E u^2=\E v^2=1\}
\]
equals $|\rho|$.  After replacing $Y$ by $1-Y$, the Pearson correlation is
$-\rho$ and the maximal correlation remains $|\rho|$.
\end{lemma}

\begin{proof}
The centered subspace of $L^2(X)$ is one-dimensional, spanned by the
standardized variable $(X-\E X)/\sqrt{\operatorname{Var}X}$; the same holds
for $Y$.  Every admissible $u,v$ is therefore one of these standardized
variables up to sign.  The supremum is the absolute Pearson correlation.
Complementation negates the standardized second variable.
\end{proof}

For the $q_s$-biased product measure, use the orthonormal basis
\[
 \chi_J^{(s)}(x)=\prod_{i\in J}
 \frac{x_i-q_s}{\sqrt{q_s(1-q_s)}},\qquad J\subseteq[R].
\]
If $g_s=\sum_J\widehat g_s(J)\chi_J^{(s)}$, its
\emph{degree-at-most-$d$ influence} at coordinate $i$ is
\[
 \Inf_i^{\le d}(g_s)
 =\sum_{\substack{J\ni i\\|J|\le d}}\widehat g_s(J)^2.
\]
A coordinate has large low-degree influence when the coefficients involving
it account for a substantial share of the low-degree Fourier mass.

\begin{lemma}[Finite-type invariance]\label{hard:lem:finite-type-invariance}
Fix a finite type system $(\mathcal T,W,\nu)$ whose used binary pair spaces
are connected.  For every $0<\eta<1$, there are an integer $d\ge1$ and
constant $0<\tau\le1$ with the following property.  Let
$g_s:\{0,1\}^R\to[0,1]$ have mean $m_s$ under the $q_s$-biased product
measure.  Write $\Inf_i^{\le d}(g_s)$ for the degree-at-most-$d$ influence of
coordinate $i$ in the corresponding biased Fourier basis.  Then either
\begin{equation}\label{hard:eq:finite-type-invariance}
 \sum_{s,t}W_{st}
 \E_{\nu_{st}^{\otimes R}}[g_s(X)(1-g_t(Y))]
 \le \Psi(m)+\eta,
\end{equation}
or some used pair $(s,t)$ and coordinate $i$ satisfy
\[
 \Inf_i^{\le d}(g_s)\ge\tau,
 \qquad
 \Inf_i^{\le d}(g_t)\ge\tau.
\]
\end{lemma}

\begin{proof}[Proof outline]
Complement the second binary space and its output, giving maximal correlation $-\rho_{st}$. Mossel's two-function bound then gives $D_{\rho_{st}}(m_s,m_t)$, up to the prescribed error, unless a coordinate has large low-degree influence in both endpoint functions. Choose common influence parameters for the finite set of used pairs and sum with the normalized weights. Appendix~\ref{app:invariance_proof} gives the details and the precise source dependence.
\end{proof}

A bipartite \emph{Unique-Games instance} assigns a permutation of
$[R]$ to each edge. A labeling assigns one label to each vertex; its value
is the fraction of edges whose permutation constraint it satisfies, and
$\operatorname{val}(\mathcal U)$ is the largest such fraction. We use the
regular bipartite form of the \emph{Unique Games Conjecture}: for every
$\zeta,\kappa>0$, some fixed alphabet size $R$ makes it NP-hard to
distinguish instances of value at least $1-\zeta$ from those of value at
most $\kappa$ \citep{Khot2002,AustrinStankovic2019}. Regularity means that
degrees are constant on each side, so uniform edge sampling has uniform
vertex marginals on each side.

The decoding step uses averages over neighboring clouds. Let $\mathcal U$
be a regular bipartite Unique-Games instance with bipartition $(U,V)$, label
set $[R]$, and edge permutations $\sigma_{uv}$. In the reduction, vertices in
$U$ are the \emph{outer vertices}, and clouds are attached to vertices in $V$.
We use the convention that
$(u,v)$ is satisfied when
$\ell(v)=\sigma_{uv}(\ell(u))$.  For cloud functions
$g_{v,s}:\{0,1\}^R\to[0,1]$, let $g^u_{v,s}$ denote the coordinate permutation
of $g_{v,s}$ induced by $\sigma_{uv}$, explicitly
$g^u_{v,s}(x)=g_{v,s}((x_{\sigma_{uv}^{-1}(j)})_{j\in[R]})$, and define
\begin{equation}\label{hard:eq:neighbor-average}
 f_{u,s}:=\E_{v\in N(u)}g^u_{v,s},
 \qquad m_{u,s}:=\E[f_{u,s}].
\end{equation}
The vector $m_u=(m_{u,s})_{s\in\mathcal T}$ is the \emph{mean profile} at
outer vertex $u$.

\begin{lemma}[Influence decoding]\label{hard:lem:influence-decoding}
Fix a finite type system and $0<\eta<1$.  There is a constant $\kappa>0$,
independent of $R$ and of the size of $\mathcal U$, such that
\begin{align}
 &\E_{u\in U}\sum_{s,t}W_{st}
 \E_{\nu_{st}^{\otimes R}}
 [f_{u,s}(X)(1-f_{u,t}(Y))]\notag\\
 &\hspace{30mm}>
 \E_{u\in U}[\Psi(m_u)]+\eta
 \quad\Longrightarrow\quad
 \operatorname{val}(\mathcal U)\ge\kappa .
 \label{hard:eq:decoding-implication}
\end{align}
\end{lemma}

\begin{proof}[Proof outline]
If acceptance exceeds the Gaussian profile on average, a positive fraction of outer vertices must have a shared influential coordinate. Convexity passes this influence from each neighbor average to a positive fraction of its incident clouds. The total low-degree influence bound limits the number of candidate labels. Choosing labels uniformly from these lists gives a positive expected fraction of satisfied edges. Appendix~\ref{app:decoding_proof} gives the constants and permutation convention.
\end{proof}

\begin{theorem}[Affine finite-type criterion]\label{hard:thm:affine-criterion}
Suppose the masses, biases, type weights, and pair-distribution atoms are
rational, $C>0$, and $\lambda\ge0$, $\beta\in\R$ satisfy
\begin{equation}\label{hard:eq:affine-majorant}
 \Psi(z)\le\beta+\lambda\pi^\top z
 \qquad\text{for every }z\in[0,1]^{\mathcal T}.
\end{equation}
Assume also that $S:=\beta+\lambda\mu<C$.  For every
$0<\delta<(C-S)/2$, the Unique Games Conjecture implies NP-hardness of
distinguishing normalized weighted Max-DiCut instances with an
at-most-cardinality budget and $|k/n'-\mu|\le\delta$ satisfying
\begin{equation}\label{hard:eq:ugc-gap}
 \OPT\ge C-\delta
 \qquad\text{from those satisfying}\qquad
 \OPT\le\beta+\lambda\mu+\delta.
\end{equation}
Here $n'$ is the number of vertices and the total arc weight is normalized to
one.  Consequently, for every $\eps>0$ there is no polynomial-time
$((\beta+\lambda\mu)/C+\eps)$-approximation on this family.
\end{theorem}

\begin{proof}[Proof outline]
The regular long-code construction gives dictator density $\mu$ and completeness approaching $C$. Influence decoding bounds each no-case cut by the average Gaussian profile; the affine inequality bounds that average through the cardinality budget. Rational replication converts vertex masses to ordinary cardinalities. A cut may separate copies of the same vertex, but the resulting fractional cloud indicators obey the same bilinear soundness bound. Appendix~\ref{app:affine_reduction_proof} gives the construction and accounts for replication, loop removal, and the error terms.
\end{proof}

\subsection{Exact one-budget concavification}
The budget limits the average cost of the mean profiles across outer vertices. It does not require each profile to satisfy the budget separately. The relevant optimization is therefore over distributions of profiles. Let $\mathcal P([0,1]^{\mathcal T})$ be the Borel probability measures on the profile cube and, for $0<\mu<1$, set
\begin{equation}\label{hard:eq:mixing-primal}
 S_\Psi(\mu):=
 \sup_{\xi\in\mathcal P([0,1]^{\mathcal T})}
 \left\{
  \int\Psi(z)\,d\xi(z):
  \int\pi^\top z\,d\xi(z)\le\mu
 \right\}.
\end{equation}

\begin{proposition}[Exact affine dual]\label{hard:prop:affine-dual}
For every continuous $\Psi$ and $0<\mu<1$,
\begin{align}
 S_\Psi(\mu)
 &=\min_{\lambda\ge0}
 \left\{
  \lambda\mu+
  \max_{z\in[0,1]^{\mathcal T}}
   [\Psi(z)-\lambda\pi^\top z]
 \right\}\label{hard:eq:dual-one}\\
 &=\min_{\lambda\ge0,\,\beta}
 \left\{
  \beta+\lambda\mu:
  \Psi(z)\le\beta+\lambda\pi^\top z
  \ \forall z
 \right\}.\label{hard:eq:dual-two}
\end{align}
Thus the optimal affine certificate gives the exact value
$S_\Psi(\mu)$ of the one-budget profile-mixing problem.
\end{proposition}

\begin{proof}
Lagrangian relaxation of the moment constraint gives the right-hand side of
\eqref{hard:eq:dual-one} as an upper bound.  For the reverse inequality, define
\[
 G(t):=\sup_\xi
 \left\{\int\Psi\,d\xi:\int\pi^\top z\,d\xi\le t\right\}.
\]
Compactness and continuity ensure that the supremum is attained. Mixing
feasible measures makes $G$ concave; increasing the budget makes it
nondecreasing.
At the interior point $\mu$, choose a supporting supergradient
$\lambda\ge0$.  A point mass at $z$ is feasible for budget $\pi^\top z$, so
\[
 \Psi(z)-\lambda\pi^\top z
 \le G(\mu)-\lambda\mu.
\]
Taking the maximum over $z$ gives the reverse inequality.  Introducing
$\beta=\max_z[\Psi(z)-\lambda\pi^\top z]$ gives
\eqref{hard:eq:dual-two}.
\end{proof}

The design objective is consequently
\begin{equation}\label{hard:eq:design-objective}
 \frac{S_\Psi(\mu)}{C},
\end{equation}
jointly minimized over type masses, biases, directed weights, admissible
binary pair distributions, and the exact dual slope.

\subsection{Proof and certificate hierarchy}
The hardness proof needs a valid global affine Gaussian bound, not an optimal one. Proposition~\ref{hard:prop:gaussian-certificate} supplies a computer-assisted bound for the rational two-type test. Theorem~\ref{hard:thm:affine-criterion} converts it into a gap, and Theorem~\ref{hard:thm:ugc-hardness} gives the resulting factor. The constructed nonnegative directed-cut objectives belong to the promised quadratic class by Proposition~\ref{hard:prop:cut-equivalence}.

Proposition~\ref{hard:prop:affine-dual} serves the separate task of designing tests: it identifies the exact profile-mixing value obtained by optimizing over affine bounds.

\section{Hardness instantiations}\label{hard:sec:instantiations}
Adding isolated vertices transfers the NP-hardness bound from unconstrained Max-DiCut. The rational two-type test gives a stronger UGC-hardness bound, whose Gaussian certificate we also examine under the losing-half transfer.

\subsection{NP-hardness inherited from Max-DiCut}

\begin{theorem}[Explicit-input NP-hardness]\label{hard:thm:np-hardness}
For every fixed $\eps>0$, it is NP-hard to approximate
\eqref{hard:eq:qsub-cardinality} within $12/13+\eps$.  The statement holds even
for weighted Max-DiCut objectives with nonnegative arc weights, no terminals,
and budget $k=n'/2$, where $n'$ is the enlarged ground-set size.
\end{theorem}

\begin{proof}
For a weighted directed graph $D=([n],A)$, its directed-cut objective is
\[
 f_D(S)=\sum_{(i,j)\in A}w_{ij}\ind\{i\in S,\ j\notin S\}.
\]
This is an explicit quadratic submodular pseudo-Boolean function that is
nonnegative on the whole cube.  Trevisan et al.\ prove that it is NP-hard to
approximate unconstrained Max-DiCut within $12/13+\eps$
\citep{TrevisanSorkinSudanWilliamson2000}.

Add $n$ isolated vertices and set $k=n$ on the resulting $n'=2n$ vertices.
Every original cut now satisfies the budget. Conversely, removing selected
isolated vertices preserves its value. The two optima are equal, so an
approximation for the constrained instance gives the same factor for the
original instance.
\end{proof}

\subsection{A certified two-type test}

Let the types be $A,B$.  Set
\begin{equation}\label{hard:eq:pi-q}
 \pi_A=\frac1{10^6},\quad
 \pi_B=\frac{999999}{10^6},\qquad
 q_A=\frac{2289}{5000},\quad
 q_B=\frac{933}{10000},
\end{equation}
and
\begin{equation}\label{hard:eq:weights}
 W_{AA}=\frac{1493}{10000},\quad
 W_{BA}=\frac{6627}{10000},\quad
 W_{BB}=\frac{1880}{10000},\quad
 W_{AB}=0.
\end{equation}
Write $p_{st}=\nu_{st}(1,1)$.  Take
\begin{equation}\label{hard:eq:pairs}
 p_{BA}=p_{BB}=0,
 \qquad
 p_{AA}=\frac{73}{500}q_A^2.
\end{equation}
The two marginals determine the remaining probabilities. They are all
nonnegative, and the three support graphs are connected. The correlations are
\begin{equation}\label{hard:eq:correlations}
 \rho_{AA}=-\frac{977403}{1355500},\qquad
 \rho_{BA}=-\sqrt{\frac{2135637}{24580637}},\qquad
 \rho_{BB}=-\frac{933}{9067}.
\end{equation}
The exact density and dictator completeness are
\begin{align}
 \mu&=\frac{186600729}{2000000000}=0.0933003645,\label{hard:eq:mu}\\
 C&=\frac{17893931095731}{125000000000000}
   =0.143151448765848.\label{hard:eq:C}
\end{align}
For $(a,b)\in[0,1]^2$, the soundness profile is
\begin{align}
 \Psi(a,b)={}&\frac{1493}{10000}D_{\rho_{AA}}(a,a)
 +\frac{6627}{10000}D_{\rho_{BA}}(b,a)\notag\\
 &+\frac{1880}{10000}D_{\rho_{BB}}(b,b).
 \label{hard:eq:two-type-profile}
\end{align}

\begin{proposition}[Certified Gaussian inequality]\label{hard:prop:gaussian-certificate}
For every $(a,b)\in[0,1]^2$,
\begin{equation}\label{hard:eq:certified-affine}
 \Psi(a,b)\le
 \frac{58213}{1000000}
 +\frac{64713}{100000}
 \left(\frac{a}{1000000}+\frac{999999b}{1000000}\right).
\end{equation}
\end{proposition}

\begin{proof}
Appendix~\ref{hard:app:certificate} bounds the quadrature error and derives
coordinatewise nonincreasing derivative bounds. The verifier combines these
bounds with 256-bit MPFR evaluation and explicit outward allowances.
Adaptive subdivision then certifies the inequality on every cell.
\end{proof}

\begin{theorem}[Explicit-input UGC-hardness]\label{hard:thm:ugc-hardness}
Assuming the Unique Games Conjecture, for every $\eps>0$ there is no
polynomial-time
\[
 \left(
  \frac{118590464878885}{143151448765848}+\eps
 \right)
 =\left(0.828426578293753\ldots+\eps\right)
\]
approximation for \eqref{hard:eq:qsub-cardinality}.  The hardness holds for the
terminal-free Max-DiCut subclass at asymptotic budget density
\[
 \frac{k}{n'}=\frac{186600729}{2000000000}+o(1)
 =0.0933003645+o(1).
\]
\end{theorem}

\begin{proof}
Use Theorem~\ref{hard:thm:affine-criterion} with
\[
 \beta=\frac{58213}{1000000},
 \qquad
 \lambda=\frac{64713}{100000}.
\]
Proposition~\ref{hard:prop:gaussian-certificate} bounds soundness. Dividing
this bound by completeness gives
\[
 \frac{\beta+\lambda\mu}{C}
 =\frac{118590464878885}{143151448765848}
 =0.828426578293753\ldots.
\]
The reduction directly produces a directed-cut objective with nonnegative
arc weights, hence an explicit instance of \eqref{hard:eq:qsub-cardinality}.
\end{proof}

Type $A$ has cardinality mass $\pi_A=10^{-6}$ but positive incident
arc weight. The type matrix is non-Eulerian because $W_{AB}=0<W_{BA}$,
allowing nonzero directed imbalance in Proposition~\ref{hard:prop:flow}.
For Bernoulli marginals $q_s,q_t$, the Fr\'echet lower bound on the joint
probability $p_{st}$ is $\max\{0,q_s+q_t-1\}$. The $BA$ and $BB$
distributions attain this bound, while the $AA$ distribution is interior.

\subsection{The losing-half transfer for this test}

For this test, the losing-half composition changes the completeness bound
while retaining the certified soundness bound. Define the independent-coordinate value and the
certified soundness upper bound by
\[
 I:=\sum_{s,t}W_{st}q_s(1-q_t),
 \qquad S_{\mathrm{cert}}:=\beta+\lambda\mu.
\]
Exact substitution gives
\begin{align}
 C&=0.143151448765848,\qquad I=0.08648717847,\label{hard:eq:ci}\\
 S_{\mathrm{cert}}&=0.118590464878885.\label{hard:eq:soundness-value}
\end{align}
The strong half-completeness theorem of Bhangale and Khot gives an outer
labeling that satisfies every incident constraint at asymptotically half
of the relevant vertices \citep{BhangaleKhot2022}. At each such vertex,
the sampled endpoint dictators use the same coordinate and contribute $C$.
At any other outer vertex, the sampled coordinates can agree or differ.
For a type pair, agreement contributes $\nu_{st}(1,0)$ and disagreement
contributes $q_s(1-q_t)$. Negative correlation gives
$\nu_{st}(1,0)\ge q_s(1-q_t)$, so the contribution at every such vertex
is at least $I$. The completeness lower bound is therefore
$(C+I)/2-o(1)$. For a left-regular outer instance, assign right-vertex masses according to
the distribution induced by a uniformly sampled edge. Rational replication
converts these masses to ordinary cardinalities as above. For the available
soundness bound to certify a gap, one would need
$S_{\mathrm{cert}}<(C+I)/2$. Instead,
\[
 S_{\mathrm{cert}}>\frac{C+I}{2}=0.114819313617924,
 \qquad
 \frac{2S_{\mathrm{cert}}}{C+I}=1.032844224\ldots>1.
\]
The certificate therefore gives no completeness--soundness gap after this
conversion. This calculation does not determine the test's exact soundness
or rule out stronger NP-hardness bounds by other means.

\section{Further questions}\label{sec:conclusion}
\subsection{Recovery and representation}
We can construct certificates for the NNS subclass, but the other higher-degree guarantees require a supplied compatible certificate. Which further DR-submodular polynomials admit polynomial-size certificates? Corollary~\ref{cor:certificate_separation} excludes one nonnegative quartic DR polynomial from the compatible rooted cone through literal degree twelve. Representing that polynomial therefore requires a higher literal degree or different generators.

What is the exact value of $\rho_{13}^{\rm cw}$, and what do optimal or near-optimal portfolios look like? Theorem~\ref{thm:thirteen} proves a strict gap below $8/17$ without determining its size. At degree thirteen, the unrestricted joint modulus $\rho_{13}^{\rm joint}$ and the intrinsic recovery ratio relative to the global concave closure also remain open. Which overlap patterns admit stronger common recovery than arbitrary overlap? Can a tractable grouped roof, together with a suitable recovery rule, improve the $8/17$ bound for a class of constrained objectives?

\subsection{Explicit-input hardness}\label{hard:sec:open}

\begin{description}
 \item[Improved NP-hardness ratio.]
 Does there exist a constant $\delta>0$ for which it is NP-hard to
 approximate at-most-cardinality Max-DiCut within $12/13-\delta$?
 One possible route is to avoid the certificate failure in
 \eqref{hard:eq:ci}--\eqref{hard:eq:soundness-value} by a
 cardinality-aware composition or by optimizing a test directly for the
 losing-half certificate objective $2S_{\mathrm{cert}}/(C+I)$.

 \item[Optimal finite-type UGC test.]
 For a fixed density $\mu$, determine
 \[
  \inf\left\{S_\Psi(\mu)/C:
  \begin{array}{l}
  (\pi,q,W,\nu)\text{ is an admissible finite-type directed test},\\[-1mm]
  \sum_s\pi_sq_s=\mu,\ C>0
  \end{array}\right\}.
 \]
 Is a two-type construction in which one type mass tends to zero optimal in this class, or do more types and
 non-extremal pair distributions improve the ratio?

 \item[Bounded-complexity hard instances.]
 Can the ratio in Theorem~\ref{hard:thm:ugc-hardness} be retained for unweighted bounded-
 degree directed graphs with a maximum degree bounded independently of the
 approximation tolerance?
\end{description}

\Needspace{14\baselineskip}
\appendix
\part*{Appendices}
\addcontentsline{toc}{part}{Appendices}
\section{Elementary cubic scalar inequalities}\label{app:cubic}
We prove the scalar inequalities used by the cubic algorithm by reducing the slot domain to knot configurations and switching faces. Two cells require separate arguments because their Bernstein coefficients are not all nonnegative. Lemma~\ref{lem:cubic_maps} requires the unary inequality together with
\begin{equation}\label{eq:cubic_pair}
 \begin{split}
 8\phi_0(a)[1-\phi_0(u)]+9\phi_1(a)[1-\phi_1(u)]
 \ge8\min\{a,1-u\}.
 \end{split}
\end{equation}
\begin{equation}\label{eq:cubic_positive}
 \begin{split}
 8\phi_0(a)[1-\phi_0(u)\phi_0(v)]
 +9\phi_1(a)[1-\phi_1(u)\phi_1(v)]\\
 \ge8\min\{a,2-u-v\}.
 \end{split}
\end{equation}
\begin{equation}\label{eq:cubic_negative}
 \begin{split}
 8[1-\phi_0(a)][1-(1-\phi_0(u))(1-\phi_0(v))]\\
 {}+9[1-\phi_1(a)][1-(1-\phi_1(u))(1-\phi_1(v))]\\
 \ge8\min\{1-a,u+v\}.
 \end{split}
\end{equation}
The complemented rank-one inequality follows from \eqref{eq:cubic_pair} by interchanging the two coordinates.

Let $I_j=[(j-1)/4,j/4]$, $j=1,2,3,4$. Both maps are affine on each interval. Within a product of these intervals and a fixed affine branch of the roof, the recovery deficit defined in Section~\ref{sec:cubic_maps} is affine in each independent slot. Move one slot to an endpoint of its feasible interval without increasing the deficit. The endpoint is either a knot or a roof-switching point. Repeating this step shows that any negative deficit must occur with every slot at a knot or on the switching hyperplane. Rational substitution at the five unary knots, 25 pair knots, and 125 triples for each sign gives nonnegative deficits. At knots with zero roof value, atom nonnegativity gives the same conclusion.

Each nondegenerate intersection of a knot box with $a+u+v=1$ or $a+u+v=2$ is a triangle: after translation and scaling by four it is a section of $[0,1]^3$ with coordinate sum one or two. Write its vertices as $v_0,v_1,v_2$, in any order, and substitute $v_0+r(v_1-v_0)+t(v_2-v_0)$, where $r,t\ge0$ and $r+t\le1$. If the resulting polynomial is $\sum_{a+b\le3}c_{ab}r^at^b$, its total-degree-three Bernstein coefficients are
\begin{equation}\label{eq:bernstein_conversion}
 B_{ij}=\sum_{a\le i,\ b\le j}c_{ab}
       \frac{(i)_a(j)_b}{(3)_{a+b}},\qquad i,j\ge0,\quad i+j\le3,
\end{equation}
where $(k)_a=k!/(k-a)!$ and $(k)_0=1$. The corresponding Bernstein basis functions are
\[
 b_{ij}(r,t)=\frac{3!}{i!\,j!\,(3-i-j)!}\,r^it^j(1-r-t)^{3-i-j},
 \qquad i+j\le3.
\]
The polynomial equals $\sum_{i+j\le3}B_{ij}b_{ij}$. Since the basis functions are nonnegative and sum to one, nonnegative coefficients certify the entire triangle \citep{Bernstein1912,DeLoeraRambauSantos2010}. Permuting the vertices leaves the minimum coefficient unchanged.

On the pair switch $a=1-u$, the four nondegenerate interval pairs $(a,u)\in I_1\times I_4,I_2\times I_3,I_3\times I_2,I_4\times I_1$ have deficits
\[
 3(u-1)^2,\qquad\tfrac34(2u-1)^2,\qquad
 \tfrac34(2u-1)^2,\qquad(u-1)(3u-1),
\]
respectively. The last expression is nonnegative for $u\in[0,1/4]$, and the others are squares.

On the positive switch $a+u+v=2$, symmetry lets us restrict to interval indices with that of $u$ no larger than that of $v$. Table~\ref{tab:cubic_positive} lists the minimum coefficient from \eqref{eq:bernstein_conversion} for every nondegenerate triangle. Only one is negative.
\begin{table}[H]
\centering
\caption{Positive cubic switching triangles. A cell $(r;s,t)$ means $(a,u,v)\in I_r\times I_s\times I_t$ and $a+u+v=2$.}\label{tab:cubic_positive}
\begin{tabular}{cc@{\qquad}cc}\toprule
Cell&Minimum coefficient&Cell&Minimum coefficient\\\midrule
$(1;4,4)$&$0$&$(2;3,4)$&$233/192$\\
$(2;4,4)$&$325/384$&$(3;2,4)$&$89/192$\\
$(3;3,3)$&$2069/4608$&$(3;3,4)$&$89/192$\\
$(4;1,4)$&$-1/16$&$(4;2,3)$&$0$\\
$(4;2,4)$&$5/384$&$(4;3,3)$&$0$\\\bottomrule
\end{tabular}
\end{table}
For the exceptional cell $(4;1,4)$, write $x=u$, $d=1-v$. Its domain is $0\le d\le x\le1/4$ and $a=1-x+d$. The exact deficit is
\[
 Q_x(d)=1-4x+21x^2-5d+\frac{179}{2}xd(d-x).
\]
Let $x_0=\sqrt{10/179}$. For $0\le x\le x_0$, convexity in $d$ and the endpoint derivative give
\[
 \min_{0\le d\le x}Q_x(d)=Q_x(x)
   =21\left(x-\frac3{14}\right)^2+\frac1{28}>0.
\]
For $x_0<x\le1/4$, the constrained minimizer is $d_x=x/2+5/(179x)$, and
\[
 Q_x(d_x)=-\frac{32041x^4-30072x^3+9308x^2-1432x+100}{1432x}.
\]
Its derivative is $-H(x)/(1432x^2)$, where
\[
 H(x)=x^2C(x)-100,\qquad C(x)=96123x^2-60144x+9308.
\]
On $[x_0,1/4]$, $C$ is decreasing and $x_0>17/72$. If $C(x)>0$, then
\[
 x^2C(x)\le C(17/72)/16=805321/27648<100;
\]
if $C(x)\le0$, the same strict bound is immediate. Thus $H<0$, so the minimum deficit increases from a positive value.

Table~\ref{tab:cubic_negative} lists all nondegenerate triangles on the complemented switch $a+u+v=1$. For the sole exception $(4;1,1)$, substituting $a=1-u-v$ gives the exact deficit
\[
 3(u-v)^2+uv[4-3(u+v)]\ge0
 \qquad(u,v\ge0,\ u+v\le1/4).
\]
\begin{table}[H]
\centering
\caption{Complemented cubic switching triangles, with $a+u+v=1$.}\label{tab:cubic_negative}
\begin{tabular}{cc@{\qquad}cc}\toprule
Cell&Minimum coefficient&Cell&Minimum coefficient\\\midrule
$(1;1,3)$&$3/16$&$(1;1,4)$&$3/16$\\
$(1;2,2)$&$63/64$&$(1;2,3)$&$9/8$\\
$(2;1,2)$&$0$&$(2;1,3)$&$0$\\
$(2;2,2)$&$3/8$&$(3;1,1)$&$1/192$\\
$(3;1,2)$&$0$&$(4;1,1)$&$-1/48$\\\bottomrule
\end{tabular}
\end{table}
The triangle faces cover the degenerate intersections, completing the proof of all scalar inequalities and Lemma~\ref{lem:cubic_maps}. The ancillary check regenerates the rational knot and triangle calculations from \eqref{eq:elementary_maps}.

\section{Arbitrary-rank inequalities for the common single map}\label{app:single}
Product bounds and one-variable inequalities prove the arbitrary-rank lower bound in Theorem~\ref{thm:single}. The matching upper bound appears in Section~\ref{sec:single_upper}.

Write $g=\gamma_*=2\sqrt3-3$, $q=1-g$, $\sigma=\sigma_*$, and $\phi=\phi_*$. The map is increasing and concave, its slopes are $1,g,\sigma>0$, and $\phi(1)=q$. Thus $gt\le\phi(t)\le t$, $\phi(t)/t$ is nonincreasing, and $\log\phi$ is concave on $(0,1]$. Define
\begin{equation}\label{eq:chi_pieces}
 \chi(s)=1-\phi(1-s)=\begin{cases}
 g+\sigma s,&0\le s\le1/3,\\
 3q/4+gs,&1/3\le s\le3/4,\\s,&3/4\le s\le1.
 \end{cases}
\end{equation}
\subsection{Positive-root atoms}
Put $D=\sum_j(1-u_j)$. Log-concavity gives the product bound below; if the product is zero, it holds directly:
\[
 \prod_{j=1}^m\phi(u_j)\le\phi(1-D/m)^m.
\]
For $D>0$, the exact root-factor minimum is
\begin{equation}\label{eq:root_factor_positive}
 \inf_{0<a\le1}\frac{\phi(a)}{\min\{a,D\}}
 =\begin{cases}\phi(D)/D,&0<D\le1,\\q,&D\ge1.\end{cases}
\end{equation}
Monotonicity of $\phi$ and the nonincreasing quotient $\phi(a)/a$ give this minimum.

For one head, it suffices to prove
\begin{equation}\label{eq:single_pair}
 \phi(D)\chi(D)\ge gD\qquad(0\le D\le1).
\end{equation}
For $D\le1/4$ the deficit is $\sigma D^2$. For $D\ge3/4$ it is $D(\phi(D)-g)\ge0$. On $[1/3,2/3]$ the deficit is $3(1-2g)(2D-1)^2/4$. On $[1/4,1/3]$, its second derivative is $2g\sigma$, so it is convex. At the right endpoint, its derivative is $7(2g-1)/4<0$ and its value is $(1-2g)/12>0$. On $[2/3,3/4]$, it is again convex; at the right endpoint the derivative is $-3(11g-5)/8<0$ and the value is $3(7-15g)/64>0$. Thus \eqref{eq:single_pair} holds.

For two heads set
\[
 R_2(D)=\frac{\phi(D)}D\,[1-\phi(1-D/2)^2].
\]
For $0<D\le2/3$, monotonicity gives
\[
 R_2(D)\ge\frac{\phi(2/3)}{2/3}[1-\phi(1)^2]
  =\frac{111-231g}{8}
  =g+\frac{414-239\sqrt3}{4}>g.
\]
On $[2/3,1]$, differentiation gives $D^2R_2'(D)=3P(D)/8$, where
\[
 P(D)=(478\sqrt3-828)D^3+(1488-859\sqrt3)D^2+345\sqrt3-598.
\]
Indeed,
\[
 P'(D)=2D\bigl[(717\sqrt3-1242)D+1488-859\sqrt3\bigr]>0
 \qquad(2/3\le D\le1);
\]
the affine factor is positive at both endpoints. Also $P(1)=2(31-18\sqrt3)<0$. Hence $R_2$ is decreasing and
$R_2(D)\ge R_2(1)=q[1-\phi(1/2)^2]=g$.

To compare higher ranks with rank two, define $J(d)=\log\phi(1-d)/d$ for $0<d\le1/2$. On an affine piece $\phi(1-d)=A-Bd$,
\[
 J'(d)=\frac{-Bd/(A-Bd)-\log(A-Bd)}{d^2}.
\]
The numerator decreases with $d$. It therefore suffices to check it at the right endpoint of each piece. The bound $-\log z\ge1-z$ reduces these checks to
\[
 \phi(2/3)(2-\phi(2/3))-q=(64\sqrt3-109)/12>0,
\]
\[
 \phi(1/2)(2-\phi(1/2))-(g+q/4)=0.
\]
Thus $J$ is nondecreasing, and for $m\ge2$, $0<D\le1$,
\[
 \phi(1-D/m)^m=\exp(DJ(D/m))\le\exp(DJ(D/2))=\phi(1-D/2)^2.
\]
Together with \eqref{eq:root_factor_positive}, this proves the positive-root inequality for $D\le1$. For $D\ge1$, the root factor is $q$ and the product gap is nondecreasing in $D$, so the case $D=1$ suffices. The case $D=0$ has zero roof.

\subsection{Complement-rooted atoms}
Put $p(t)=1-\phi(t)$ and $S=\sum_j u_j$. If $S=0$, the roof is zero, so the inequality is immediate. Assume $S>0$. The quotient $\chi(s)/s$ is nonincreasing, and
\begin{equation}\label{eq:root_factor_negative}
 \inf_{0<r\le1}\frac{\chi(r)}{\min\{r,S\}}
 =\begin{cases}\chi(S)/S,&0<S\le1,\\1,&S\ge1.\end{cases}
\end{equation}
Consider $H(t)=\log p(t)/t$ and its limit $H(0+)=-1$. It decreases on $(0,1/4]$ because $-\log(1-t)<t/(1-t)$. On each later affine piece $p(t)=A-Bt$, the numerator of its derivative is $-Bt/(A-Bt)-\log(A-Bt)$ and decreases with $t$. The bound $-\log z\ge2(1-z)/(1+z)$ gives positive endpoint lower bounds
\[
 \frac{6(7-4\sqrt3)}{(18-5\sqrt3)(12-5\sqrt3)},\qquad
 \frac{33-19\sqrt3}{2(2\sqrt3-3)(\sqrt3-1)}
\]
at $2/3$ and one, respectively. Hence $H$ increases on $[1/4,1]$. For $0<S\le1$, the supremum of $H$ on $(0,S]$ is $\max\{H(0+),H(S)\}$; for $S\ge1$, the supremum over $0<t\le1$ is $H(1)=\log g>-1$. Heads with $u_j=0$ contribute $p(0)=1$ and may be omitted from the product. These bounds on the logarithmic quotient give
\begin{equation}\label{eq:negative_product_bound}
 \prod_jp(u_j)\le\begin{cases}
 \max\{e^{-S},p(S)\},&0\le S\le1,\\g^S,&S\ge1.
 \end{cases}
\end{equation}
For $0<S\le1$, equations \eqref{eq:root_factor_negative}--\eqref{eq:negative_product_bound} reduce the desired inequality to
\[
 \chi(S)\phi(S)\ge gS,\qquad \chi(S)(1-e^{-S})\ge gS.
\]
The first is \eqref{eq:single_pair}. For the second, write $K(S)=\chi(S)(1-e^{-S})-gS$. On $[0,1/3]$,
\[
 K''(S)=e^{-S}(2\sigma-g-\sigma S)
\]
changes sign at most once from positive to negative. Thus the minimum of $K'$ occurs at an endpoint. We have $K'(0)=0$ and
\[
 K'(1/3)>\sigma-g+\frac57\left(g-\frac{2\sigma}{3}\right)
          =\frac{366-211\sqrt3}{42}>0.
\]
Here $e^{-1/3}>5/7$, following from $e<68/25<(7/5)^3$. Hence $K\ge0$ on this interval. On $[1/3,3/4]$,
\[
 K'(S)=e^{-S}[3q/4+g(S-1)]>0,
\]
because the bracket at $1/3$ equals $\sigma/3>0$. On $[3/4,1]$, $K(S)=S(q-e^{-S})\ge0$, since $q>1/2>e^{-3/4}$. Finally, for $S\ge1$, the product gap is at least $1-g^S\ge1-g>g$. The complemented-root inequality follows, completing the proof of Theorem~\ref{thm:single}.

\section{Degree-twelve formal statement and map data}\label{app:twelve}
The Lean formalization proves the atomwise and unary inequalities in Theorem~\ref{thm:twelve_atoms}. We identify its statement, exact maps, and pinned dependencies below.

\subsection{Exact maps}
Both 23-knot portfolios use abscissas $\tau_j=k_j/40$ and ordinate denominator $M=1{,}200{,}000$. Table~\ref{tab:knots} specifies the four maps by linear interpolation between consecutive rows. The first pair is the degree-twelve pair; the second is the arity-four pair.
\begin{table}[!ht]
\centering\small
\caption{Exact ordinate numerators. Every ordinate shown is divided by $1{,}200{,}000$; the abscissa is $k_j/40$.}\label{tab:knots}
\begin{tabular}{rrrrrr}\toprule
& & \multicolumn{2}{c}{Degree-twelve atoms} & \multicolumn{2}{c}{Arity-four local functions}\\
\cmidrule(lr){3-4}\cmidrule(lr){5-6}
$j$&$k_j$&$M\alpha_{12}$&$M\beta_{12}$&$M\psi_0$&$M\psi_1$\\\midrule
0&0&0&0&0&0\\
1&2&60000&60000&60000&60000\\
2&4&120000&120000&120000&120000\\
3&6&180000&176600&180000&168000\\
4&8&240000&225180&240000&216000\\
5&10&300000&256980&300000&256980\\
6&12&360000&295200&360000&295200\\
7&14&417720&340020&417720&340020\\
8&16&476580&359076&476580&359076\\
9&18&537720&380000&537720&380000\\
10&19&567429&391720&570000&390000\\
11&20&600000&400000&600000&400000\\
12&21&608799&408280&607500&410000\\
13&22&613320&418452&613320&418452\\
14&24&627552&433908&627552&433908\\
15&26&643008&447768&643008&447768\\
16&28&660000&459600&660000&459600\\
17&30&673992&474096&673992&474096\\
18&32&540000&610600&540000&610600\\
19&34&405000&762400&405000&762400\\
20&36&270000&896400&272400&900973\\
21&38&135000&1046400&30000&1140000\\
22&40&0&1200000&0&1200000\\\bottomrule
\end{tabular}
\end{table}
At each knot, comparing the ordinate with the abscissa proves downwardness; linear interpolation preserves it between knots. The midpoint and endpoint contacts are the same for both portfolios; six knots have differing ordinates, with ten entries different in total.

\subsection{Formal correspondence and fixed version}
The source is identified by the public repository
\begin{center}\url{https://github.com/Duke1968/roof-recovery-theorem-4-10-lean}\end{center}
and tag \texttt{theorem-4.10-public-v1}, at commit
\begin{center}\texttt{ee0959069cf8a741893cf0dc9fcfeef3cb304d74}.\end{center}
The file \path{lean-toolchain} pins Lean version 4.33.0; \path{lake-manifest.json} pins Mathlib and its transitive dependencies. The exported declaration is:
\begin{quote}\ttfamily\small
RoofRecovery.theorem\_4\_10 : Degree12Theorem410Statement
\end{quote}
The file \path{RoofRecovery/Atoms.lean} defines \texttt{hPlus}, \texttt{lPlus}, \texttt{hMinus}, and \texttt{lMinus}, exactly as in \eqref{eq:slot_atoms}. Its \texttt{plusDeficit} and \texttt{minusDeficit} are the left side of \eqref{eq:twelve} minus its right side. The statement in \path{RoofRecovery/Degree12/Statement.lean} quantifies over every integer head rank from one through eleven and all root and head values in $[0,1]$. Its conjunction with the unary inequality is the exported statement above.

The map data and definitions are in the respective files
\begin{quote}\small
\path{RoofRecovery/Maps/Data.lean}\\
\path{RoofRecovery/Maps/Definition.lean}
\end{quote}
The agreement and facts modules identify these maps with the knot interpolation and prove downwardness, the contact values, and the unary inequality. The final dependency hierarchy is
\begin{quote}\small\ttfamily
 theorem\_4\_10\quad$\leftarrow$\quad
 (degree12Atomwise\_proved, degree12UnaryStatement),\\
 degree12Atomwise\_proved\quad$\leftarrow$\quad
 (degree12Plus\_proved,\\
 complementedOffSwitch\_proved).
\end{quote}
The statement quantifies over independent occurrence slots. Repeated-variable cases follow by identifying slots, without any change to the formal theorem.

\subsection{Analytical architecture of the formal proof}
A switching face consists of points where the two arguments of an atom roof agree. A strict knot configuration is a point off that face with every occurrence slot equal to a knot. On a knot box and an affine roof branch, a deficit is affine in each occurrence slot separately. Moving each slot to an endpoint reduces a minimum off the switching face to a strict knot configuration or to the face itself. Repeated coordinates are imposed only after this reduction on the larger independent-slot domain.

Two analytical estimates reduce the cases on switching faces. First, for $X,Y\ge0$,
\[
 X+Y=\max_{0\le\theta\le1}
 \frac{X^\theta Y^{1-\theta}}{\theta^\theta(1-\theta)^{1-\theta}},
\]
with the continuous endpoint conventions; when $X+Y>0$, the maximizing weight is $\theta=X/(X+Y)$. This is the weighted AM--GM variational identity.

Second, let $k\ge1$, let $x_1,\ldots,x_k$ lie in an interval on which $0<r_0\le r_1$, and set $P_\nu=\prod_{i=1}^k r_\nu(x_i)$ and $\bar x=k^{-1}\sum_i x_i<1$. Let $C_0,C_1\ge0$. If $C_0P_0+C_1P_1=0$, the bound below is immediate. Otherwise define its normalized first-product weight, called the \emph{posterior weight}, by
$\theta=C_0P_0/(C_0P_0+C_1P_1)$. Suppose $0\le\theta_0\le\theta\le1$ and
\[
 \theta_0\log r_0(x_i)+(1-\theta_0)\log r_1(x_i)
 \le\log(1-\bar x)-\frac{x_i-\bar x}{1-\bar x}
\]
for each $i$. Summation at $\theta$ and the variational identity give
\[
 C_0P_0+C_1P_1\le(C_0+C_1)(1-\bar x)^k.
\]
We call this estimate posterior--tangent compression. It replaces the head values by a bound involving their average, reducing the remaining cases to equal heads or to one distinguished head with the others equal.

The proof separates the two polarities and, for each, points on and off the switching face. On complemented switching faces, a log-ratio estimate between $1-\alpha_{12}$ and $1-\beta_{12}$ bounds the posterior weight. Replacing several heads by their common average then leaves bounds for one distinguished head. Complemented off-switch cases reach the switching face by monotonicity or scaling; other cases are bounded directly or extended to rank eleven by adding zero heads. On positive switching faces, explicit product bounds handle ranks four through eleven, with separate calculations for smaller ranks. Off the positive switching face, the endpoint reduction and exact product bounds at knots suffice, except that rank two requires separate quadratic calculations. The final assembly modules are
\begin{quote}\small
\path{RoofRecovery/Degree12/ComplementedOffSwitchCompleted.lean}\\
\path{RoofRecovery/Degree12/PositiveOffSwitchCompleted.lean}
\end{quote}

The finite data are rational. Algebraic residuals use exact arithmetic, interval monotonicity, and rational Bernstein coefficients; logarithmic and exponential residuals use analytic inequalities and rational-power bounds. No inequality is accepted on the basis of a floating-point test.

For $m\ge5$, both maps are the identity at $1/(2m)$. At root $a=1/2$ and equal complemented heads $u_i=1/(2m)$, the necessary inequality for this portfolio is $1-(1-1/(2m))^m\ge2/5$. Exact values at the terminal ranks are
\[
 1-(21/22)^{11}-2/5=
 \frac{1567401522879}{2921591507056640}>0,
\]
\[
 1-(23/24)^{12}-2/5=
 -\frac{12079851931877}{182601737180282880}<0.
\]
The listed portfolio therefore reaches literal degree twelve and fails at thirteen. Theorem~\ref{thm:thirteen} gives the stronger conclusion that even the supremum over all finite common-coordinatewise termwise portfolio factors is strictly below $8/17$ at degree thirteen. This analytical proof is independent of the degree-twelve formalization.

\subsection{Reproduction and verification scope}
Use the commit specified above without updating the dependency manifest. From its repository root, the documented commands are
\begin{verbatim}
lake exe cache get
LEAN_NUM_THREADS=4 lake build
lake env lean RoofRecovery/Degree12/Theorem410AxiomAudit.lean
\end{verbatim}
The final command prints the axiom dependencies of the exported theorem and its principal assembled components. The repository records the kernel-checked proof and its axiom audit. The statement correspondence above matches the exported theorem to Theorem~\ref{thm:twelve_atoms}.

\section{Exact-arithmetic verification for arity-four common maps}\label{app:arity_four}
To prove Theorem~\ref{thm:arity_four}, it suffices to verify inequalities on a complete set of cone generators over the marginal cube. The companion supplies exact certificates for the affine majorants, coverage of the domain, and polynomial nonnegativity.

\subsection{Polyhedral and roof reductions}
The cone of nonnegative submodular value tables on four labeled elements has the 16 inequalities $g(S)\ge0$ and the 24 elementary submodularity inequalities
\[
 g(K\cup\{i\})+g(K\cup\{j\})-g(K)-g(K\cup\{i,j\})\ge0,
 \quad i<j,\quad K\subseteq[4]\setminus\{i,j\}.
\]
The exact double-description enumeration gives 222 primitive integer ray generators, partitioned by coordinate permutations into 21 orbits. Double description adds the defining inequalities one at a time, updating the complete list of extreme rays after each addition. Each primitive generator is normalized to greatest common divisor one. Fixed common maps are invariant under these permutations. Since $G_g=\sum_r\lambda_rG_r$ and $g^+\le\sum_r\lambda_r r^+$ for nonnegative conic coefficients, it suffices to verify the same weighted map inequality on each ray.

For a ray $r$, an affine majorant $\ell$ satisfies $\ell(\ind_S)\ge r(S)$ at all 16 Boolean points. Taking expectations gives $\ell(x)\ge r^+(x)$. The minimum of a collection of valid majorants is also an upper bound. It therefore suffices to prove the recovery inequality against a selected affine majorant or a convex combination of valid majorants.

For the single map $\phi_*$, partition at $1/4$ and $2/3$. The active region of a stored majorant $\ell$ is the set where $\ell$ is no larger than any other stored majorant. Intersect each knot box with these regions. Triangulate each region and express $G_r(\phi_*(x))-\gamma_*\ell(x)$ in a Bernstein basis. Nonnegative coefficients prove nonnegativity on the simplex. The certificate records 3,461 active cells, 72,540 initial simplices, and 75,123 final simplices, after at most five exact longest-edge bisections. Every accepted sign is decided exactly in the field $\Q(\sqrt3)=\{a+b\sqrt3:a,b\in\Q\}$; floating-point tolerances play no role.

For the rational pair $\psi_0,\psi_1$, each coordinate has 22 knot intervals. For a fixed affine majorant, the polynomial
\[
 H_{r,\ell}(x)=8G_r(\psi_0(x))+9G_r(\psi_1(x))-8\ell(x)
\]
is affine in each coordinate separately. Nonnegative values at all 16 corners prove nonnegativity throughout the box. The same argument applies to an exact convex combination of two or three majorants. The 21 representatives give $21\cdot22^4=4{,}919{,}376$ coarse boxes. The certificate records 6,276 adaptive nodes with maximum depth eight.

When the deficit vanishes along a boundary, the corner test may fail to certify a box. The companion calls these \emph{singular cells}. Partition each such box into chambers indexed by a minimizing majorant. In the chamber for $j$, the inequalities $\ell_j\le\ell_k$ hold for every active $k$. On this chamber, verify the exact identity
\[
 H_{r,\ell_j}(x)=\sum_\nu c_\nu\prod_{h\in M_\nu}h(x),\qquad c_\nu\in\Q_{\ge0}.
\]
Each factor $h$ is either a nonnegative affine box slack, such as $x_i-l_i$ or $u_i-x_i$, or an ordering slack $\ell_k-\ell_j$. A \emph{Handelman identity} is such a nonnegative linear combination of products of affine slacks. Every term is nonnegative on the chamber. The certificate consists of 32 stored singular-cell certificates with 79 such degree-at-most-four Handelman identities and 1,238 nonnegative terms. Symmetries and one free ray coordinate extend them to 95 singular boxes. The certificate checks chamber coverage and domination exactly; sampled points do not establish either obligation.

\subsection{Acceptance obligations, GitHub source, and reproduction}
\label{app:proof_companion_archive}
The common-map proof requires completeness of the ray enumeration and
permutation coverage, Boolean validity of the affine majorants, exact map
reconstruction and downwardness, complete cell coverage, all box-corner signs
and simplex Bernstein-coefficient signs, and every singular-cell polynomial
identity with nonnegative coefficients.

The MIT-licensed arity-four proof companion is available from the GitHub
repository~\cite{DuArityFourProofCompanion2026}:
\begin{center}\small
\url{https://github.com/Duke1968/concave-roofs-proof-companion}.
\end{center}
The proof snapshot is also included in the ancillary directory
\path{anc/proof_companion}. It contains the verifier source and exact
certificate data for the arity-four common-map inequalities in
Theorem~\ref{thm:arity_four}, together with exact checks of the matching
aggregate upper-bound construction in Theorem~\ref{thm:aggregate_frontier}.
The companion includes the ray and affine-majorant data, rational recovery
maps, singular-cell Handelman certificates, supporting exact data,
environment documentation, and verification output.

The snapshot is identified by the SHA-256 digest of its \path{SHA256SUMS}
manifest:
\begin{center}\small
\texttt{f2168c77ded5a330f4627afadd58174a48c710fcbccabd5a981710ac1aa0e48e}.
\end{center}
This manifest records the digests of the other 339 companion files. To confirm
that a downloaded copy matches the article's snapshot, check the manifest digest
and then the file inventory and digests it records. These checks establish
file identity, not mathematical validity.

The commands below check the stored ray dictionary, its validity, and its
permutation coverage. Completeness of the generating list is established
by the double-description enumeration in
\begin{quote}\footnotesize
\path{research/unrestricted_four_factor/enumerate_four_rays_exact.py}
\end{quote}
To reconstruct the list, run that script in a working copy of the companion and compare the following output with the archived file:
\begin{quote}\footnotesize
\path{research/unrestricted_four_factor/nonnegative_submodular_four_rays_exact.json}
\end{quote} The script starts from the nonnegative orthant,
adds each elementary submodularity inequality, and uses exact adjacency
tests to update the ray list. Its additional symmetry audit uses global
complementation and reports 14 orbits; the common-map proof uses only
coordinate permutations and therefore retains 21 orbits.

\begin{samepage}
The two-map composite certificate is
\begin{center}
\path{research/recovery_pep/cegis_817_full_certificate.json}
\end{center}
with SHA-256 digest
\begin{center}\small
\texttt{b0a7cccc3b87d8652d800db679f7fbda94c655bff590239e5de502ecaa338d63}.
\end{center}
\end{samepage}
\begin{samepage}
Its candidate-map record is
\begin{center}
\path{research/recovery_pep/cegis817_root_12_candidate.json}
\end{center}
with digest
\begin{center}\small
\texttt{210b79dd8f1df90cc7a05717741b529a3f7b9dd06a5098771e789c36f7962493}.
\end{center}
\end{samepage}
The candidate-map record specifies the maps. The composite certificate
supplies the identities and coverage data needed to verify the inequalities.

The verification uses CPython~3.9 or later and only the Python standard
library. It must be run without \texttt{-O} or \texttt{-OO}, and with
\texttt{PYTHONOPTIMIZE} unset, because assertions form part of the proof
checks. Acceptance decisions use integer, rational, or exact
$\mathbb Q(\sqrt3)$ arithmetic. No inequality is accepted using a
floating-point tolerance. From the \path{proof_companion} directory, the checks
corresponding to the results used in this article are:
\begin{verbatim}
python3 replay_all.py --verify-only
python3 research/overlapping_factors/verify_phi_star_four_factor_exact.py
python3 research/recovery_pep/cegis_817_bridge.py \
  --replay research/recovery_pep/cegis_817_full_certificate.json
python3 research/recovery_pep/arbitrary_map_upper.py
python3 replay_all.py --verify-only
\end{verbatim}
The first and last commands check the complete file inventory and digests;
they do not execute the mathematical verifiers. Conditional on completeness
of the archived ray list, the three intervening commands check the
single-map inequality, the two-map inequality, and the aggregate upper
bound, respectively. Missing inputs, failed proof checks, or uncovered
cells cause the corresponding verifier to fail. File-integrity checks and
mathematical acceptance are separate obligations.

This article uses only the common-map and aggregate-bound routines listed
above. Every ray and local function uses the same maps and weights in the
common-map inequalities.

\section{Representation freedom and multilinear sharpness}\label{app:representation_frontiers}
Changing a representation allows coefficient adjustments as well as root reallocations. The adjustments that can change the roof form a fractional triple-packing system. Triangle-free capacity graphs force this packing to be zero, preserving the roofs of the sharpness examples.

\subsection{All uniform-polarity representation freedom in degree three}
Let $F(x)=\sum_{|S|\le3}a_Sx^S$ be nonnegative and submodular on the Boolean cube. For a pair $e$, put
\[
 r_e=-a_e-\sum_{T\supset e,\ |T|=3}\pos{a_T}.
\]
Submodularity is equivalent to $r_e\ge0$ for all pairs. In a representation by the functions $U_{T,i},V_{T,i}$, let $\bar\alpha_T,\bar\beta_T$ be the total positive- and complemented-product masses, summed over roots. Coefficient matching gives, for $|U|\ge2$,
\begin{equation}\label{eq:coefficient_freedom}
 a_U=-\bar\alpha_U+(-1)^{|U|+1}\sum_{T\supseteq U}\bar\beta_T.
\end{equation}
For any four-set, the left side is zero and every term on the right is nonpositive. Thus every complemented-product mass on a support of size at least four vanishes. Equation~\eqref{eq:coefficient_freedom} then also eliminates every positive-product mass on such a support. The remaining freedom on triples is exactly
\begin{equation}\label{eq:triple_packing}
 \bar\alpha_T=\pos{-a_T}+t_T,\qquad
 \bar\beta_T=\pos{a_T}+t_T,\qquad
 t_T\ge0,\qquad \sum_{T\supset e}t_T\le r_e.
\end{equation}
Conversely every vector in \eqref{eq:triple_packing} gives a uniform-polarity negaform of $F$ with nonpositive remaining pair coefficients. Theorem~\ref{thm:orientation} supplies a nonnegative rooted certificate because $F$ is Boolean-nonnegative.

Reallocating the root of a fixed product preserves its roof gap. For pairs, complementing the product also preserves the gap:
\[
 \min\{a,1-b\}-a(1-b)=\min\{b,1-a\}-b(1-a).
\]
Thus \eqref{eq:triple_packing} describes every adjustment that can affect $L_F^*$. The \emph{positive pair-capacity graph} has vertex set $N$ and an edge $e$ precisely when $r_e>0$. A graph is \emph{triangle-free} when it has no cycle of length three. If this graph is triangle-free, no triple can receive positive $t_T$, and the induced roof is unique within the entire representation family in \eqref{eq:best_representation_roof}.

\subsection{Quadratic odd cycles with an incomparable alternative}
For an odd cycle of length $\ell\ge5$, let $F_\ell(x)=\sum_{i=1}^\ell x_i(1-x_{i+1})$, with $x_{\ell+1}=x_1$. Its positive pair-capacity graph is that cycle and is triangle-free. Thus its best representation roof is
\[
 L_{F_\ell}^*(x)=\sum_i\min\{x_i,1-x_{i+1}\}.
\]
Fix $0<\varepsilon<\ell/2$, add coordinate $z$, and set
\[
 \widehat F_{\ell,\varepsilon}(x,z)=F_\ell(x)+(\ell/2-\varepsilon)z,
 \qquad
 P_\ell=\{(x,z):0\le z\le1,\ 0\le x_i\le(1-z)/2\}.
\]
The added coordinate has no pair capacity. On $P_\ell$ each minimum equals $x_i$, so
\[
 L_{\widehat F}^*(x,z)\le\ell(1-z)/2+(\ell/2-\varepsilon)z
 =\ell/2-\varepsilon z.
\]
The roof has unique maximizer $(\ind/2,0)$. Its lower-box maximum is $\ell/4$, whereas $(0,1)$ has objective value $\ell/2-\varepsilon$. The ratio to the true optimum is therefore at most $(\ell/4)/(\ell/2-\varepsilon)\to1/2$.

\subsection{A triangle-free cubic family}
The construction uses triangle-free graphs whose edge counts approximate a uniform density on every vertex subset.
\begin{lemma}\label{lem:quasirandom}
For all sufficiently large integers $m$, there are triangle-free graphs $G_m=(V_m,E_m)$, $|V_m|=m$, such that, with $Z_m=2|E_m|$, uniformly over $A\subseteq V_m$ and $u=|A|/m$,
\[
 \frac{\vol_{G_m}(A)}{Z_m}=u+o(1),\qquad
 \frac{e_{G_m}(A)}{Z_m}=\frac{u^2}{2}+o(1).
\]
Here $e_G(A)$ counts edges internal to $A$, and $\vol_G(A)$ is the sum of degrees in $A$. Uniform $o(1)$ means that the maximum absolute error over all subsets $A$ tends to zero as $m\to\infty$.
\end{lemma}
\begin{proof}
Start with the random graph $G(m,p)$, in which each pair of the $m$ vertices is independently an edge with probability $p=m^{-2/3}$. For any fixed $A$, the induced-edge count is a sum of independent Bernoulli variables with variance at most $pm^2/2$. The exponential bound
\[
 \Pr(|X-\E X|\ge t)\le2\exp\left(-\frac{t^2}{2(\operatorname{Var}X+t/3)}\right)
\]
follows from the exponential-moment bound for Bernoulli sums. Taking $t=Cm^{7/6}$ with a sufficiently large constant and a union bound over $2^m$ sets yields, with probability tending to one,
\[
 e(A)=p\binom{|A|}{2}+O(m^{7/6})
       =p\binom{|A|}{2}+o(pm^2)
\]
uniformly. The expected number of triangles is $O(m)$, so Markov's inequality makes it at most $m\log m$ with probability tending to one. Choose a realization satisfying both events, then delete one edge from each triangle. At most $m\log m=o(pm^2)$ edges are deleted, preserving the uniform estimate and $|E_m|=(1+o(1))pm^2/2$. Finally,
\[
 \vol_{G_m}(A)=|E_m|+e_{G_m}(A)-e_{G_m}(V_m\setminus A)
\]
gives the volume estimate after division by $Z_m$.
\end{proof}
Pass to $m=3q$ with integers $q\ge2$ tending to infinity, put $U_q=17/12-1/(2q)$, and define
\begin{equation}\label{eq:trianglefree_objective}
\begin{split}
 F_{q,G}(x_0,x_1,t,z)={}&(U_q-q^{-2})z+\frac{5q}{12}x_0(1-x_1)\\
 &+\frac1{Z_m}\sum_{\{i,j\}\in E_m}x_1(1-t_it_j)\\
 &+\frac1{Z_m}\sum_{\{i,j\}\in E_m}
       [t_i(1-t_j)+t_j(1-t_i)].
\end{split}
\end{equation}
Every nonlinear term is a nonnegative rooted multilinear atom of degree at most three. Use the down-closed polytope
\[
 P_q=\{(x_0,x_1,t,z)\in[0,1]^{m+3}:
 x_0\le q^{-1}(1-z),\quad x_1\le(1-q^{-1})(1-z),\quad
 t_i\le(1-z)/2\}.
\]
The positive pair-capacity graph is the disjoint union of the edge for $x_0,x_1$ and $G_m$. It is triangle-free, so \eqref{eq:triple_packing} forces zero triple-packing mass. Thus $L_F^*$ equals the roof of the displayed representation. On $P_q$ all its minima use their root arguments, giving
\begin{equation}\label{eq:trianglefree_roof}
 L_F^*=\frac{5q}{12}x_0+\frac12x_1+
       \frac1{Z_m}\sum_{\{i,j\}\in E_m}(t_i+t_j)
       +(U_q-q^{-2})z\le U_q-q^{-2}z.
\end{equation}
Equality holds at $\bar x_q=(q^{-1},1-q^{-1},\ind/2,0)$. Every roof maximizer agrees with this point on the coordinates appearing in the objective; isolated vertices may be deleted.

By multiaffinity, a lower-box optimum has $x_0=q^{-1}$, $x_1=s\in\{0,1-q^{-1}\}$, and $t_i\in\{0,1/2\}$. Writing $A$ for the retained tail coordinates and $u=|A|/m$, substitution gives
\begin{align*}
 F_{q,G}&=\frac5{12}+\frac{s}{12}
       +\frac{\vol_{G_m}(A)}{2Z_m}
       -\frac{(2+s)e_{G_m}(A)}{4Z_m}\\
 &=\frac5{12}+\frac{s}{12}+\frac u2-\frac{2+s}{8}u^2+o(1),
\end{align*}
uniformly in $A$. At $s=0$ the limiting maximum is $2/3$ at $u=1$; at $s=1-q^{-1}$ it is
$5/12+s/12+1/[2(2+s)]=2/3+o(1)$.
The lower-box maximum is therefore $2/3+o(1)$, while the roof value tends to $17/12$. Their ratio tends to $8/17$. The incomparable feasible point $(0,0,0,1)$ has value $U_q-q^{-2}$, and \eqref{eq:trianglefree_roof} bounds the true optimum by $U_q$. Thus the same limit holds relative to the true multilinear optimum, proving Theorem~\ref{thm:best_representation}.

\section{Concave-cardinality and symmetric sequences}\label{app:cardinality}
An explicit decomposition of concave sequences into affine terms and tents reduces the proof to these generators. Pairwise averaging then reduces recovery to binomial expectations, from which the sharp constants and witnesses follow.

\subsection{The full nonnegative concave sequence cone}
For $d\ge2$, let $h(0),\ldots,h(d)$ be concave and nonnegative. Define
\[
 \ell_0(k)=1-k/d,\quad\ell_d(k)=k/d,\qquad
 \tau_j(k)=\min\{k/j,(d-k)/(d-j)\}\quad(1\le j<d).
\]
The sequence $\tau_j$ is a \emph{tent}: its piecewise-affine interpolation rises from zero to one at $j$ and then falls to zero. Write $\bar\tau_j$ for this interpolation. With $\omega_j=2h(j)-h(j-1)-h(j+1)\ge0$, the exact expansion is
\begin{equation}\label{eq:tent_expansion}
 h(k)=h(0)\ell_0(k)+h(d)\ell_d(k)
       +\sum_{j=1}^{d-1}\omega_j\frac{j(d-j)}d\tau_j(k).
\end{equation}
The endpoints and second differences agree, proving the identity. It generates the full cone because a concave sequence is nonnegative whenever its endpoint values are nonnegative.

At fixed total marginal sum, averaging two Bernoulli probabilities cannot increase the expectation of a concave sequence. Indeed, conditional on the remaining sum $W=w$, for probabilities $a,b$ with $a+b=t$ the expectation is
\[
 (1-t)h(w)+th(w+1)+ab[h(w)-2h(w+1)+h(w+2)].
\]
The bracket is nonpositive and $ab$ is maximized at $a=b=t/2$. By repeated averaging and continuity, equal probabilities $s/d$ minimize the expectation at the given marginal sum.

Write $\Bin(d,p)$ for the distribution of the number of successes in $d$ independent Bernoulli trials of success probability $p$. For $X_s\sim\Bin(d,s/d)$ and $0<s\le j$,
\[
 \frac{\E\tau_j(X_s)}{\bar\tau_j(s)}
 =1-\frac d{d-j}\frac{\E(X_s-j)_+}{s}.
\]
The binomial expectation of $(X-j)_+$ is convex in the success probability. More generally, if $B_a(p)=\E a(\Bin(d,p))$, then
\[
 B_a''(p)=d(d-1)\E[\Delta^2a(\Bin(d-2,p))],
 \quad\Delta^2a(k)=a(k+2)-2a(k+1)+a(k).
\]
For $a(k)=(k-j)_+$, the expectation vanishes at zero; convexity therefore makes its quotient by $p$ nondecreasing. The tent ratio is thus nonincreasing up to $s=j$. Complementing the binomial variable shows that it is nondecreasing after $j$. At $p=j/d$, telescoping gives
\[
 \E(X-j)_+=dp(1-p)\Pr(\Bin(d-1,p)=j).
\]
The minimum tent ratio is consequently $1-Q_{d,j}$, where
\[
 Q_{d,j}=\binom dj(j/d)^j(1-j/d)^{d-j}.
\]
These numbers are symmetric in $j,d-j$, and
\[
 \frac{Q_{d,j+1}}{Q_{d,j}}
 =\frac{(1+1/j)^j}{(1+1/(d-j-1))^{d-j-1}}\qquad(1\le j\le d-2).
\]
Since $(1+1/u)^u$ increases with $u$, the maximum is at $j=1,d-1$ and equals $(1-1/d)^{d-1}$. For the affine generators in \eqref{eq:tent_expansion}, the expectation equals the closure at every marginal vector. The identity-map inequality in Theorem~\ref{thm:cardinality} follows by nonnegative aggregation, since all closures here equal the linear interpolation evaluated at $s$.

For sharpness under arbitrary downward recovery, use $\tau_1$ and $x_i=1/d$. Its multilinear extension is
\[
 G_{\tau_1}(y)=\frac{d(1-\prod_i(1-y_i))-\sum_i y_i}{d-1}.
\]
For $0\le y_i\le1/d$,
\[
 \frac{\partial G_{\tau_1}}{\partial y_i}
 =\frac{d\prod_{\ell\ne i}(1-y_\ell)-1}{d-1}\ge0,
\]
because Bernoulli's inequality gives $(1-1/d)^{d-1}\ge1-(d-1)/d=1/d$. Thus no downward point improves on $x$, and its ratio is $\rho_d$.

\subsection{Symmetric zero-endpoint sequences}
Let $m=\lfloor d/2\rfloor$. The symmetric, nonnegative, concave sequences with zero endpoints are generated by
\[
 \tau_{d,r}(k)=\min\{k,r,d-k\},\qquad1\le r\le m.
\]
The successive differences of the nonincreasing, nonnegative slopes on the first half give the nonnegative coefficients. Pairwise averaging again reduces to equal Bernoulli probabilities. For
$\Phi_{d,r}(s)=\E\tau_{d,r}(\Bin(d,s/d))$,
the second-derivative identity makes $\Phi$ concave, symmetric about $d/2$, and zero at zero. Therefore $\Phi(s)/s$ is nonincreasing up to $r$, while $\Phi(s)/r$ is nondecreasing from $r$ to $d/2$. Complementation gives the same conclusion on the remaining half. The exact constant is
\begin{equation}\label{eq:symmetric_finite}
 \rho_d^{\rm sym}=\min_{1\le r\le\lfloor d/2\rfloor}
 \frac1r\E\min\{X,r,d-X\},\qquad X\sim\Bin(d,r/d).
\end{equation}
The term $r=1$ equals $1-(1-1/d)^d-d^{-d}$.

To determine the minimizing ray for $d\ge5$, express the deficit of the $r$th quotient as $(A_{d,r}+U_{d,r})/r$, where
\[
 A_{d,r}=\E(r-X)_+,\qquad U_{d,r}=\E(X-d+r)_+.
\]
A binomial identity gives
\[
 \frac{A_{d,r}}r=\binom{d-1}{r-1}(r/d)^{r-1}(1-r/d)^{d-r+1}
 \le\pi_r:=e^{-r}\frac{r^r}{r!}.
\]
For fixed $r$ the ratio of the expression at $d+1$ to that at $d$ is
\[
 \frac{(1+1/(d-r))^{d-r+1}}{(1+1/d)^{d+1}}>1,
\]
and its limit is $\pi_r$, proving the bound. At $d=2r$, symmetry gives deficit $\binom{2r}{r}/4^r\le5/16$ for $r\ge3$.

For $d=2r+\Delta$, $\Delta\ge1$, put $\vartheta=r/(r+\Delta)$. Pair the terms with $X=j$ and $X=d-j$, $0\le j<r$, which have the same weights $r-j$. Their probability ratio is
\[
 \frac{\Pr(X=d-j)}{\Pr(X=j)}=\vartheta^{d-2j}
 \le\vartheta^{\Delta+2}.
\]
Thus the deficit is at most $\pi_r(1+\vartheta^{\Delta+2})$. At $(r,\Delta)=(2,1),(3,1)$, direct evaluation gives recovery quotients $93/125$ and $13108/16807$. For $r=2,3$ and $\Delta\ge2$ the deficit bound is below $5/16$. For $r=4,5,6$, its respective bounds are below $37/125,139/500,131/500$; these follow by taking $\Delta=1$ and using $e>1359/500$. The factor $(r/(r+\Delta))^{\Delta+2}$ decreases with $\Delta$ for $\Delta\ge1$, as differentiation of its logarithm shows. For $r\ge7$, the elementary Stirling lower bound $r!\ge\sqrt{2\pi r}(r/e)^r$ gives deficit at most $2/\sqrt{2\pi r}<5/16$.

Every $r\ge2$ therefore has quotient at least $11/16$ when $d\ge5$, whereas the $r=1$ quotient is smaller because $(1-1/d)^d\ge(4/5)^5>5/16$. Therefore $r=1$ is the minimizer. For $d=2,3,4$, direct evaluation of \eqref{eq:symmetric_finite} gives $1/2$, $2/3$, and $\min\{87/128,5/8\}=5/8$. These prove \eqref{eq:symmetric_cardinality}. The limiting value is $1-1/e$; for $d\ge5$, the inequalities $(1-1/d)^d<e^{-1}$, $d^{-d}\le5^{-5}$, and $e>27/10$ show that the value exceeds $5/8$.

\section{Exact overlap examples and portfolio separation}\label{app:overlap}
Two overlap examples establish bounds relative to additive local roofs. We then show that the 23-knot portfolios fail each other's fixed-weight inequalities, despite sharing knots and contact values.

\subsection{Two local functions with incompatible closures}
On coordinates $(x_0,x_1,t_1,t_2,t_3)$, put $\bar x=(1/20,19/20,1/2,1/2,1/2)$ and define
\begin{align*}
 G_1={}&2200x_0(1-x_1)+40x_1(1-t_1t_2)+19(t_1+t_2-2t_1t_2),\\
 G_2={}&40x_1(1-t_1t_3)+40x_1(1-t_2t_3)\\
       &+19(t_1+t_3-2t_1t_3)+19(t_2+t_3-2t_2t_3).
\end{align*}
Their Boolean restrictions $g_1,g_2$ are nonnegative and submodular, with respective supports $\{x_0,x_1,t_1,t_2\}$ and $\{x_1,t_1,t_2,t_3\}$. To attain the first closure, impose $X_0=1-X_1$ and $T_2=1-T_1$, with the two pairs independent. Each summand attains its local upper bound, giving $g_1^+(\bar x)=167$. For the second, impose $T_1=T_2=1-T_3$ and take $X_1$ independently; this gives $g_2^+(\bar x)=114$. Thus the specified two-function roof has value 281.

The aggregate is nondecreasing in $x_0$, so retain $x_0=1/20$. By Lemma~\ref{lem:vertex}, let $s\in\{0,19/20\}$ be the retained value of $x_1$ and $k\in\{0,1,2,3\}$ the number of retained tail coordinates. The eight values are given by
\[
 H(s,k)=110(1-s)+40s\left(3-\frac{k(k-1)}8\right)
             +\frac{19}{2}\left(3-\frac{(3-k)(2-k)}2\right).
\]
Their maximum is $277/2$, and hence
\[
 \frac{\max_{0\le y\le\bar x}(G_1(y)+G_2(y))}
 {g_1^+(\bar x)+g_2^+(\bar x)}=\frac{277}{562}<\frac12.
\]
Attaining either closure requires its positive cut terms to saturate. This forces \(T_1+T_2=1\) almost surely for \(g_1\) and \(T_1=T_2\) almost surely for \(g_2\), so no joint law attains both. The example limits recovery relative to the additive roof; it does not exclude better recovery from a closure computed after grouping the functions.

\subsection{Exact arity-three aggregate barrier}
Fix $q\ge2$, $m=3q$, and $D=m(m-1)$. Set
\begin{equation}\label{eq:aggregate_example}
 \begin{split}
 F_q(x_0,x_1,t)={}&\frac{5q}{12}x_0(1-x_1)
    +\frac1D\sum_{i<j}x_1(1-t_it_j)\\
 &+\frac1D\sum_{i<j}[t_i(1-t_j)+t_j(1-t_i)],
 \end{split}
\end{equation}
treating each displayed rooted atom as a local function. Put
$\bar x_q=(q^{-1},1-q^{-1},\ind/2)$.
The additive roof is
\[
 L(\bar x_q)=\frac5{12}+\frac12(1-q^{-1})+\frac12
 =\frac{17}{12}-\frac1{2q}.
\]
Below $\bar x_q$, the root argument is the smaller argument in every minimum, so
\[
 L(x)=\frac{5q}{12}x_0+\frac12x_1+\frac1m\sum_i t_i.
\]
Thus $\bar x_q$ maximizes the roof on its lower box. The objective is increasing in $x_0$, so set $x_0=q^{-1}$. For $s=x_1$, $\mu=m^{-1}\sum_it_i$, and $Q=D^{-1}\sum_{i\ne j}t_it_j$, expansion gives
\[
 F_q(q^{-1},s,t)=\frac5{12}+\frac{s}{12}+\mu-(1+s/2)Q.
\]
At a lower-box vertex, let $k$ tail coordinates equal $1/2$. At $s=0$ and $s=1-q^{-1}$ the respective values are
\[
 \frac23-\frac{(m-k)(m-k-1)}{4m(m-1)},\qquad
 \frac23-\frac{(k-2q)(k-2q-1)}{24q^2}.
\]
The subtracted products are nonnegative, since their factors are consecutive integers, and equality is attained. Therefore $\max_{0\le y\le\bar x_q}F_q(y)=2/3$, and the ratio is $8q/(17q-6)\downarrow8/17$. This proves the upper bound in Theorem~\ref{thm:aggregate_frontier} without any restriction on the joint downward rule.

\subsection{The two 23-knot portfolios cannot be substituted}\label{app:maps}
The knots, midpoint contacts, and weights shared by the portfolios in Table~\ref{tab:knots} do not make their inequalities interchangeable. Exact substitution shows failure in both directions.

First, take
\[
 G(a,b,c,d)=a(2-c-d)+(1-a)b(1-cd),\qquad
 x=(1/20,19/20,19/40,19/40).
\]
This polynomial is multilinear and nonnegative. Its six mixed second derivatives are
\[
 -(1-cd),\quad-1+bd,\quad-1+bc,\quad
 -(1-a)d,\quad-(1-a)c,\quad-(1-a)b,
\]
all nonpositive. Hence its Boolean restriction $g$ is submodular. A distribution placing masses $1/20,19/40,19/40$ on $(1,0,0,0),(0,1,1,0),(0,1,0,1)$ has marginals $x$ and expected value $21/20$. The affine majorant $2-c-d$ dominates all 16 Boolean values and attains $21/20$ at $x$, proving $g^+(x)=21/20$.

The degree-twelve maps give
\[
 \alpha_{12}(x)=(1/20,9/80,189143/400000,189143/400000),
\]
\[
 \beta_{12}(x)=(1/20,109/125,9793/30000,9793/30000).
\]
Exact substitution yields
\begin{equation}\label{eq:map_counterexample_C}
 8G(\alpha_{12}(x))+9G(\beta_{12}(x))-8g^+(x)
 =-\frac{1507464166491}{32000000000000}<0.
\end{equation}
The degree-twelve pair therefore fails the universal arity-four inequality. Its weighted recovery-to-closure ratio on this point is
$89097511944503/190400000000000<8/17$.

Conversely, take the positive rooted atom $H(a,u)=a(1-\prod_{j=1}^6u_j)$, with $a=1$ and all $u_j=19/20$. Write $x=(a,u_1,\ldots,u_6)$ and $\ell(x)=\min\{a,\sum_{j=1}^6(1-u_j)\}$; then $\ell(x)=3/10$. The arity-four pair has $\psi_0(1)=0$, $\psi_1(1)=1$, and $\psi_1(19/20)=19/20$. Hence
\begin{equation}\label{eq:map_counterexample_D}
 8H(\psi_0(x))+9H(\psi_1(x))-8\ell(x)
 =9[1-(19/20)^6]-12/5=-\frac{1012929}{64000000}<0.
\end{equation}
The witness uses distinct variables and has literal degree seven.

These calculations disprove the fixed-weight universal inequalities. They do not disprove the best-branch guarantees on these instances, exclude reweighting, or rule out another portfolio for the union of classes.

\begin{corollary}[A separation from the degree-twelve certificate cone]\label{cor:certificate_separation}
The polynomial $G$ in \eqref{eq:map_counterexample_C} does not admit a nonnegative compatible continuous rooted certificate of literal degree at most twelve.
\end{corollary}
\begin{proof}
Such a certificate would induce a concave Boolean-exact roof $L\ge g^+$ by the minimality of concave closure. Theorem~\ref{thm:twelve_atoms} and Lemma~\ref{lem:mixed} would then give
$8G(\alpha_{12}(x))+9G(\beta_{12}(x))\ge8L(x)\ge8g^+(x)$,
contradicting \eqref{eq:map_counterexample_C}.
\end{proof}

\section{Computer-assisted Gaussian certificate}\label{hard:app:certificate}
The uniform Gaussian bound follows from an analytic quadrature remainder, derivative enclosures, and subdivision of the profile square into finitely many rectangles. Each accepted rectangle has enough margin to absorb the stated rounding allowances.

Define
\begin{equation}\label{hard:eq:L}
 H_G(a,b):=\Psi(a,b)-\frac{64713}{100000}
 \left(\frac{a}{1000000}+\frac{999999b}{1000000}\right).
\end{equation}
Proposition~\ref{hard:prop:gaussian-certificate} requires
\begin{equation}\label{hard:eq:L-target}
 H_G(a,b)\le\frac{58213}{1000000}
 \qquad((a,b)\in[0,1]^2).
\end{equation}

\subsection{Integral formula and uniform remainder}

For interior $a,b$, set $x=\Phi^{-1}(a)$ and $y=\Phi^{-1}(b)$.  Plackett's
identity \citep{Plackett1954} gives, for $\rho<0$,
\begin{equation}\label{hard:eq:plackett}
 D_\rho(a,b)=a(1-b)+\int_\rho^0
 \frac{\exp\!\left[-\dfrac{x^2-2rxy+y^2}{2(1-r^2)}\right]}
 {2\pi\sqrt{1-r^2}}\,dr.
\end{equation}
Let the integrand be $g(r)$, put $d=1-r^2$, and define
\[
 U=\frac{x+y}{\sqrt{2(1+r)}},
 \qquad
 V=\frac{x-y}{\sqrt{2(1-r)}}.
\]
Writing $\ell=\log g$ gives
\[
 \ell'=\frac r d+\frac{U^2}{2(1+r)}-\frac{V^2}{2(1-r)},
\]
and
\[
 \ell''=\frac{1+r^2}{d^2}
 -\frac{U^2}{(1+r)^2}-\frac{V^2}{(1-r)^2}.
\]
All correlations in \eqref{hard:eq:correlations} lie in $[-151/200,0]$.  With
\[
 r_0=\frac{151}{200},\qquad
 d_0=1-r_0^2=\frac{17199}{40000},\qquad
 h_0=1-r_0=\frac{49}{200},
\]
the identities $g''=g[(\ell')^2+\ell'']$,
$\sup_t t^2e^{-t^2/2}=2/e$, and
$\sup_t t^4e^{-t^2/2}=16/e^2$ imply
\begin{align*}
 |g''(r)|\le{}&\frac1{2\pi\sqrt{d_0}}
 \left[
  \frac{3r_0^2}{d_0^2}
  +\frac{3(16/e^2)}{4h_0^2}
  +\frac{3(16/e^2)}4\right.\\
 &\left.\hspace{24mm}
  +\frac{1+r_0^2}{d_0^2}
  +\frac{2/e}{h_0^2}+\frac2e
 \right]
 <14.421<16.
\end{align*}
The $4000$-panel composite trapezoidal rule therefore has uniform error at
most
\begin{equation}\label{hard:eq:trap-error}
 \frac{16(151/200)^3}{12\cdot4000^2}
 \approx3.5864072917\times10^{-8}
\end{equation}
for each Gaussian term.

\subsection{Monotone derivative enclosure}

The exact derivatives are
\begin{align*}
 \partial_1D_\rho(a,b)
 &=1-\Phi\!\left(
   \frac{\Phi^{-1}(b)-\rho\Phi^{-1}(a)}{\sqrt{1-\rho^2}}
  \right),\\
 \partial_2D_\rho(a,b)
 &=-\Phi\!\left(
   \frac{\Phi^{-1}(a)-\rho\Phi^{-1}(b)}{\sqrt{1-\rho^2}}
  \right),
\end{align*}
and, on the diagonal,
\[
 \frac d{dt}D_\rho(t,t)
 =1-2\Phi\!\left(
  \sqrt{\frac{1-\rho}{1+\rho}}\,\Phi^{-1}(t)
 \right).
\]
For $\rho<0$, each component of $\nabla H_G$ is nonincreasing in both profile
coordinates. To bound a rectangle, the
verifier evaluates $H_G$ at its center and integrates the endpoint derivative
enclosures along each of the two coordinate orders. Both give valid variation
bounds, so it retains the smaller one.

A Gaussian quantile is $\Phi^{-1}(a)$. MPFR provides multiple-precision
floating-point arithmetic with specified rounding modes; binary64 uses a
53-bit significand.
The verifier uses round-to-nearest binary64 arithmetic and is compiled without
unsafe reassociation or fast-math options.
Gaussian quantiles are located by 220 bisection steps at 256-bit precision.
The integral is evaluated with MPFR arithmetic, directed output rounding, the
analytic error \eqref{hard:eq:trap-error}, and an outward $10^{-12}$ allowance at
every density node.  The derivative enclosures use a further $10^{-9}$
outward allowance for conversion of quantiles and Gaussian arguments to
binary64 before the final MPFR evaluation of the normal cumulative distribution
function (CDF). The verifier records these allowances explicitly. The bounds below show that
they cover conversion and roundoff errors throughout the allowed subdivision
depth and correlation interval.

At subdivision depth at most $40$, every interior profile coordinate passed to the quantile
routine lies in $[2^{-41},1-2^{-41}]$, so its normal quantile lies in
$[-8,8]$.  The 220-step MPFR bisection and final binary64 conversion give
absolute quantile error below $4\times10^{-15}$.  If $g$ denotes the density
in \eqref{hard:eq:plackett}, completing the square gives, uniformly for
$r\in[-151/200,0]$,
\[
 |\partial_xg|,\ |\partial_yg|
 \le \frac{e^{-1/2}}{2\pi d_0}<0.225.
\]
The $10^{-12}$ density allowance therefore covers quantile conversion and
binary64 errors in the base term.  The round-to-nearest operations internal to each
256-bit MPFR density evaluation, including the stored values of $\pi$ and the
quadrature node, contribute less than $10^{-70}$ in absolute value: on the
stated domain all algebraic intermediates are bounded, the denominator is at
least $2\pi\sqrt{d_0}$, and fewer than 40 correctly rounded operations are
performed.  The same precision makes CDF roundoff negligible relative to the
quantile-conversion bound.

The verifier stores the right-hand side of
\eqref{hard:eq:trap-error} in binary64; it is below the exact rational remainder by
less than $2\times10^{-24}$.  The $10^{-12}$ node allowance contributes at
least
\[
 |\rho_{BB}|10^{-12}=\frac{933}{9067}10^{-12}>10^{-13}
\]
to every computed Gaussian integral, and therefore covers the stored remainder's deficit and the
density-evaluation errors above.  Using binary64 unit roundoff
$2^{-53}$,
$|x|,|y|\le8$, and $\sqrt{1-\rho^2}\ge\sqrt{d_0}>0.65$, direct propagation
through each derivative argument gives error below $10^{-12}$.  Since
$\sup_z\Phi'(z)=1/\sqrt{2\pi}<0.4$ and the CDF truncation beyond $12$ is
below $2\times10^{-33}$, the $10^{-9}$ derivative allowance is conservative.

\subsection{Certificate output}

A cell's \emph{pruning margin} is the binary64 comparison target minus the
computed upper bound on $H_G$ over the entire cell. The \emph{rounding reserve}
is an additional margin that covers errors in the final assembly of this bound
and rounding of the comparison target. A cell is \emph{pruned}, and becomes a
\emph{terminal cell}, when its margin is finite and exceeds the reserve; it is
then not subdivided further. The reported \emph{residual margin} is the
smallest pruning margin minus the reserve, rounded downward.
Version \texttt{gaussian-reserve-v1} uses this acceptance rule and reports:
\begin{center}
\begin{tabular}{@{}lr@{}}
\toprule
precision & 256 bits\\
trapezoidal panels & 4,000\\
stored quadrature addend & $3.5864072916666665\times10^{-8}$\\
cells visited & 595\\
cells pruned & 298\\
maximum subdivision depth & 23\\
largest point upper bound & $0.058212883635230132$\\
location of largest point bound & $(0.44775390625,0.0712890625)$\\
smallest final pruning margin & $8.4277705231228239\times10^{-9}$\\
required final reserve & $2.0000000000000004\times10^{-12}$\\
residual margin lower bound & $8.4257705231228217\times10^{-9}$\\
\bottomrule
\end{tabular}
\end{center}
\par\medskip\noindent
The verifier's binary64 comparison target is
$0.058213000000000001$, which exceeds the exact rational
$58213/1000000$ by less than $9\times10^{-19}$.  The smallest reported
pruning margin is greater than $8.42\times10^{-9}$.  We allow a further $10^{-12}$ for assembling each final cell bound in binary64. The endpoint derivative formulas are bounded in absolute value
by two, the cell widths are at most one, and the assembly uses fewer than
50 arithmetic operations with intermediate absolute values below 16.
With unit roundoff $2^{-53}$, the total absolute error, including the final
subtraction, is below $10^{-12}$. The recorded pruning margin exceeds this
allowance together with the error in the comparison target. Every cell is
therefore below the exact rational target $58213/1000000=0.058213$, proving
\eqref{hard:eq:L-target}.

Here \texttt{nextafter(x,+infinity)} returns the smallest binary64 number
strictly greater than \texttt{x}. The acceptance threshold is
\texttt{nextafter(2e-12,+infinity)}. A terminal cell is accepted only if its
computed margin is finite and exceeds this threshold. This reserve is larger
than $10^{-12}+9\times10^{-19}$ and thus covers both final assembly and target
rounding. The executable checks the minimum margin once more before reporting
success. Since the terminal cells cover the entire square, these checks
establish the uniform inequality.

\subsection{Public source and reproduction}\label{hard:app:github}

The original source and certificate output supporting
Proposition~\ref{hard:prop:gaussian-certificate} are available from
\citet{DuGaussianCertificate2026} in the GitHub repository
\begin{center}
\url{https://github.com/Duke1968/dicut-gaussian-certificate}.
\end{center}
The source snapshot is identified by commit
\begin{center}
\href{https://github.com/Duke1968/dicut-gaussian-certificate/tree/cbca544a2f248558f6303590214485f15562b7d4}{\nolinkurl{cbca544a2f248558f6303590214485f15562b7d4}}.
\end{center}
The same certificate files are included at \path{anc/explicit_input_hardness/}:
the C verifier, local MPFR header, exact-parameter script, verification output,
and compilation and environment records. The SHA-256 digest of the
\path{SHA256SUMS} manifest is
\begin{center}
\nolinkurl{16a92643ba25ccea23696027db0ab9289ce20648feb56a2fbf5a78c0a68ce561}.
\end{center}
The manifest identifies the other nine files in that directory.

The C verifier checks the Gaussian inequality using MPFR evaluation,
the quadrature remainder, outward error allowances, and adaptive
subdivision. In the C source and its output, \texttt{L} denotes $H_G$.
The Python script checks the rational parameters and derived ratios in
Appendix~\ref{hard:app:parameters}. Theorem~\ref{hard:thm:affine-criterion} supplies the separate reduction from the
Gaussian inequality to computational hardness.
The directory's README provides the compilation and verification commands,
together with the required arithmetic assumptions and software environment.

\paragraph{Reserve-aware source-package version.}
The directory \path{anc/gaussian_reserve_v1/} contains the C source that enforces the reserve, its local header, a patch against the cited public source, and a separate manifest. The cited commit identifies the original verifier; the reserve-enforcing variant is included in this source package. Its output, compiler and environment information, and measured resources are recorded in \path{anc/presentation_revision_2026_09_22/validation/}. The README files explain how the two versions correspond and how to reproduce the checks.

\section{Exact parameter identities}\label{hard:app:parameters}
The parameters in Theorem~\ref{hard:thm:ugc-hardness} give the following exact probabilities, correlations, density, and ratios.

For a binary pair with marginals $q_s,q_t$ and
$p_{st}=\nu_{st}(1,1)$, the four atoms are
\[
 p_{st},\quad q_s-p_{st},\quad q_t-p_{st},\quad
 1-q_s-q_t+p_{st}.
\]
Substituting \eqref{hard:eq:pi-q}--\eqref{hard:eq:pairs} shows that all four probabilities are nonnegative. Their Pearson correlation is
\[
 \rho_{st}=
 \frac{p_{st}-q_sq_t}
 {\sqrt{q_s(1-q_s)q_t(1-q_t)}},
\]
which gives \eqref{hard:eq:correlations}. The density, completeness, and certified soundness satisfy
\begin{align*}
 \mu&=\pi_Aq_A+\pi_Bq_B
 =\frac{186600729}{2000000000},\\
 C&=W_{AA}(q_A-p_{AA})+W_{BA}q_B+W_{BB}q_B
 =\frac{17893931095731}{125000000000000},\\
 \beta+\lambda\mu
 &=\frac{23718092975777}{200000000000000},\\
 \frac{\beta+\lambda\mu}{C}
 &=\frac{118590464878885}{143151448765848}.
\end{align*}

\section{Quadratic linear and semidefinite relaxations}\label{app:quadratic_comparison}
For the quadratic splitting in \eqref{eq:diagonal_split}, put $a_i=\ell_i-d_i$ and write $M(x)=c+\sum_i a_ix_i-\sum_{i<j}b_{ij}x_ix_j$. Proposition~\ref{prop:root_independence} identifies the standard roof of this degree-two negaform as
\begin{equation}\label{eq:quadratic_mccormick}
 L_M(x)=c+\sum_i a_ix_i-\sum_{i<j}b_{ij}\max\{0,x_i+x_j-1\}.
\end{equation}
At fixed $x$, this is the maximum of $c+\sum_i a_ix_i-\sum_{i<j}b_{ij}X_{ij}$ over the pairwise McCormick bounds \citep{mccormick1976computability}
\[
 \max\{0,x_i+x_j-1\}\le X_{ij}\le\min\{x_i,x_j\}\qquad(i<j).
\]
Because $-b_{ij}$ is nonpositive, each product variable can independently take its lower bound. Let $L_{\rm SDP}(x)$ be the value obtained by retaining these bounds and adding the Boolean moment constraints
\[
 \begin{pmatrix}1&x^\top\\x&X\end{pmatrix}\succeq0,
 \qquad X=X^\top,\qquad X_{ii}=x_i\quad(i\in N).
\]
Here $\succeq0$ denotes positive semidefiniteness. If $m$ is the Boolean restriction of $M$, then
\begin{equation}\label{eq:quadratic_sdp_order}
 M(x)\le m^+(x)\le L_{\rm SDP}(x)\le L_M(x).
\end{equation}
Any joint Bernoulli law with marginals $x$ gives feasible first and second moments; independent sampling has expected value $M(x)$. Adding constraints to the same linear relaxation gives the final inequality. Mixing feasible moment matrices proves concavity of $L_{\rm SDP}$. This comparison retains the McCormick inequalities and uses the pairwise roof \eqref{eq:quadratic_mccormick}. For the continuous quadratic $F$, keep the exact diagonal residual in \eqref{eq:diagonal_split} separately: the Boolean moment identity $X_{ii}=x_i$ does not impose the continuous identity $x_i^2=x_i$.

A recovery inequality paying $\rho L_M$ also pays $\rho L_{\rm SDP}$. Reoptimizing the tighter bound need not improve the factor or the returned objective value; Section~\ref{sec:design} discusses this issue for stronger roofs. Spectral assumptions also require care: an entrywise nonpositive Hessian, as required by DR-submodularity, need not be negative semidefinite. Finally, a constant shift changes multiplicative guarantees. For $C\ge0$, a guarantee for $F+C$ yields only
\[
 F(y)+C\ge\alpha\bigl(\OPT(F,\K)+C\bigr)
 \quad\Longrightarrow\quad
 F(y)\ge\alpha\OPT(F,\K)-(1-\alpha)C.
\]
Thus a diagonal shift can make the quadratic matrix positive semidefinite and be constant on the chosen Boolean encoding, yet still fail to preserve a multiplicative guarantee for the original objective.

\section{Outcome-space benchmark and prophet transfer}\label{app:outcomes}
The prophet consequence of Theorem~\ref{thm:online_preservation} uses a valuation on the entire outcome universe. Its fractional benchmark must respect both outcome probabilities and item feasibility; we establish that benchmark and the transfer here.

\subsection{Finite outcomes and the fractional benchmark}
There are independent items $i\in[m]$, with disjoint finite outcome sets $U_i$. Outcome $e\in U_i$ occurs with probability $p_e$, where $\sum_{e\in U_i}p_e\le1$; residual probability produces no element. The valuation $f:2^U\to\R_{\ge0}$ is nonnegative and submodular on the entire outcome universe $U=\bigsqcup_iU_i$. Accepted item indices must belong to $\I$. Let $\OPT_{\rm prop}$ be the expected value of an optimal feasible selection after all outcomes are observed.

Define the outcome-space region from the item-independence polytope by
\begin{equation}\label{eq:outcome_polytope}
 P^{\rm out}=\left\{x\in\R_{\ge0}^{U}: x_e\le p_e\ (e\in U),\quad
          \left(\sum_{e\in U_i}x_e\right)_{i=1}^m\in P_{\I}\right\}.
\end{equation}
The region is compact, down-closed, and contains zero, also when null outcomes have positive probability. It contains the marginals of any outcome set selected by the prophet. Hence
\begin{equation}\label{eq:prophet_roof_bound}
 \OPT_{\rm prop}\le\max_{x\in P^{\rm out}}f^+(x)\le\max_{x\in P^{\rm out}}L(x)
\end{equation}
whenever a local roof is available. For the first inequality, use the distribution of the prophet's selected set as a feasible coupling in the concave closure.

\begin{proposition}[Outcome-space transfer]\label{prop:outcome_transfer}
Assume the valuation on all of $U$ has representation \eqref{eq:online_class}. For $z\in bP^{\rm out}$ and an item-space $(b,c)$-selectable greedy OCRS, activate realized outcome $e$ independently with probability $z_e/p_e$, interpreted as zero when $p_e=0$, and run the OCRS on the corresponding item. The resulting feasible outcome set satisfies
\[
 \E f(S_{\rm alg})\ge cF(z).
\]
\end{proposition}
\begin{proof}
Put $\bar z_i=\sum_{e\in U_i}z_e$. Item activations are independent with marginals $\bar z\in bP_{\I}$. An equivalent procedure independently chooses a candidate outcome for each item with $\bar z_i>0$, using probabilities $z_e/\bar z_i$, and then activates the candidate with probability $\bar z_i$. Conditional on these candidates, apply Theorem~\ref{thm:online_preservation}. Restricting to the candidate set preserves \eqref{eq:online_class}: a cut becomes a cut, a coverage indicator, or zero, and unaries may become constants. Because the OCRS depends only on $\bar z$ and is independent of the candidate identities, its selectability remains $c$ for every fixed candidate set.

Let $A_z$ be the set of activated candidates, which contains at most one
outcome per item. Averaging the conditional OCRS inequalities gives
$\E f(S_{\rm alg})\ge c\E f(A_z)$. It remains to show $\E f(A_z)\ge F(z)$.
Conditional on the other items, the submodular function $h$ on one item's outcomes satisfies
\[
 h(S)\le h(\varnothing)+\sum_{e\in S}\bigl(h(\{e\})-h(\varnothing)\bigr).
\]
Taking an independent sample with marginals $(z_e)_{e\in U_i}$ bounds its expectation by
$(1-\bar z_i)h(\varnothing)+\sum_e z_eh(\{e\})$,
which is the expectation for that item under $A_z$. Replace the item blocks one
at a time to prove this comparison. This is the
submodular-dominance mechanism relevant here \citep{QiuSingla2022}.
\end{proof}

\begin{corollary}[Modular fractional-to-prophet guarantee]\label{cor:prophet_interface}
Under Proposition~\ref{prop:outcome_transfer}, assume the item-space OCRS is efficient and let $\alpha,\varepsilon\ge0$. Any efficiently computed $z\in bP^{\rm out}$ satisfying $F(z)\ge\alpha\OPT_{\rm prop}-\varepsilon$ gives an online guarantee
\[
 \E f(S_{\rm alg})\ge c\alpha\OPT_{\rm prop}-c\varepsilon.
\]
In particular, suppose the full outcome valuation also has an efficiently optimizable roof and a downward recovery inequality of factor $\rho$. Then a $\delta$-accurate roof optimizer gives
\begin{equation}\label{eq:roof_prophet}
 \E f(S_{\rm alg})\ge bc\rho(\OPT_{\rm prop}-\delta).
\end{equation}
\end{corollary}
\begin{proof}
The first assertion is Proposition~\ref{prop:outcome_transfer}. For the second, optimize $L$ over \eqref{eq:outcome_polytope}, recover $y\le x$ with $F(y)\ge\rho L(x)$, and set $z=by$. Then $z\in bP^{\rm out}$. Since $y\ge0$ and every entry of $\nabla^2F$ is nonpositive, $d^2F(ty)/dt^2=y^\top\nabla^2F(ty)y\le0$. Concavity and $F(0)\ge0$ give $F(by)\ge bF(y)$. Apply \eqref{eq:prophet_roof_bound}.
\end{proof}
For bit-polynomial implementations, the finite outcome lists and known rational probabilities are explicitly supplied with polynomially bounded encoding length.
On $U$, quadratic, cubic, and supplied bounded-degree rooted valuations have factors $\rho$ of $1/2$, $8/17$, and $8/17$ through degree twelve, respectively. Arbitrary finite rooted representations have factor $\gamma_*$. For arbitrary arity-four factors, this online transfer additionally requires the global representation \eqref{eq:online_class}.

\section{Details of the finite-type hardness reduction}\label{app:hardness_details}
We prove the invariance, decoding, and replication steps behind the affine criterion, using the definitions and hypotheses of Section~\ref{hard:sec:methodology}.

\subsection{Finite-type invariance}\label{app:invariance_proof}
\begin{proof}[Proof of Lemma~\ref{hard:lem:finite-type-invariance}]
For a used pair $(s,t)$, put $Y':=1-Y$ and
$h_t(y'):=1-g_t(1-y')$.  Apply Mossel's two-function theorem to $g_s$ and
$h_t$ on the correlated space $(X,Y')$.  By
Lemma~\ref{hard:lem:binary-correlation}, its maximal correlation is $-\rho_{st}$, so
the Gaussian upper bound is
\[
 \Gamma_{-\rho_{st}}(m_s,1-m_t)
 =D_{\rho_{st}}(m_s,m_t).
\]
The two-function refinement applies whenever no coordinate has large
low-degree influence in both functions \citep[Theorem~6.3 and Lemma~6.8, arXiv:math/0703683v6]{Mossel2010}; input and output
complementation do not change influence.  Choose the same $d,\tau$ for the finitely many used pairs and apply each
pairwise bound with error $\eta$. Multiplying by
$W_{st}$ and using $\sum_{s,t}W_{st}=1$ proves
\eqref{hard:eq:finite-type-invariance}.
\end{proof}

\subsection{Influence decoding}\label{app:decoding_proof}
\begin{proof}[Proof of Lemma~\ref{hard:lem:influence-decoding}]
Apply Lemma~\ref{hard:lem:finite-type-invariance} with error $\eta/2$, obtaining
$d$ and $\tau$.  Let $B\subseteq U$ be the vertices where the family $(f_{u,s})_s$
satisfies the influence alternative.  Outside $B$, the
local test value is at most $\Psi(m_u)+\eta/2$; on $B$, its excess over
$\Psi(m_u)$ is at most one.  Thus \eqref{hard:eq:decoding-implication} implies
$|B|/|U|>\eta/2$.

For every $u\in B$, fix a witness $(s_u,t_u,i_u)$.  The first endpoint function
has influence at least $\tau$, so label $u$ by $i_u$.
Convexity of squared biased Fourier coefficients gives
\[
 \Inf_{i_u}^{\le d}(f_{u,s_u})
 \le \E_{v\in N(u)}
 \Inf_{\sigma_{uv}(i_u)}^{\le d}(g_{v,s_u}).
\]
Consequently, for at least a $\tau/2$ fraction of the incident edges,
$\Inf_{\sigma_{uv}(i_u)}^{\le d}(g_{v,s_u})\ge\tau/2$.  For $v\in V$, set
\[
 I_v:=\bigcup_{s\in\mathcal T}
 \{j:\Inf_j^{\le d}(g_{v,s})\ge\tau/2\}.
\]
The total low-degree influence bound gives
$|I_v|\le M:=\lceil2|\mathcal T|d/\tau\rceil$.  Label each $v$ uniformly
from $I_v$, using an arbitrary label when $I_v$ is empty.  For each qualifying edge $(u,v)$, the label $\sigma_{uv}(i_u)$ belongs to $I_v$
and is chosen with probability at least $1/M$. Under our permutation
convention, this choice satisfies the edge. Regularity and the fraction of
qualifying incident edges give an expected satisfied-edge fraction of at least
\[
 \frac{|B|}{|U|}\frac{\tau}{2M}
 >\frac{\eta\tau}{4M}=:\kappa.
\]
Some deterministic labeling attains this value.
\end{proof}

\subsection{Affine finite-type reduction}\label{app:affine_reduction_proof}
\begin{proof}[Proof of Theorem~\ref{hard:thm:affine-criterion}]
Use the regular long-code construction with a global cardinality constraint
\citep{RaghavendraTan2012,AustrinStankovic2019}.  For every $v\in V$ and
$s\in\mathcal T$, create a type-$s$ cloud indexed by $x\in\{0,1\}^R$ with
vertex mass
\[
 \frac1{|V|}\pi_s\Prb_{q_s}[x].
\]
Sample an arc by choosing $u\in U$ uniformly and ordered neighbors
$v,w\in N(u)$ independently. Choose $(s,t)$ with probability $W_{st}$,
draw $(X,Y)\sim\nu_{st}^{\otimes R}$, and apply the edge permutations.
Place the arc from the type-$s$ cloud of $v$ to the type-$t$ cloud of $w$.  For a cut with cloud indicators $g_{v,s}$, its value
is exactly the left-hand side of \eqref{hard:eq:decoding-implication}, with
$f_{u,s}$ as in \eqref{hard:eq:neighbor-average}.

Suppose a labeling satisfies at least $1-\zeta$ of the Unique-Games edges.
Select in each cloud the dictator coordinate given by its right-vertex label.
The selected vertex mass is exactly $\mu$.  Both sampled constraints are
satisfied with probability at least $1-2\zeta$; conditional on this event the
two pulled-back dictators use the same coordinate and a type-$(s,t)$ arc is
cut with probability $\nu_{st}(1,0)$.  Nonnegativity of arc weights gives
completeness at least $(1-2\zeta)C$.

For soundness, fix $\eta>0$ and take $\kappa$ from
Lemma~\ref{hard:lem:influence-decoding}. Invoke the Unique Games Conjecture with no-case value at most $\kappa/2$, so that the decoding contrapositive applies.
Equation~\eqref{hard:eq:decoding-implication} then bounds every cut by
$\E_u[\Psi(m_u)]+\eta$. Regularity expresses its selected vertex mass as
$\E_u[\pi^\top m_u]$.  A budget density at most $\mu+r$ and
\eqref{hard:eq:affine-majorant} therefore give
\[
 \text{cut value}
 \le\beta+\lambda(\mu+r)+\eta.
\]

Because all cloud probabilities are fixed rationals, replace a cloud vertex
of rational mass by the corresponding number of equal-mass copies.  Replace
each original arc by the complete bipartite bundle between its two copy sets,
distributing the arc weight by the product of the endpoint copy proportions.
For an arbitrary cut of the replicated instance, define $g_{v,s}(x)\in[0,1]$
as the selected fraction of the copies of $(v,s,x)$.  By bilinearity, its cut value and selected-vertex fraction are given by the
same expressions with these fractional indicators. Soundness therefore holds
even for cuts that separate copies. For a fixed gap tolerance, the denominators
of the type masses and product probabilities, the alphabet size $R$, and the
additional replication factor are constants. Both the number $n'$ of copies
and the explicit rational arc list are consequently polynomial in the outer
instance size. Multiplying all copy counts by a further integer $M$ makes
every copy set have at least $M$ vertices. In a bundle from a cloud vertex to itself, self-loops carry at most a $1/M$
fraction of the bundle weight. They contribute zero to every cut, so delete
them and normalize the remaining arc weight to one.
If their total weight is $\ell\le1/M$, the soundness bound is divided by
$1-\ell$, while completeness cannot decrease. Taking
$k=\lceil\mu n'\rceil$ gives $k/n'=\mu+o(1)$.
Choose $\zeta$, $\eta$, $1/M$, and the resulting density error $r$
sufficiently small relative to $\delta$.  The Unique Games Conjecture then yields the gap
\eqref{hard:eq:ugc-gap}.  If an approximation ratio exceeded $(\beta+\lambda\mu)/C$ by a fixed
amount, choosing $\delta$ sufficiently small would make the algorithm's
output values distinguish the two cases.
\end{proof}

\Needspace{17\baselineskip}
\section*{Statements and declarations}
\addcontentsline{toc}{section}{Statements and declarations}
\paragraph{Funding.} This work was supported by the Natural Sciences and Engineering Research Council of Canada, grant RGPIN/3840-2024.
\paragraph{Competing interests.} The author declares no relevant financial or nonfinancial competing interests.
\paragraph{Code and proof certificates.}
The three GitHub repositories listed
in Appendices~\ref{app:twelve}, \ref{app:arity_four}, and
\ref{hard:app:github} contain, respectively, the Lean formalization of
Theorem~\ref{thm:twelve_atoms}, the exact-arithmetic certificates for
Theorem~\ref{thm:arity_four}, and the MPFR-based Gaussian certificate for
Proposition~\ref{hard:prop:gaussian-certificate}. These appendices specify
the source versions, certificate identifiers, and reproduction commands.
The article's ancillary material also includes the arity-four and Gaussian
certificate files. The reserve-enforcing Gaussian verifier and its output
are supplied alongside the separately identified public snapshot;
Appendix~\ref{hard:app:certificate} explains how they correspond.
All remaining proofs appear in the manuscript.
\paragraph{Generative AI assistance.}
ChatGPT assisted with exposition and LaTeX preparation,
literature-source auditing, and supporting symbolic,
exact-arithmetic, and proof-software checks. The author is responsible for providing the preliminary research notes that initiated the project, guiding multiple rounds of collaborative feedback. Additionally, the author maintains full accountability for the mathematical content and its alignment with the proof artifacts.

\setlength{\bibsep}{6pt plus 1pt minus 1pt}
\bibliographystyle{plainnat}
\bibliography{references}
\end{document}